\documentclass[leqno, 10pt, a4paper]{amsart}
\usepackage{amsmath}
\usepackage{amssymb, latexsym, mathrsfs,  mathabx, dsfont, a4wide}
\usepackage{color}

 \newtheorem{theorem}{Theorem}[section]
 \newtheorem{lemma}[theorem]{Lemma}
 \newtheorem{proposition}[theorem]{Proposition}
 \newtheorem{corollary}[theorem]{Corollary}

 \newtheorem*{theorem*}{Theorem}
\newtheorem*{proposition*}{Proposition}
\newtheorem*{lemma*}{Lemma}

\theoremstyle{definition}
 \newtheorem{definition}[theorem]{Definition}

 \theoremstyle{remark}
 \newtheorem{example}[theorem]{Example}
 \newtheorem{remark}[theorem]{Remark}

   \newtheorem*{claim*}{Claim}

\newcommand{\op}[1]{\operatorname{#1}}

\newcommand{\cf}{\emph{cf}.}

\newcommand{\norm}[1]{\ensuremath{\left\|#1\right\|}}

\newcommand{\NormN}[1]{\left\vvvert #1\right\vvvert_N}

\newcommand{\acou}[2]{\ensuremath{\left\langle #1 , #2 \right\rangle}}

\newcommand{\brak}[1]{\ensuremath{\langle #1\rangle}}
\newcommand{\acoup}[2]{\ensuremath{\left(#1|#2\right)}}

\newcommand{\Tr}{\ensuremath{\op{Tr}}}

\def\XXint#1#2#3{{\setbox0=\hbox{$#1{#2#3}{\int}$}
\vcenter{\hbox{$#2#3$}}\kern-.5\wd0}}

\newcommand{\C}{\ensuremath{\mathbb{C}}} 
 
\newcommand{\N}{\ensuremath{\mathbb{N}}} 
\newcommand{\Q}{\ensuremath{\mathbb{Q}}} 
\newcommand{\R}{\ensuremath{\mathbb{R}}} 
\newcommand{\bS}{\ensuremath{\mathbb{S}}} 
\newcommand{\T}{\ensuremath{\mathbb{T}}} 
\newcommand{\Z}{\ensuremath{\mathbb{Z}}}

\newcommand{\Rn}{\ensuremath{\R^{n}}}

\newcommand{\fS}{\ensuremath{\mathfrak{S}}}

\newcommand{\Ca}[1]{\ensuremath{\mathscr{#1}}}
\newcommand{\cA}{\mathscr{A}}
\newcommand{\cB}{\Ca{B}}

\newcommand{\cH}{\ensuremath{\mathscr{H}}}

\newcommand{\cL}{\ensuremath{\mathscr{L}}}

\newcommand{\cR}{\Ca{R}}
\newcommand{\cS}{\ensuremath{\mathscr{S}}}

\newcommand{\psido}{$\Psi$DO} 
\newcommand{\psidos}{$\Psi$DOs}

\newcommand{\stS}{\mathbb{S}}

\def\dba{{\mathchar'26\mkern-12mu d}}
\newcommand{\dbar}{{\, \dba}}

\newcommand{\Sp}{\op{Sp}}

\numberwithin{equation}{section}

\begin{document}

\title{Pseudodifferential Calculus on Noncommutative Tori, I. Oscillating Integrals}

\author{Hyunsu Ha}
 \address{Department of Mathematical Sciences, Seoul National University, Seoul, South Korea}
 \email{mamaps@snu.ac.kr}

 \author{Gihyun Lee}
 \address{Department of Mathematical Sciences, Seoul National University, Seoul, South Korea}
 \email{gihyun.math@gmail.com}

\author{Rapha\"el Ponge}
 \address{School of Mathematics, Sichuan University, Chengdu, China}
 \email{ponge.math@icloud.com}

\thanks{The research for this article was partially supported by 
  NRF grants 2013R1A1A2008802 and 2016R1D1A1B01015971 (South Korea).}
 
 
 
\begin{abstract}
This paper is the first part of a two-paper series whose aim is to give a thorough account on Connes' pseudodifferential calculus on noncommutative tori. This pseudodifferential calculus has been used in numerous recent papers, but a detailed description is still missing. In this paper, we focus on constructing an oscillating integral for noncommutative tori and laying down the main functional analysis ground for understanding Connes' pseudodifferential calculus. In particular, this allows us to give a precise explanation of the definition of pseudodifferential operators on noncommutative tori. More generally, this paper introduces the main technical tools that are used in the 2nd part of the series to derive the main properties of these operators. 
\end{abstract}

\maketitle 

\section{Introduction}
Noncommutative tori are important examples of noncommutative spaces which occur in numerous parts of mathematics and mathematical physics.  For instance, searching for ``noncommutative torus" or ``noncommutative tori" on Google Scholar produces nearly 2,700 hits.  In particular, Connes~\cite{Co:CRAS80, Co:IHES85, Co:NCG} obtained a reformulation of the Atiyah-Singer index formula for noncommutative tori. This was an important impetus for the noncommutative geometry approaches to the quantum Hall effect~\cite{BES:JMP94} and topological insulators~\cite{BCR:RMP16, PS:Springer16}. In addition, noncommutative tori have been considered in the context of string theory (see, e.g., \cite{CDS:JHEP98, SW:JHEP99}).

The noncommutative 2-torus arises from the action of $\Z$ on the circle $\bS^1$ generated by a rotation of angle $2\pi \theta$ with $\theta \in \R\setminus \Q$ (see~\cite{Co:NCG, Ri:PJM81}). 
For such an action the orbits are dense, and so the orbit space is not even Hausdorff. Following the \emph{motto} of noncommutative geometry~\cite{Co:NCG} we trade the orbit space for the crossed-product algebra $C^\infty(\bS^1)\rtimes \Z$ associated with this action. We then get an algebra generated by two unitaries $U$ and $V$ with the relations,  
\begin{equation*}
 UV = e^{2i\pi \theta} VU. 
\end{equation*}
More generally, we can consider noncommutative $n$-tori associated with any anti-symmetric $n\times n$-matrix $\theta=(\theta_{jl})\in M_n(\R)$ with $n\geq 2$. The corresponding $C^*$-algebra $A_\theta$ is generated by unitaries $U_1,\ldots, U_n$ satisfying the relations, 
\begin{equation*}
 U_lU_j=e^{2i\pi \theta_{jl}} U_jU_l, \qquad j,l=1,\ldots, n. 
\end{equation*}
A dense spanning linearly independent set of $A_\theta$ is provided by the unitaries, 
\begin{equation*}
 U^k=U_1^{k_1}\cdots U_n^{k_n}, \qquad k=(k_1,\ldots, k_n)\in \Z^n.
\end{equation*}
We may think of series $\sum u_k U^k$, $u_k\in \C$,  as analogues of the Fourier series on the ordinary $n$-torus 
$\T^n=\R^n\slash 2\pi \Z^n$. The analogue of the integral is provided by the normalized state $\tau:A_\theta \rightarrow \C$ such that $\tau(1)=1$ and $\tau(U^k)=0$ for $k\neq 0$. The associated GNS representation provides us with a unital $*$-representation of $A_\theta$ into a Hilbert space $\cH_\theta$, which is the analogue of the Hilbert space $L^2(\T^n)$ (see Section~\ref{section:NCtori}). In particular, for $\theta=0$ we recover the representation of continuous functions on $\T^n$ by multiplication operators on $L^2(\T^n)$. 

There is a natural periodic $C^*$-action $(s,u)\rightarrow \alpha_s(u)$ of $\R^n$ on $A_\theta$, so that we obtain a  $C^*$-dynamical system. We are interested in the dense subalgebra $\cA_\theta$ of smooth elements of this action. These are elements of the form $u=\sum_{k\in \Z^n} u_k U^k$, where the sequence $(u_k)_{k\in \Z^n}$ has rapid decay (see Section~\ref{section:NCtori}). When $\theta=0$ we recover the algebra of smooth functions on $\T^n$.  In general, $\cA_\theta$ is a Fr\'echet $*$-algebra which is stable under holomorphic functional calculus.  In addition, the action of $\R^n$ induces a smooth action on $\cA_\theta$ which is infinitesimally generated  by 
the derivations $\delta_1, \ldots, \delta_n$ such that $\delta_j(U_k)=\delta_{jk}U_k$, $j,k=1,\ldots, n$. 
They are the analogues of the partial derivatives $\frac{1}{i} \partial_{x_1},\ldots, \frac{1}{i} \partial_{x_n}$ on  $\T^n$. Differential operators on $\cA_\theta$ are then defined as operators of the form $\sum_{|\alpha|\leq N} a_\alpha \delta^\alpha$, $a_\alpha \in \cA_\theta$, where $\delta^\alpha=\delta_1^{\alpha_1} \cdots \delta_n^{\alpha_n}$ (see~\cite{Co:CRAS80, Co:NCG}). For instance, the Laplacian $\Delta=\delta_1^2+\cdots +\delta_n^2$ is such an operator. 

In his famous note~\cite{Co:CRAS80} Connes introduced a pseudodifferential calculus for $C^*$-dynamical systems. In particular, this provides us with a natural pseudodifferential calculus on noncommutative tori. 
Further details of the calculus were announced in Baaj's notes~\cite{Ba:CRAS88} (by using a slightly different point of view). 
About two decades later Connes-Tretkoff~\cite{CT:Baltimore11} used this pseudodifferential calculus to prove a version of the Gauss-Bonnet theorem for noncommutative 2-tori (see also~\cite{FK:JNCG12}). 
The paper of Connes-Tretkoff  sparked a vast surge of research activity on applying pseudodifferential techniques to the differential geometry study of noncommutative tori. The main directions of research include reformulations of the Gauss-Bonnet and Hirzebruch-Riemann-Roch theorems for noncommutative tori and similar noncommutative manifolds~\cite{CT:Baltimore11, DS:JMP13, DS:SIGMA15, FG:SIGMA16, FK:JNCG12, KM:JGP14, KS:arXiv17}, constructions of scalar and Ricci curvatures for conformal deformations of noncommutative tori~\cite{CF:arXiv16, DGK:arXiv18, Fa:JMP15, FK:JNCG13, FK:JNCG15, FGK:arXiv16, Liu:JGP17, Liu:ArXiv18a, Liu:ArXiv18b}, and construction and study of noncommutative residue, zeta functions and log-determinants of elliptic operators~\cite{CM:JAMS14, FGK:MPAG17, FK:LMP13, FW:JPDOA11, LM:GAFA16, LNP:TAMS16, Si:JPDOA14}. There is also a construction of a Ricci flow for noncommutative 2-tori~\cite{BM:LMP12}. 

In view of the recent amount of papers using the pseudodifferential calculus on noncommutative tori, it would seem sensible to have a detailed account on this calculus.  This paper is part of a series of two papers whose aim is precisely to produce such an account.  This includes a precise definition of pseudodifferential operators (\psidos) and full proofs of all the properties of \psidos\ surveyed in~\cite{CT:Baltimore11}, which have been used in the subsequent papers mentioned above. In addition, we include further results regarding action of \psidos\ on Sobolev spaces, regularity properties of elliptic operators, spectral theory of elliptic operators and Schatten-class properties of \psidos. 
In particular, this shows that \psidos\ on noncommutative tori satisfy the same properties as \psidos\ on ordinary (closed) manifolds. 

The present paper focuses on constructing an oscillating integral for $\cA_\theta$-valued maps and laying down the functional analysis ground for the study of \psidos\ on nocommutative tori. Incidentally, we give a precise definition of \psidos\ and derive a few immediate properties of these operators. In the sequel~\cite{HLP:Part2} (referred throughout the paper as Part~II) we shall deal with the main properties of \psidos\ on noncommutative tori (including the properties alluded to above). 

Given an $n$-dimensional noncommutative torus $\cA_\theta$, Connes~\cite{Co:CRAS80} defined a \psido\ on $\cA_\theta$ as a linear operator $P:\cA_\theta \rightarrow \cA_\theta$ of the form, 
\begin{equation}
 Pu = (2\pi)^{-n} \iint e^{is\cdot \xi} \rho(\xi) \alpha_{-s}(u) ds d\xi, \qquad u\in \cA_\theta, 
 \label{eq:Intro.integral-Pu}
\end{equation}
where $\rho:\R^n \rightarrow \cA_\theta$ is an $\cA_\theta$-valued symbol.  This means that there is $m\in \R$ such that every partial derivative 
$\delta^\alpha \partial_\xi^\beta \rho(\xi)$ is $\op{O}(|\xi|^{m-|\beta|})$ at infinity (see Section~\ref{section:Symbols} for the precise definition). As the action $s\rightarrow \alpha_{-s}(u)$ is periodic, the integral on the r.h.s.~does not make sense as a usual integral. In fact, this situation is even different from the situation in the usual definition of \psidos\ on an open set of $\R^n$, where $\alpha_{-s}(u)$ is replaced by a Schwartz-class function (see, e.g., \cite{Sh:Springer01}).  An important part of this paper is devoted to give a precise meaning for this integral. 
At the exception of the original references~\cite{Ba:CRAS88, Co:CRAS80}, and of~\cite{LM:GAFA16, Ta:JPCS18}, most papers on pseudodifferential operators on noncommutative tori do not really address the precise meaning of this integral. Therefore, it seems sensible to take some care to deal with this question.   

An important tool in the study of \psidos\ on ordinary manifolds is the oscillating integral (see, e.g., \cite{AG:AMS07, Ho:Springer85}). The existence of an $\cA_\theta$-valued oscillating integral is pointed out in~\cite{Ba:CRAS88, Co:CRAS80} (see also~\cite{LM:GAFA16}). 
The oscillating integrals that we consider are of the form, 
\begin{equation}
 I(a)=  (2\pi)^{-n}\iint e^{is\cdot \xi} a(s,\xi) ds d \xi,
 \label{eq:Intro.oscillating-int}
\end{equation}
where $a:\R^n\times \R^n \rightarrow \cA_\theta$ is an $\cA_\theta$-valued amplitude. This means this is a smooth map and there is some $m\in \R$ such that all the derivatives $\delta^\alpha \partial_s^\beta \partial_\xi^\gamma a(s,\xi)$ are $\op{O}((|s|+|\xi|)^m)$ at infinity (see Section~\ref{section:Amplitudes} for the precise definition). 
The classes of amplitudes that we consider are natural generalizations of the classes of scalar-valued amplitudes considered in~\cite{AG:AMS07, Ho:Springer85}.  If $\rho(\xi)$ is a symbol, then we observe that $\rho(\xi) \alpha_{-s}(u)$ is an amplitude (see Lemma~\ref{lem:PsiDOs.amplitudes-alpha-product-continuity}), and so the integral in~(\ref{eq:Intro.integral-Pu}) makes sense as an oscillating integral. More generally if $a(s,\xi)$ is an amplitude, then $a(s,\xi)\alpha_{-s}(u)$ is an amplitude, and so we even can define \psidos\ associated with amplitudes (see Section~\ref{sec:PsiDOs}). We thus obtain a larger class of \psidos\ than the class originally considered in~\cite{Ba:CRAS88, Co:CRAS80}. The interest of using this larger class of \psidos\ will become more apparent in Part~II. 

The integral in~(\ref{eq:Intro.oscillating-int}) makes sense when $a(s,\xi)$ decays fast enough to be integrable. In particular, it makes sense as a Riemann integral when $a(s,\xi)$  has compact support. The spaces of amplitudes (and the spaces of symbols as well) carry natural Fr\'echet-space topologies (see Section~\ref{section:Amplitudes}).  
 The compactly supported amplitudes are  dense among the other amplitudes. By means of similar integration by parts arguments as with scalar-valued amplitudes it can be shown that the linear map $a\rightarrow I(a)$ has a unique continuous extension to the space of all amplitudes (see Proposition~\ref{prop:Amplitudes.extension-J0} for the precise statement). As a consequence of this continuity property we see that the \psidos\ are continuous linear operators and depend continuously on their symbols or amplitudes (see Proposition~\ref{prop:PsiDOs.pdos-continuity} and Proposition~\ref{prop:PsiDOs.symbols-to-pdos-continuity}).  

One salient feature of the noncommutative torus is the existence of an orthonormal basis provided by the unitaries, 
\begin{equation*}
   U^k=U_1^{k_1} \cdots U_n^{k_n}, \qquad k\in \Z^n. 
\end{equation*}
The corresponding decomposition along this basis is the analogue of the Fourier series decomposition for functions on the ordinary torus $\T^n$. We can take advantage of this decomposition. Given any \psido\ $P$ associated with a symbol $\rho(\xi)$, we have 
\begin{equation}
 Pu = \sum_{k\in \Z^n} u_k\rho(k) U^k \qquad \text{for any ${\displaystyle u=\sum_{k\in \Z^n} u_k U^k}$ in $\cA_\theta$}. 
 \label{eq:Intro.toroidal}
\end{equation}
 The above formula was mentioned in~\cite{CT:Baltimore11}. A proof is given in Section~\ref{sec:PsiDOs}.  This formula allows us to show that differential operators are \psidos\ on noncommutative tori (see Section~\ref{sec:PsiDOs}). 
 
The formula~(\ref{eq:Intro.toroidal}) further allows us to relate our class of \psidos\ associated with standard symbols to the \psidos\ associated with toroidal symbols as considered in~\cite{GJP:MAMS17, LNP:TAMS16}. The toroidal symbols are discrete versions of the standard symbols considered in this paper (see Section~\ref{sec:toroidal} for their precise definition). As it turns out, the formula~(\ref{eq:Intro.toroidal}) continues to define an operator on $\cA_\theta$ if we replace the sequence $(\rho(k))_{k\in\Z^n}$ by a sequence arising from a toroidal symbol (see Section~\ref{sec:toroidal}). In Section~\ref{sec:toroidal} we clarify the relationships between the classes of standard \psidos\ and toroidal \psidos. In particular, we show that the two classes actually agree (Proposition~\ref{prop:toroidal.toroidal=standard}). As an application we are able to characterize smoothing operators as precisely the \psidos\ with Schwartz-class symbols (Proposition~\ref{prop:PsiDos.smoothing-condition}). 

As mentioned above, the construction of the oscillating integral for $\cA_\theta$-valued amplitudes is carried out by means of similar integration by parts arguments as those in the construction of the oscillating integral for scalar-valued amplitudes. The main difference is the fact that we consider amplitudes with values in the locally convex space $\cA_\theta$. Therefore, the bulk of the construction is to justify that the manipulations of integrals and the differentiation arguments  can be carried through for maps with values in locally convex spaces.  For this reason we have included two appendices on integration and differentiation of maps with values in locally convex spaces.  

It is well known how to extend the Riemann integral to maps with values in locally convex spaces (see Section~\ref{subsec:Riemann}). There are several extensions of Lebesgue integrals to maps with values in locally convex spaces. We have followed the approach of~\cite{Th:TAMS75}. This provides us with a notion of integral which mediates between the Gel'fand-Pettis integral and the Bochner integral. In fact, as $\cA_\theta$ is a nuclear space there is no real distinction between the aforementioned integrals for $\cA_\theta$-valued maps. Nevertheless, the exposition is given in the general setting of maps with values in quasi-complete Suslin locally convex spaces (see Section~\ref{subsec:Lebesgue}). This setting is large enough to include integration of maps with values in $\cA_\theta'$ or $\cL(\cA_\theta)$, which should be useful for further her purposes. 

It is fairly straightforward to deal with differentiation of maps on an open set $U\subset \R^d$ with values in a locally convex space, but we have attempted to give a detailed account for the sake of reader's convenience. 
Some care is devoted to the differentiation under the integral sign, for which the integral of~\cite{Th:TAMS75} proves to be especially useful (see Section~\ref{subec:diff-int}). Ultimately, this allows us to deal with the Fourier and inverse Fourier transform of Schwartz-class maps with values in quasi-complete Suslin locally convex spaces (see Section~\ref{subsec:Fourier-Schwartz}).  

Although this paper and the 2nd part~\cite{HLP:Part2} focus exclusively on noncommutative tori, most of the results of the papers continue hold for $C^*$-dynamical systems associated with the action of $\R^n$ on a unital $C^*$-algebras. In fact, all the results that that do not rely on Fourier series decompositions hold \emph{verbatim} in this general setting. 

There are various approaches to noncommutative tori. Accordingly, there are various types of pseudodifferential calculi depending on the perspective on noncommutative tori. In this paper and in the 2nd part we regard noncommutative tori as $C^*$-dynamical systems, which is natural from the point of view of noncommutative geometry. 
We refer to Section~\ref{sec:toroidal} for the equivalence of our class of \psidos\ of this paper with the toroidal \psidos\ of~\cite{LNP:TAMS16}. A pseudodifferential calculus for Heisenberg modules over noncommutative tori is given in~\cite{LM:GAFA16}. Noncommutative tori can also be interpreted as twisted crossed-products (see, e.g., \cite{Ri:CM90}). We refer to~\cite{BLM:JOT13, DL:RMP11, LMR:RIMS10, MPR:Bucharest05} for pseudodifferential calculi for twisted crossed-products with coefficients in $C^*$-algebras.  Another approach is to look at pseudodifferential operators on noncommutative tori as periodic pseudodifferential operators on noncommutative Euclidean spaces (see~\cite{GJP:MAMS17}). We also refer to~\cite{Liu:JNCG18} for an asymptotic pseudodifferential calculus on noncommutative toric manifolds. 

This paper is organized as follows. In Section~\ref{section:NCtori}, we survey the main facts regarding noncommutative tori. 
In Section~\ref{section:Symbols}, we introduce and derive the main properties of the classes of symbols that are used to define \psidos. 
In Section~\ref{section:Amplitudes}, we construct the oscillating integral for $\cA_\theta$-valued amplitudes and derive several of its properties. 
In Section~\ref{sec:PsiDOs}, we define the classes of \psidos\ on noncommutative tori associated with symbols and amplitudes, and derive some of their properties. In Section~\ref{sec:toroidal}, we clarify the relationship between the standard \psidos\ and toroidal \psidos\ and characterize smoothing operators. 
We also include three appendices. 
In Appendix~\ref{appendix:cAtheta-basic-properties}, we include proofs of some of the results on noncommutative tori mentioned in Section~\ref{section:NCtori}. 
In Appendix~\ref{app:LCS-int}, we gather a few facts on the integration of maps with values in locally convex spaces. In Appendix~\ref{app:LCS-diff}, we also gather some facts on the differentiation of maps with values in locally convex spaces. This includes a description of the Fourier transform for this type of maps. 

Jim Tao~\cite{Ta:JPCS18} has independently announced a detailed account on Connes' pseudodifferential calculus on noncommutative tori. The full details of his approach were not available at the time of completion of our paper series. 

This paper is part of the PhD dissertations of the first two named authors under the guidance of the third named author at Seoul National University (South Korea). 

\subsection*{Acknowledgements}
The authors would like to thank Edward McDonald, Hanfeng Li, Masoud Khalkhali, Max Lein, Franz Luef, Henri Moscovici, Javier Parcet, Fedor Sukochev, Jim Tao, Xiao Xiong, and Dmitriy Zanin for for useful discussions related to the subject matter of this paper. R.P.\ also wishes to thank McGill University (Montr\'eal, Canada) for its hospitality during the preparation of this paper.  

\section{Noncommutative Tori} \label{section:NCtori}
In this section, we review the main definitions and properties of noncommutative $n$-tori, $n\geq 2$. We refer to~\cite{Co:NCG, Ri:PJM81, Ri:CM90}, and the references therein, for a more comprehensive account. We also postpone to Appendix~\ref{appendix:cAtheta-basic-properties} the proofs of some of the results stated in this section.  

 \subsection{Noncommutative tori}  
In what follows we let $\T^n=\R^n\slash 2\pi \Z^n$ be the ordinary $n$-torus. In addition, we equip $L^2(\T^n)$ with the  inner product, 
\begin{equation} \label{eq:NCtori.innerproduct-L2}
 \acoup{\xi}{\eta}= \int_{\T^n} \xi(x)\overline{\eta(x)}\dbar x, \qquad \xi, \eta \in L^2(\T^n), 
\end{equation}
 where we have set $\dbar x= (2\pi)^{-n} dx$. 
 
 Given any $\theta \in \R$, the rotation of angle $\theta$ generates an action of $\Z$ on the unit circle $\bS^1\subset \C$ given by 
\begin{equation*}
k\cdot z = e^{2i\pi\theta k} z, \qquad k\in \Z, \ z\in \bS^1.  
\end{equation*}
 Upon identifying $\bS^1$ with $\T$, the corresponding crossed-product $C^*$-algebra $C^{0}(\bS^1)\rtimes \Z$ embeds into the $C^*$-algebra of operators on $L^2(\T)$ generated by the unitary operators $U$ and $V$ given by 
 \begin{equation*}
 (U\xi) (x)= e^{ix} \xi(x), \quad (V\xi)(x)=\xi\left( x+2\pi \theta\right), \quad \xi\in L^2(\T). 
\end{equation*}
This embedding is actually an isomorphism when $\theta \in \R\setminus \Q$. In any case, we have the relation, 
\begin{equation*}
 VU=e^{2i\pi \theta}UV. 
\end{equation*}

An equivalent representation of the above $C^*$-algebra is the $C^*$-algebra of operators on $L^2(\T^2)$ generated by the unitary operators $U_1$ and $U_2$ given by 
 \begin{equation*}
 (U_1\xi)(x)= e^{ix_1}\xi\left(x_1,x_2-\pi\theta\right), \quad (U_2\xi)(x)= e^{ix_2}\xi\left(x_1+\pi\theta,x_2\right), \quad \xi \in L^2(\T^2). 
\end{equation*}
Note that $U_2U_1=e^{2i\pi \theta}U_1U_2$. Moreover, for $\theta=0$ we obtain the $C^*$-algebra generated by $e^{ix_1}$ and $e^{ix_2}$. This is precisely the $C^*$-algebra $C^{0}(\T^2)$ of continuous functions on $\T^2$ represented by multiplication operators on $L^2(\T^2)$. 
 
 More generally, let $\theta =(\theta_{jk})$ be a real anti-symmetric $n\times n$-matrix ($n\geq 2$). We denote by $\theta_1, \ldots, \theta_n$ its column vectors. For $j=1,\ldots, n$ let $U_j:L^2(\T^n)\rightarrow L^2(\T^n)$ be the unitary operator given by 
 \begin{equation*}
 \left( U_j\xi\right)(x)= e^{ix_j} \xi\left( x+\pi \theta_j\right), \qquad \xi \in L^2(\T^n). 
\end{equation*}
 We then have the relations, 
 \begin{equation} \label{eq:NCtori.unitaries-relations}
 U_kU_j = e^{2i\pi \theta_{jk}} U_jU_k, \qquad j,k=1, \ldots, n. 
\end{equation}

\begin{definition}
 $A_\theta$ is the $C^*$-algebra generated by the unitary operators $U_1, \ldots, U_n$. 
\end{definition}
 
\begin{remark}
 For $\theta=0$ we obtain the $C^*$-algebra $C^{0}(\T^n)$ of continuous functions on the ordinary $n$-torus $\T^n$. When $\theta\neq 0$ the relations~(\ref{eq:NCtori.unitaries-relations}) imply that $A_\theta$ is a noncommutative algebra. In this case $A_\theta$ is called the \emph{noncommutative torus} associated with $\theta$. 
\end{remark}
 
\begin{remark}
  When $n=2$ simple characterizations of isomorphism classes and Morita equivalence classes of noncommutative tori have been established by Rieffel~\cite{Ri:PJM81}. 
  The extension of Rieffel's results to higher dimensional noncommutative tori is non-trivial (see~\cite{EL:ActaM07, EL:MathA08, Li:Crelle04, RS:IJM99}). 
 \end{remark}
  
The relations~(\ref{eq:NCtori.unitaries-relations}) imply that $A_\theta$ is the closure of the span of the unitary operators, 
 \begin{equation*}
 U^k:=U_1^{k_1} \cdots U_n^{k_n}, \qquad k=(k_1,\ldots, k_n)\in \Z^n. 
\end{equation*}
 For $k,l\in \Z^n$ set 
\begin{equation*}
c(k,l)=\sum_{q<p}k_p\theta_{pq}l_q. 
\end{equation*}
As $\theta$ is an anti-symmetric matrix, we have $c(k,l)-c(l,k)=k\cdot (\theta l)$. A calculation shows that 
 \begin{equation} \label{eq:NCtori.U^k-formula}
 \left( U^k\xi\right)(x)= e^{i\pi c(k,k)} e^{i k\cdot x}\xi\left( x+\pi \theta k\right), \qquad \xi \in L^2(\T^n). 
\end{equation}
Using this we obtain
\begin{equation} \label{eq:NCtori.U^kU^l-formula}
 U^k U^l= e^{-2i\pi c(k,l)} U^{k+l}, \qquad k,l \in \Z^n.
\end{equation}
 In particular, we get the relations, 
 \begin{equation} \label{eq:NCtori.U^k-inverse}
 \left(U^k\right)^{-1}= \left(U^k\right)^{*}=e^{-2i\pi c(k,k)}U^{-k}, \qquad U^lU^k=e^{2i\pi k\cdot (\theta l)} U^k U^l,  \qquad k,l \in \Z^n.
\end{equation}
 
 Let $\tau:\cL(L^2(\T^n))\rightarrow \C$ be the state defined by the constant function $1$, i.e., 
 \begin{equation*}
 \tau (T)= \acoup{T1}{1}=\int_{\T^n} (T1)(x) \dbar x, \qquad T\in \cL\left(L^2(\T^n)\right).
\end{equation*}
 In particular, as by~(\ref{eq:NCtori.U^k-formula}) we have $U^k1=e^{i\pi c(k,k)} e^{i k\cdot x}$, we obtain
 \begin{equation} \label{eq:NCtori.tauU^k}
 \tau\left( U^k\right) = \left\{
\begin{array}{ll}
 1 &  \text{if $k=0$},  \\
 0 &  \text{otherwise.}
 \end{array}\right. 
\end{equation}
We stress out that $\tau$ is a continuous linear form on $\cL(L^2(\T^n))$ of norm~$1$ (since this is a state). 

\begin{lemma}
 The functional $\tau$ defines a tracial state on $A_\theta$. 
\end{lemma}
 \begin{proof}
 We only need to show that $\tau$ is a trace on $A_\theta$, i.e., $\tau (uv)=\tau(vu)$ for all $u,v\in A_\theta$. As $\tau$ is continuous and the unitaries $U^k$, $k\in \Z^n$, span a dense subspace of $A_\theta$, it is enough to check that
 \begin{equation}\label{eq:NCtori.tau-tracial}
 \tau\left( U^k U^l\right) = \tau\left( U^l U^k\right) \qquad \text{for all $k,l\in \Z^n$}.
\end{equation}
When $k+l\neq 0$ it follows from~(\ref{eq:NCtori.U^kU^l-formula}) and~(\ref{eq:NCtori.tauU^k}) that $\tau(U^k U^l)=\tau(U^l U^k)=0$. Moreover, it follows from~(\ref{eq:NCtori.U^k-inverse}) that $U^kU^{-k}=U^{-k}U^{k}=e^{2i\pi c(k,k)}$, and so $\tau(U^kU^{-k})=\tau(U^{-k}U^{k})=e^{2i\pi c(k,k)}$. This proves~(\ref{eq:NCtori.tau-tracial}). The proof is complete. 
\end{proof}

\begin{remark}
 Let $\Lambda_\theta$ be the lattice generated by the columns of the matrix $\theta$ in $\R^n$. We say that $\theta$ is \emph{very irrational} when $\Lambda_\theta +\Z^n$ is dense in $\R^n$. In this case it can be shown that the algebra $A_\theta$ is simple and $\tau$ is its unique (normalized) trace (see~\cite{Gr:ActaM78, OPT:JOT80, Sl:CMP72}). 
\end{remark}

The GNS construction allows us to associate with $\tau$ a $*$-representation of $A_\theta$ as follows. Let $\acoup{\cdot}{\cdot}$ be the sesquilinear form on $A_\theta$ defined by
\begin{equation}
 \acoup{u}{v} = \tau\left( u v^*\right), \qquad u,v\in A_\theta. 
 \label{eq:NCtori.cAtheta-innerproduct}
\end{equation}
Given $k,l\in \Z^n$, using~(\ref{eq:NCtori.U^kU^l-formula}) and~(\ref{eq:NCtori.U^k-inverse}) we get
 \begin{equation*}
 \acoup{U^k}{U^l}=\tau\left( U^k(U^l)^*\right) = e^{-2i\pi c(l,l)}\tau\left (U^k U^{-l}\right)=e^{2i\pi \left(c(k,l)-c(l,l)\right)}\tau\left( U^{k-l}\right). 
\end{equation*}
Combining this with~(\ref{eq:NCtori.tauU^k}) then shows that $\{ U^k; k \in \Z^n\}$ is an orthonormal family. In particular, we see that we have a pre-inner product on the dense subalgebra, 
\begin{equation*}
 \cA_\theta^0:=\op{Span}\{ U^k; \ k\in \Z^n\}.
\end{equation*}

\begin{definition}
 $\cH_\theta$ is the Hilbert space arising from the completion of $\cA_\theta^0$ with respect to the pre-inner product~(\ref{eq:NCtori.cAtheta-innerproduct}). 
\end{definition}

When $\theta=0$ we recover the Hilbert space $L^2(\T^n)$ with the inner product~(\ref{eq:NCtori.innerproduct-L2}). In what follows we shall denote by $\|\cdot\|_0$ the norm of $\cH_\theta$. This notation allows us to distinguish it from the norm of $A_\theta$, which we denote by $\|\cdot\|$.

By construction $(U^k)_{k \in \Z^n}$ is an orthonormal basis of $\cH_\theta$. Thus, every $u\in \cH_\theta$ can be uniquely written as 
\begin{equation} \label{eq:NCtori.Fourier-series-u}
 u =\sum_{k \in \Z^n} u_k U^k, \qquad u_k=\acoup{u}{U^k}, 
\end{equation}
where the series converges in $\cH_\theta$. When $\theta =0$ we recover the Fourier series decomposition in  $L^2(\T^n)$. By analogy with the case $\theta=0$ we shall call the series $\sum_{k \in \Z^n} u_k U^k$ in~(\ref{eq:NCtori.Fourier-series-u}) the Fourier series of $u\in \cH_\theta$. Note that the Fourier series makes sense for all $u\in A_\theta$. 

\begin{proposition}\label{prop:NCTori.GNS-representation}
The following holds.
\begin{enumerate}
\item The multiplication of $\cA_\theta^0$ uniquely extends to a continuous bilinear map $A_\theta\times \cH_\theta \rightarrow \cH_\theta$. This provides us with a 
         $*$-representation of $A_\theta$. In particular, we have
          \begin{equation} \label{eq:NCtori.nrom0<norm-u}
                      \left\| u \right\| = \sup_{\|v\|_0=1} \|uv\|_0 \qquad \forall u \in A_\theta. 
           \end{equation}
\item The inclusion of $\cA_\theta^0$ into $\cH_\theta$ uniquely extends to a continuous embedding of $A_\theta$ into $\cH_\theta$. 
\end{enumerate}
\end{proposition}
\begin{proof}
 See Appendix~\ref{appendix:cAtheta-basic-properties}. 
\end{proof}

\begin{remark}
 The 2nd part  allows us to identify any $u \in A_\theta$ with the sum of its Fourier series in $\cH_\theta$. In general the Fourier series need not converge in $A_\theta$, although it does converge when $\sum |u_k|<\infty$. In addition, the injectivity of the embedding of $A_\theta$ into $\cH_\theta$ implies that the trace $\tau$ is faithful. 
\end{remark}

\subsection{Dynamics on $A_\theta$} The natural action of $\R^n$ on $\T^n$ by translation gives rise to a unitary representation $s\rightarrow V_s$ of $\R^n$ given by
\begin{equation*}
\left( V_s \xi\right)(x)=\xi(x+s), \qquad \xi \in L^2(\T^n), \ s\in \R^n. 
\end{equation*}
We then get an action $(s,T)\rightarrow \alpha_s(T)$ of $\R^n$ on $\cL(L^2(\T^n))$ given by 
\begin{equation} \label{eq:NCtori.action-on-T}
 \alpha_s(T)=V_s TV_s^{-1}, \qquad T\in \cL(L^2(\T^n)), \ s\in \R^n. 
\end{equation}
Note also that, for all $s\in \R^n$ and $T\in  \cL(L^2(\T^n))$, we have 
\begin{equation}
\alpha_s(T^*)=\alpha_s(T)^*, \qquad \|\alpha_s(T)\|= \|T\|, \qquad \tau\left[ \alpha_s(T)\right]=\tau\left[T\right].  
\label{eq:NCtori.cAtheta-groupaction}
\end{equation}
In addition, using~(\ref{eq:NCtori.action-on-T}) we get 
\begin{equation*}
 \alpha_s(U_j)=e^{is_j}U_j, \qquad j=1,\ldots, n. 
\end{equation*}
Thus,
\begin{equation} \label{eq:NCtori.action-on-U^k}
\alpha_s(U^k)= e^{is\cdot k} U^k, \qquad  k\in \Z^n. 
\end{equation}
This last property implies that the $C^*$-algebra $A_\theta$ is preserved by the action of $\R^n$. 

\begin{proposition} \label{prop:NCtori.continuity-action}
 The action of $\R^n$ on $A_\theta$ is continuous, i.e., $(s,u)\rightarrow \alpha_s(u)$ is a continuous map from $\R^n\times A_\theta$ to $A_\theta$. In particular, the triple $(A_\theta, \R^n, \alpha)$ is a $C^*$-dynamical system. 
\end{proposition}
\begin{proof}
 As above let $\cA_\theta^0$ be the subspace of $A_\theta$ spanned by the unitaries $U^k$, $k\in \Z^n$. It follows from~(\ref{eq:NCtori.action-on-U^k}) that $\alpha_s(u)\in C^\infty(\R^n; A_\theta)$ for all $u\in \cA_\theta^0$. Bearing this in mind, let $s,t\in \R^n$ and $u,v\in A_\theta$. In addition, let $\epsilon>0$. As $\cA_\theta^0$ is dense in $A_\theta$, there is $u_0\in \cA_\theta^0$ such that $\| u-u_0\|<\epsilon$. We have
 \begin{align*}
 \left\| \alpha_t(v)-\alpha_s(u)\right\|  & \leq \left\|\alpha_t(v-u_0)\right\|  + \left\| \alpha_t(u_0)-\alpha_s(u_0)\right\|  +\left\| \alpha_s(u-u_0)\right\|  \\
 & \leq \left\| v-u_0\right\|  + \left\| \alpha_t(u_0)-\alpha_s(u_0)\right\|  +\left\| u-u_0\right\| . 
\end{align*}
Here $\| u-u_0\|<\epsilon$. Moreover, as $\alpha_t(u_0)\in C^\infty(\R^n; A_\theta)$ we see that $\left\| \alpha_t(u_0)-\alpha_s(u_0)\right\| <\epsilon$ as soon as $t$ is close enough to $s$. It then follows that $\| \alpha_t(v)-\alpha_s(u)\| <3\epsilon$ as soon as $(t,v)$ is close enough to $(s,u)$. This shows that $(s,u)\rightarrow \alpha_s(u)$ is a continuous map from $\R^n\times A_\theta$ to $A_\theta$. The proof is complete. 
\end{proof}

We are especially interested in the subalgebra of smooth elements of the $C^*$-dynamical system $(A_\theta, \R^n, \alpha)$, i.e., the \emph{smooth noncommutative torus}, 
\begin{equation*}
 \cA_\theta:=\left\{ u \in A_\theta; \ \alpha_s(u) \in C^\infty(\R^n; A_\theta)\right\}. 
\end{equation*}
As mentioned in the proof of Proposition~\ref{prop:NCtori.continuity-action}, all the unitaries $U^k$, $k\in \Z^n$, are contained in $\cA_\theta$, and so $\cA_\theta$ is a dense subalgebra of $A_\theta$. When $\theta=0$ we recover the algebra $C^\infty(\T^n)$ of smooth functions on the ordinary torus $\T^n$. 

For $N\geq 1$, we also define 
\begin{equation*}
 A_\theta^{(N)}:=\left\{ u \in A_\theta; \ \alpha_s(u) \in C^N(\R^n; A_\theta)\right\}. 
\end{equation*}
We also set $A_\theta^{(0)}=A_\theta$. We note that $\cA_\theta$ and $A^{(N)}_{\theta}$, $N\geq 1$, are involutive subalgebras of $A_\theta$ that are preserved by the action of $\R^n$. 

For $j=1,\ldots, n$ let $\delta_j:A_\theta^{(1)}\rightarrow A_\theta$ be the  derivation defined by 
\begin{equation*}
 \delta_j(u) = D_{s_j} \alpha_s(u)|_{s=0}, \qquad u\in A_\theta^{(1)}, 
\end{equation*}
where we have set $D_{s_j}=\frac{1}{i}\partial_{s_j}$. We have the following properties: 
\begin{gather}
 \delta_j(uv)=\delta_j(u)v+u\delta_j(v), \qquad u,v\in A_\theta^{(1)},\label{eq:NCtori.derivation-Leibniz}\\
  \delta_j(u^*)=-\delta_j(u)^*, \qquad u\in A_\theta^{(1)},
   \label{eq:NCtori.derivation-involution} \\
    D_{s_j} \alpha_s(u) = \delta_j \left( \alpha_s(u)\right) =  \alpha_s\left( \delta_j(u)\right), \qquad u \in A_\theta^{(1)}, \ s\in \R^n,
    \label{eq:NCtori.derivation-groupaction}\\
     \delta_j \delta_l(u) = \delta_l \delta_j(u), \qquad u\in A_\theta^{(2)}, \   j,l=1,\ldots, n.  \nonumber
\end{gather}
When $\theta=0$ the derivation $\delta_j$ is just the derivation $D_{x_j}=\frac{1}{i}\frac{\partial}{\partial x_j}$ on $C^1(\T^n)$. In general, for $j,l=1,\ldots, n$, we have
\begin{equation*}
 \delta_j(U_l) = \left\{ 
 \begin{array}{ll}
 U_j & \text{if $l=j$},\\
 0 & \text{if $l\neq j$}. 
\end{array}\right.
\end{equation*}
In addition, we have the following result. 

\begin{lemma}[\cite{Ro:APDE08}]
 For $j=1,\ldots, n$, we have 
 \begin{gather}
 \tau\left[ \delta_j(u)\right] = 0 \qquad \forall u\in A_\theta^{(1)},\nonumber \\
 \tau\left[ u\delta_j(v)\right] =- \tau\left[ \delta_j(u)v\right] \qquad \forall u,v\in A_\theta^{(1)}. 
 \label{eq:NCtori.integration-by-parts}
\end{gather}
\end{lemma}
\begin{proof}
Let $u\in A_\theta^{(1)}$ and $j\in \{1, \ldots, n\}$. The continuity of $\tau$ and its $\R^n$-invariance as stated in~(\ref{eq:NCtori.cAtheta-groupaction}) imply that 
$\tau\left[ D_{s_j}\alpha_s(u)\right] = D_{s_j}\tau\left[ \alpha_s(u)\right] =0$. 
In particular, for $s=0$ we get $\tau\left[ \delta_j(u)\right] = 0$. In addition, by combining this with~(\ref{eq:NCtori.derivation-Leibniz}) we see that, given any $v\in A_\theta^{(1)}$, we have
\begin{equation*}
 \tau\left[ u\delta_j(v)\right] =\tau\left[ \delta_j(uv)\right] - \tau\left[ \delta_j(u)v\right]=- \tau\left[ \delta_j(u)v\right].
\end{equation*}
The lemma is proved. 
\end{proof}

More generally, given $u\in A_\theta^{(N)}$, $N\geq 2$, and $\beta \in \N_0^n$, $|\beta|=N$, we define 
\begin{equation*}
 \delta^\beta(u) = D_s^\beta \alpha_s(u)|_{s=0} = \delta_1^{\beta_1} \cdots \delta_n^{\beta_n}(u). 
\end{equation*}
We have the following properties:
\begin{gather}
\delta^\beta(uv)= \sum_{\beta'+\beta''=\beta} \binom{\beta}{\beta'} \delta^{\beta'}(u) \delta^{\beta''}(v), \qquad u,v\in A^{(N)}_\theta, \label{eq:NCtori-uv-Leibniz}\\
 \delta^\beta(u^*)=(-1)^{|\beta|}\delta^\beta(u)^*, \qquad u\in A^{(N)}_\theta, \label{eq:NCtori-involution-delta-compatibility}\\ 
D_s^\beta \alpha_s(u)= \alpha_s\left[\delta^\beta (u)\right]=\delta^\beta \left[\alpha_s(u)\right], \qquad u\in A^{(N)}_\theta,
\label{eq:NCtori.cAtheta-groupaction-differentiation}\\
\delta^\beta(U^k)= k^\beta U^k, \qquad k\in \Z^n. \label{eq:NCtori.delta-U^k}
\end{gather}

In what follows we equip $\cA_\theta$ with the locally convex topology defined by the semi-norms, 
\begin{equation*}
 \cA_\theta \ni u \longrightarrow \left\|\delta^\beta (u)\right\| ,  \qquad \beta\in \N_0^n. 
\end{equation*}
In particular, a sequence $(u_\ell)_{\ell \geq 0} \subset \cA_\theta$ converges to $u$ with respect to this topology if and only if 
\begin{equation}\label{eq:NCtori.convergence-u_l}
 \|\delta^\beta(u_\ell -u)\| \longrightarrow 0 \qquad \text{for all $\beta \in \N_0^n$}.
\end{equation}
In addition,  every multi-order derivation $\delta^\beta$ induces a continuous linear map $\delta^\beta: \cA_\theta\rightarrow \cA_\theta$.   

\begin{proposition} \label{prop:NCtori.cAtheta-Frechet}
The following holds.
\begin{enumerate}
 \item $\cA_\theta$ is a unital  Fr\'echet $*$-algebra.
 
 \item The action of $\R^n$ on $\cA_\theta$ is continuous, i.e., $(s,u)\rightarrow \alpha_s(u)$ is a continuous map from $\R^n\times \cA_\theta$ to $\cA_\theta$. 
 
 \item For every $u\in \cA_\theta$, the map $s\rightarrow \alpha_s(u)$ is a smooth map from $\R^n$ to $\cA_\theta$. 
\end{enumerate}
\end{proposition}
\begin{proof}
 See Appendix~\ref{appendix:cAtheta-basic-properties}. 
\end{proof}

In the following we denote by $\cS(\Z^n)$ the space of rapid-decay sequences $(a_k)_{k\in \Z^n}\subset \C$, i.e., sequences such that $\sup |k^\beta a_k|<\infty$ for all $\beta \in \N_0^n$. We equip it with the semi-norms,
\begin{equation*}
\cS(\Z^n)\ni (a_k)_{k\in \Z^n} \longrightarrow \sup_{k\in \Z^n} |k^\beta a_k|, \qquad  \beta \in \N_0^n. 
\end{equation*}
This turns $\cS(\Z^n)$ into a nuclear Fr\'echet-Montel space. The Fr\'echet-Montel property is a consequence of Tychonoff theorem. A proof of the nuclearity of $\cS(\Z^n)$ is given in~\cite{Tr:AP67}. Recall also that every Montel space is reflexive, i.e., the canonical embedding into its topological bidual is a linear homeomorphism (see~\cite{Tr:AP67}).

We have the following characterization of the elements of $\cA_\theta$. 

\begin{proposition}[\cite{Co:Foliations82}]\label{prop:NCtori.condition-cAtheta}
 The following holds. 
 \begin{enumerate}
\item  Let $u\in \cH_\theta$ have Fourier series $\sum u_k U^k$. Then $u\in \cA_\theta$ if and only if the sequence $(u_k)_{k \in \Z^n}$ is contained in $\cS(\Z^n)$. Moreover, in this case the Fourier series converges to $u$ in $\cA_\theta$. 

\item  The map $(u_k)\rightarrow \sum u_k U^k$ is a linear homeomorphism from $\cS(\Z^n)$ onto $\cA_\theta$.
\end{enumerate}
\end{proposition}
\begin{proof}
 See Appendix~\ref{appendix:cAtheta-basic-properties}. 
\end{proof}

\begin{remark}
We refer to~\cite{Ri:CM90}  for the explicit description of the product of $\cA_\theta$ in terms of Fourier series. In particular, this implies that $\cA_\theta$ is a strict deformation quantization of $C^\infty(\T^n)$ in the sense of~\cite{Ri:CMP89}. 
\end{remark}

As an immediate consequence of the 2nd part of Proposition~\ref{prop:NCtori.condition-cAtheta} we obtain the following result. 

\begin{corollary}\label{cor:NCTori.reflexive-NFM}
 $\cA_\theta$ is a nuclear Fr\'echet-Montel space. In particular, it is  reflexive. 
\end{corollary}

Let us now turn to the group of invertible elements of $\cA_\theta$. We shall denote this group by $\cA_\theta^{-1}$. We 
will also denote by $A_\theta^{-1}$ the  invertible group of $A_\theta$. 

\begin{proposition}[\cite{Co:AdvM81}]\label{prop:NCtori.invertibility-cAtheta}
The following holds. 
\begin{enumerate}
 \item We have $\cA_\theta^{-1}= A_\theta^{-1} \cap \cA_\theta$. 
 
 \item $\cA_\theta^{-1}$ is an open set of $\cA_\theta$ and $u\rightarrow u^{-1}$ is a continuous map from $\cA_\theta^{-1}$ to itself. 
\end{enumerate}
\end{proposition}
\begin{proof}
See Appendix~\ref{appendix:cAtheta-basic-properties}. 
\end{proof}

\begin{remark}
 The first part is a special case of a more general result of Schweitzer~\cite[Corollary~7.16]{Sc:IJM04} (see also~\cite{Bo:InvM90, Co:AdvM81, GVF:Birkh01, Po:PJM06}).  
\end{remark}

\begin{remark}
 The 2nd part means that $\cA_\theta$ is a good Fr\'echet algebra in the sense of~\cite[A.1.2]{Bo:InvM90}.
\end{remark}

Given any $u\in \cA_\theta$ we shall denote by $\Sp(u)$ its \emph{spectrum}, i.e., 
\begin{equation*}
 \Sp(u)=\left\{\lambda \in \C;\  u-\lambda \not\in \cA_\theta^{-1}\right\}. 
\end{equation*}
The first part of Proposition~\ref{prop:NCtori.invertibility-cAtheta} asserts there is no distinction between being invertible in $\cA_\theta$ or $A_\theta$. This implies that the spectrum of $u$ 
 relatively to $\cA_\theta$ agrees with its spectrum relatively to $A_\theta$, i.e., $\cA_\theta$ is spectral invariant in $A_\theta$. As $A_\theta$ is a $C^*$-algebra, it then follows that $\cA_\theta$ is spectral invariant in any $C^*$-algebra containing $\cA_\theta$. In particular, as we have a $*$-representation of $A_\theta$ in $\cH_\theta$, we obtain the following result. 
 
\begin{corollary}\label{cor:NCtori.spectrum-cHtheta}
For all $u\in \cA_\theta$, we have
\begin{equation*}
 \Sp (u) =\left\{\lambda \in \C; \ \text{$u-\lambda:\cH_\theta\rightarrow \cH_\theta$ is not a bijection}\right\}. 
\end{equation*}
\end{corollary}

There is a holomorphic functional calculus on good Fr\'echet algebras which is defined in the same way as for Banach algebras (see~\cite[A.1.5]{Bo:InvM90}). As we  have a continuous inclusion of $\cA_\theta$ into $A_\theta$, for elements of $\cA_\theta$ the holomorphic functional calculus on $A_\theta$ agrees with the functional calculus on $\cA_\theta$. Therefore, we arrive at the following statement. 

\begin{corollary}[\cite{Co:AdvM81}]
$\cA_\theta$ is stable under holomorphic functional calculus. 
\end{corollary}

\subsection{Distributions on $\cA_\theta$} \label{subsection:NCtori.Distributions}
Let $\cA_\theta'$ be the topological dual of $\cA_\theta$. We equip it with its {strong topology}. This is the locally convex topology generated by the semi-norms,
\begin{equation*}
v\longrightarrow\sup_{u\in B}|\acou{v}{u}| , \qquad \text{$B\subset \cA_\theta$ bounded}. 
\end{equation*}
It is tempting  to think of elements of $\cA_\theta'$ as distributions on $\cA_\theta$. This is consistent with the definition of distributions on $\mathbb{T}^{n}$ as continuous linear forms on $C^\infty(\mathbb{T}^n)$. 
 Any $u \in \cA_\theta$ defines a linear form on $\cA_\theta$ by
\begin{equation*}
 \acou{u}{v} =\tau(uv) \qquad \text{for all $v\in \cA_\theta$}.  
\end{equation*}
 Note that, for all $u,v \in \cA_\theta$, we have 
 \begin{equation} \label{eq:NCtori.distrb-innerproduct-eq}
  \acou{u}{v} =\acoup{v}{u^*}=\acoup{u}{v^*}.
\end{equation}
 In particular, given any $u \in \cA_\theta$, the map $v\rightarrow \acou{u}{v}$ is a continuous linear form on $\cA_\theta$. Moreover, as $\acou{u}{u^*}=\| u\|_0^2\neq 0$ if $u\neq 0$, we get a natural embedding of $\cA_\theta$ into $\cA_\theta'$. This allows us to identify $\cA_\theta$ with a subspace of $\cA_\theta'$. Furthermore, given any 
 bounded set $B\subset \cA_\theta$, we have 
\begin{equation*}
 \sup_{v \in B} |\acou{u}{v}| = \sup_{v \in B} |\acoup{u}{v^*}| \leq  \| u\|  \sup_{v \in B}\| v\|.
\end{equation*}
Therefore, we see that the embedding of $\cA_\theta$ into $\cA_\theta'$ is continuous. It is immediate from~(\ref{eq:NCtori.distrb-innerproduct-eq}) that this embedding uniquely extends to a continuous embedding of $\cH_\theta$ into $\cA_\theta$. 

Combining~(\ref{eq:NCtori.distrb-innerproduct-eq}) with the above embedding allows us to extend the definition of Fourier series to $\cA_\theta'$. Namely, given any $v\in \cA_\theta'$ its Fourier series is the series, 
\begin{equation}
 \sum_{k\in \Z^n} v_k U^k, \qquad \text{where}\ v_k:= \acou{v}{(U^k)^*}. 
 \label{eq:NCtori.Fourier-series}
\end{equation}
Here the unitaries $U^k$, $k\in \Z^n$, are regarded as elements of $\cA_\theta'$. 

In what follows we denote by $\cS'(\Z^n)$ the (topological) dual of $\cS(\Z^n)$. It is naturally identified 
with the space of sequences $(v_k)_{k\in\Z^n}\subset \C$ for which there are $N\geq 0$ and $C_N>0$ such that 
\begin{equation*} 
|v_k|\leq C_N(1+|k|)^N \qquad \text{for all $k\in\Z^n$}.
\end{equation*}

\begin{proposition}\label{prop:NCtori.distributions-Fourier-series}
Let $v\in \cA_\theta'$ have Fourier series $ \sum_{k\in \Z^n} v_k U^k$. Then $(v_k)_{k\in \Z^n}\in \cS'(\Z^n)$ and $v$ is equal to the sum of its Fourier series in $\cA_\theta'$.  
\end{proposition}
\begin{proof}
As $v$ is a continuous linear form on $\cA_\theta$ there are $N\in \N_0$ and $C_N>0$ such that 
\begin{equation*}
\left|\acou{v}{u}\right|\leq C_N \sup_{|\alpha|\leq N} \|\delta^\alpha (u)\| \qquad \text{for all $u \in \cA_\theta$}. 
\end{equation*}
In particular, for all $k\in \Z^n$, we have 
\begin{equation*} 
 |v_k|= \left|\acou{v}{(U^{k})^*}\right|\leq C_N \sup_{|\alpha|\leq N} \left\|\delta^\alpha \left(U^{k}\right)^*\right\|  =  C_N \sup_{|\alpha|\leq N}  \left\| k^\alpha U^{k}\right\| 
 \leq C_N(1+|k|)^N. 
\end{equation*}
It then follows that $(v_k)_{k\in\Z^n}\in \cS'(\Z^n)$. 

Let $u\in \cA_\theta$. Using~(\ref{eq:NCtori.distrb-innerproduct-eq}) we see that $ \overline{\acoup{u^*}{U^k}}=\acoup{U^k}{u^*}=\acou{U^k}{u}$ for all $k \in \Z^n$. Thus, 
\begin{equation*}
 u=(u^*)^*= \sum_{k\in \Z^n} \overline{\acoup{u^*}{U^k}} \left(U^k\right)^*=  \sum_{k\in \Z^n}\acou{U^k}{u} \left(U^k\right)^*,
 \end{equation*}
 where the series converge in $\cA_\theta$. As $v$ is  a continuous linear form on $\cA_\theta$, we get
\begin{equation*}
 \acou{v}{u}= \sum_{k \in \Z^n} \acou{U^k}{u} \acou{v}{\left(U^k\right)^*}= \sum_{k \in \Z^n} v_{k}\acou{U^{k}}{u}.
\end{equation*}
This shows that the series $\sum_{k\in \Z^n} v_kU^{k}$ converges weakly to $v$. As $\cA_\theta$ is a Fr\'echet space, the Banach-Steinhaus theorem ensures us that we have convergence with respect to the strong topology of $\cA_\theta'$. The proof is complete.
\end{proof}

\begin{remark}
The above result implies that we have a natural linear map $v \rightarrow (v_k)_{k\in \Z^n}$ from $\cA_\theta'$ to $\cS'(\Z^n)$. This is actually the transpose of the linear homeomorphism of $\cS(\Z^n)\ni (u_k) \rightarrow \sum u_k U^k\in \cA_\theta $ of Proposition~\ref{prop:NCtori.condition-cAtheta}. 
Therefore, if we equip $\cS'(\Z^n)$ with its strong topology then we obtain a linear homeomorphism from $\cA_\theta'$ onto $\cS'(\Z^n)$.
\end{remark}

We observe that in any Fourier series~(\ref{eq:NCtori.Fourier-series}) every summand is an element of $\cA_\theta$. Therefore, as an immediate consequence of Proposition~\ref{prop:NCtori.distributions-Fourier-series} we obtain the following density result. 

\begin{corollary}
 The inclusion of $\cA_\theta$ into $\cA_\theta'$ has dense range. 
\end{corollary}

\subsection{Differential operators}  In this subsection, we review a few facts on differential operators on noncommutative tori. 

\begin{definition}[\cite{Co:CRAS80, Co:NCG}]
 A differential operator of order~$m$ on $\cA_\theta$ is a linear operator $P: \cA_\theta \rightarrow \cA_\theta$ of the form, 
\begin{equation} \label{eq:NCtori.differential-operator}
 P= \sum_{|\alpha|\leq m} a_\alpha \delta^\alpha, \qquad a_\alpha \in \cA_\theta. 
\end{equation}
\end{definition}

\begin{remark}
 In~(\ref{eq:NCtori.differential-operator}) each coefficient $a_\alpha$, $|\alpha|\leq m$,  is identified with the operator of left-multiplication by $a_\alpha$. Thus, (\ref{eq:NCtori.differential-operator}) means that
\begin{equation*}
 Pu= \sum_{|\alpha|\leq m} a_\alpha \delta^\alpha u \qquad \text{for all $u\in \cA_\theta$}. 
\end{equation*}
\end{remark}

\begin{remark}
 Any differential operator is a continuous linear operator on $\cA_\theta$. 
\end{remark}

\begin{remark} \label{rem:NCtori.differential-op}
 Let $ P= \sum_{|\alpha|\leq m} a_\alpha \delta^\alpha$ be a differential operator of order~$m$. The \emph{symbol} of $P$ is the polynomial map $\rho: \R^n \rightarrow \cA_\theta$ defined by 
\begin{equation*}
 \rho(\xi)  =  \sum_{|\alpha|\leq m} a_\alpha \xi^\alpha, \qquad \xi\in \R^n. 
\end{equation*}
Its $m$-th degree component $\rho_m(\xi)=   \sum_{|\alpha|= m} a_\alpha \xi^\alpha$ is called the \emph{principal symbol} of $P$. 
\end{remark}

\begin{example}
 The (flat) Laplacian of $\cA_\theta$ is the 2nd order differential operator, 
\begin{equation} \label{eq:NCtori.flat-Laplacian}
 \Delta = \delta_1^2+ \cdots + \delta_n^2. 
\end{equation}
Its symbol is $\rho(\xi)= \xi_1^2 + \cdots + \xi_n^2= |\xi|^2$. 
\end{example}

\begin{example}[\cite{HP:Laplacian}] \label{ex:NCtori.Laplacian-Riemannian}
 In~\cite{Ro:SIGMA13} a Riemannian metric on $\cA_\theta$ is given by a positive invertible matrix $g=(g_{ij})\in M_n(\cA_\theta)$ whose entries are selfadjoint elements of $\cA_\theta$. Its determinant is defined by
\begin{equation*}
 \det (g):=\exp\big( \Tr[ \log(g)]\big), 
\end{equation*}
where $\log(g) \in M_n(\cA_\theta)$ is defined by holomorphic functional calculus and $\Tr$ is the matrix trace (see~\cite{HP:Laplacian}). The determinant $\det(g)$ is a positive invertible element of $\cA_\theta$, and so $\nu(g):= \sqrt{\det(g)}$ is a positive invertible element of $\cA_\theta$. Let $g^{-1}=(g^{ij})$ be the inverse matrix of $g$. In~\cite{HP:Laplacian} the Laplace-Beltrami operator associated with $g$ is the 2nd order differential operator $\Delta_g:\cA_\theta \rightarrow \cA_\theta$ given by
\begin{equation*}
 \Delta_g u =  \nu(g)^{-1} \sum_{1\leq i,j\leq n} \delta_i \left( \sqrt{\nu(g)} g^{ij} \sqrt{\nu(g)}\delta_j(u)\right), \qquad u\in \cA_\theta. 
\end{equation*}
When $g_{ij}=\delta_{ij}$ we recover the flat Laplacian~(\ref{eq:NCtori.flat-Laplacian}). When $\theta=0$ we have $\sqrt{\nu(g)} g^{ij} \sqrt{\nu(g)}=g^{ij}\nu(g)$, and so we recover the usual expression for the Laplace-Beltrami operator in Euclidean coordinates with $\delta_j=\frac{1}{\sqrt{-1}} \partial_{j}$. 
\end{example}

The following result shows that differential operators form a graded algebra. 

\begin{proposition}
 Suppose that $P$ and $Q$ are differential operators on $\cA_\theta$ of respective orders $m$ and $m'$. Then $PQ$ is a differential operator of order~$\leq m+m'$. 
\end{proposition}
\begin{proof}
It is enough to prove the result when $P=a\delta^\alpha$ and $Q=b\delta^\beta$ with $a,b\in \cA_\theta$ and $|\alpha|\leq m$ and $|\beta|\leq m'$. In fact by the Leibniz formula, for all $u\in \cA_\theta$, we have 
\begin{equation*}
 PQu=a\delta^\alpha \left(b \delta^\beta(u)\right)= \sum_{\alpha'+\alpha''=\alpha} \binom{\alpha}{\alpha'} a \delta^{\alpha'}(b) \delta^{\alpha''+\beta}(u). 
\end{equation*}
This shows that $PQ$ is a differential operator of order~$\leq |\alpha|+|\beta|\leq m+m'$. The result is proved. 
\end{proof}

In order to motivate the definition of pseudodifferential operators on $\cA_\theta$ we would like to outline an integral representation of differential operators. At this point we shall not worry about convergence issues or integration by part justifications. We will see later that all the arguments can be made rigorous  by interpreting the integrals at stake as oscillating integrals. 

Let $P=  \sum_{|\alpha|\leq m} a_\alpha \delta^\alpha$ be a differential operator of order $m$. Let $u\in \cA_\theta$. Note that $\delta^\alpha(u)= \left. (i\partial_t)^\alpha\alpha_{-t}(u)\right|_{t=0}$. Therefore, by using Fourier's inversion formula and integrating by parts we see that $\delta^\alpha(u) $ is equal to 
\begin{align*}
 \iint e^{is\cdot \xi}  (i\partial_s)^\alpha\alpha_{-s}(u)ds\dbar \xi & = \iint (-i\partial_s)^\alpha \left[ e^{is\cdot \xi}\right]  \alpha_{-s}(u)ds\dbar \xi \\
 & = \iint  e^{is\cdot \xi} \xi^\alpha \alpha_{-s}(u)ds\dbar \xi.
\end{align*}
Therefore, by linearity we get 
\begin{equation*}
 Pu =  \sum_{|\alpha|\leq m} a_\alpha \delta^\alpha(u) =   \sum_{|\alpha|\leq m}  \iint  e^{is\cdot \xi} a_\alpha \xi^\alpha \alpha_{-s}(u)ds\dbar \xi. 
\end{equation*}
Thus, if we let  $\rho(\xi)  =  \sum_{|\alpha|\leq m} a_\alpha \xi^\alpha$ be the symbol of $P$, then we have 
\begin{equation}
 Pu=  \iint  e^{is\cdot \xi}  \rho(\xi) \alpha_{-s}(u)ds\dbar \xi \qquad \text{for all $u\in \cA_\theta$}. 
 \label{eq:NCtori.diff-op-integral}
\end{equation}

The above formula is the main impetus for the definition of pseudodifferential operators on noncommutative tori. In Section~\ref{section:Symbols} and Section~\ref{section:Amplitudes} 
we shall explain how to give sense to the integrals considered above. This will allow us to define pseudodifferential operators on noncommutative tori associated with the classes of symbols that are introduced in the next section. 

\section{Classes of Symbols on Noncommutative Tori} \label{section:Symbols}
In this section, we review the main classes of symbols on noncommutative tori. 

\subsection{Standard symbols} 

\begin{definition}[\cite{Ba:CRAS88, Co:CRAS80}]
$\stS^m (\Rn ; \cA_\theta)$, $m\in\R$, consists of maps $\rho(\xi)\in C^\infty (\Rn ; \cA_\theta)$ such that, for all multi-orders $\alpha$ and $\beta$, there exists $C_{\alpha \beta} > 0$ such that
\begin{equation*} 
\norm{\delta^\alpha \partial_\xi^\beta \rho(\xi)} \leq C_{\alpha \beta} \left( 1 + | \xi | \right)^{m - | \beta |} \qquad \forall \xi \in \R^n .
\end{equation*}
\end{definition}

\begin{remark}\label{rmk:Symbols.amplitudes-intersection}
 We have 
\begin{equation*}
 \bigcap_{m\in\R}\stS^m( \R^n;\cA_\theta)  = \cS(\R^n; \cA_\theta), 
\end{equation*}\
where $\cS(\R^n; \cA_\theta)$ is the space of Schwartz-class maps $\rho:\R^n\rightarrow \cA_\theta$ (\cf\ Appendix~\ref{app:LCS-diff}). 
\end{remark}

In what follows, we endow each space $\stS^m(\Rn;\cA_\theta)$, $m\in \R$, with the locally convex topology generated by the semi-norms,
\begin{equation} \label{eq:Symbols.standard-semi-norms}
p_N^{(m)}(\rho):=\sup_{|\alpha|+|\beta|\leq N} \sup_{\xi\in\Rn}(1+|\xi|)^{-m+|\beta|}\norm{\delta^\alpha\partial_\xi^\beta\rho(\xi)}, \qquad N\in\N_0 .
\end{equation}

\begin{proposition}[\cite{Ba:CRAS88}] \label{prop:Symbols.standard-Frechetspace}
$\stS^m(\R^n ;\cA_\theta)$, $m\in \R$,  is a Fr\'{e}chet space.
\end{proposition}
\begin{proof} As the semi-norm family $(p_N^{(m)})_{N\geq 0}$ is countable, we only need to check that every Cauchy sequence  in  $\stS^m(\Rn;\cA_\theta)$ is convergent. Thus, let $\left(\rho_j(\xi)\right)_{j \geq 0}$ be a Cauchy sequence in $\stS^m(\Rn ; \cA_\theta )$. As the inclusion of $\stS^m(\Rn;\cA_\theta)$ into $C^\infty ( \Rn ; \cA_\theta )$ is continuous, this gives rise to a Cauchy sequence in $C^\infty (\Rn ; \cA_\theta )$. Since $C^\infty  (\Rn ; \cA_\theta )$ is a Fr\'echet space (\cf\ Appendix~\ref{app:LCS-diff}), it then follows that there is $\rho(\xi)\in C^\infty (\Rn ; \cA_\theta )$ such that $\rho_j(\xi) \rightarrow \rho(\xi)$ in $C^\infty (\Rn ; \cA_\theta )$ as $j\rightarrow \infty$. 

Moreover, the fact that $\left(\rho_j(\xi)\right)_{j \geq 0}$ is a Cauchy sequence in $\stS^m(\Rn ; \cA_\theta )$ means that, given any multi-orders $\alpha$, $\beta$, for every $\epsilon >0$, there is $N\in \N$ such that, for all $j,l\geq N$ and $\xi\in \R^n$,  we have
\begin{equation*}
 \left\| \delta^\alpha \partial_\xi^\beta \left(\rho_j(\xi) -  \rho_l(\xi)\right)  \right\| \leq \epsilon (1 +  |\xi| )^{m-|\beta|}.
\end{equation*}
Letting $l\rightarrow \infty$ then shows that, for all $j\geq N$ and $\xi\in \R^n$,  we have
\begin{equation*}
 \left\| \delta^\alpha \partial_\xi^\beta \left(\rho_j(\xi) -  \rho(\xi)\right)  \right\| \leq \epsilon (1 +  |\xi| )^{m-|\beta|}. 
\end{equation*}
It then follows that $\rho_j(\xi) \rightarrow \rho(\xi)$ in $\stS^m(\Rn;\cA_\theta)$.  This shows that every Cauchy sequence  in  
$\stS^m(\Rn;\cA_\theta)$ is convergent, and so $\stS^m(\Rn;\cA_\theta)$ is a Fr\'{e}chet space. 
\end{proof}

\begin{remark}
 It follows from the very definition of the spaces $\stS^m(\R^n ;\cA_\theta)$, $m\in \R$,  that, given any multi-orders $\alpha$ and $\beta$, the partial differentiation $ \delta^\alpha\partial_\xi^\beta$ gives rise to a continuous linear operator from $\stS^m(\R^n ;\cA_\theta)$ to $\stS^{m-|\beta|}(\R^n ;\cA_\theta)$ for every $m\in \R$.  
\end{remark}

\begin{lemma} \label{lem:Symbols.standard-product}
Let $m_1, m_2\in\R$. Then the product of $\cA_\theta$ gives rise to a continuous bilinear map from $\stS^{m_1}(\Rn;\cA_\theta)\times\stS^{m_2}(\Rn;\cA_\theta)$ to $\stS^{m_1+m_2}(\Rn;\cA_\theta)$. 
\end{lemma}
\begin{proof}
Let $N\in \N_0$ and $\alpha, \beta\in \N_0^n$ be such that $|\alpha|+|\beta|\leq N$. In addition, let $\rho_j(\xi)\in \stS^{m_j}(\Rn;\cA_\theta)$, $j=1,2$. 
By the Leibniz rule we have
\begin{equation*} 
\delta^\alpha\partial_\xi^\beta \left[ \rho_1(\xi)\rho_2(\xi)\right]=\sum \binom \alpha{\alpha'} \binom \beta{\beta'}  \delta^{\alpha'}\partial_\xi^{\beta'}
\rho_1(\xi)\delta^{\alpha''}\partial_\xi^{\beta''}\rho_2(\xi) ,
\end{equation*}
where the sum ranges over all multi-orders $\alpha',\alpha'',\beta',\beta'$ such that $\alpha'+\alpha''=\alpha$ and $\beta'+\beta''=\beta$. Note that in this case, 
for all $\xi \in \R^n$,  we have 
\begin{align*}
\left\| \delta^{\alpha'}\partial_\xi^{\beta'}\rho_1(\xi) \delta^{\alpha''}\partial_\xi^{\beta''} \rho_2(\xi)\right\|  
& \leq \left\| \delta^{\alpha'}\partial_\xi^{\beta'}\rho_1(\xi)\right\|  \left\| \delta^{\alpha''}\partial_\xi^{\beta''} \rho_2(\xi) \right\|  \\
& \leq p_N^{(m_1)}(\rho_1)(1+|\xi|)^{m_1-|\beta'|} p_N^{(m_2)}(\rho_2) (1+|\xi|)^{m_2-|\beta''|}\\
& \leq p_N^{(m_1)}(\rho_1)p_N^{(m_2)}(\rho_2)(1+|\xi|)^{m_1+m_2-|\beta|}. 
\end{align*}
Thus,  for all $\xi \in \R^n$,  we have 
\begin{equation*}
 \left\| \delta^\alpha\partial_\xi^\beta \left[ \rho_1(\xi) \rho_2(\xi)\right]\right\|  \leq 2^{|\alpha|+|\beta|} p_N^{(m_1)}(\rho_1) p_N^{(m_2)}(\rho_2)
  (1+|\xi|)^{m_1+m_2-|\beta|}. 
\end{equation*}
Therefore, we see that $\rho_1(\xi)\rho_2(\xi)$ belongs to $\stS^{m_1+m_2}(\Rn;\cA_\theta)$, and we have 
\begin{equation*}
 p_N^{(m_1+m_2)}(\rho_1\rho_2) \leq 2^N p_N^{(m_1)}(\rho_1) p_N^{(m_2)}(\rho_2). 
\end{equation*}
This shows that the product of $\cA_\theta$ gives rise to a continuous bilinear map from $\stS^{m_1}(\Rn;\cA_\theta)\times\stS^{m_2}(\Rn;\cA_\theta)$ to $\stS^{m_1+m_2}(\Rn;\cA_\theta)$. The proof is complete. 
\end{proof}

\begin{lemma}[{\cite[Prop.~18.1.2]{Ho:Springer85}}]\label{lem:Symbols.approximate-unit} 
Let $\chi(\xi)\in \cS(\Rn)$ be such that $\chi(0)=1$. For $0<\epsilon\leq 1$, let $\chi_\epsilon(\xi)\in \cS( \R^n; \cA_\theta)$ be defined by 
\begin{equation*}
 \chi_\epsilon(\xi)= \chi(\epsilon \xi)\cdot 1, \qquad \xi \in \R^n. 
\end{equation*}
 Then the family $(\chi_\epsilon)_{0<\epsilon\leq 1}$ is bounded in $\stS^0(\Rn; \cA_\theta)$ and, as $\epsilon \rightarrow 0^+$, it converges to $1$  in 
$\stS^m(\Rn; \cA_\theta)$ for every $m>0$. 
\end{lemma}

By combining Lemma~\ref{lem:Symbols.standard-product} and Lemma~\ref{lem:Symbols.approximate-unit} we obtain the following approximation result. 

\begin{proposition} \label{prop:Symbols.standard-density}
Let $\rho(\xi)\in \stS^m(\Rn;\cA_\theta)$, $m\in\R$, and let $\chi(\xi)\in \cS(\Rn)$ be such that $\chi(0)=1$.  For $0<\epsilon\leq 1$, let $\rho_\epsilon(\xi)\in \cS(\Rn; \cA_\theta)$ be defined by 
\begin{equation*}
 \rho_\epsilon(\xi)= \chi(\epsilon \xi) \rho(\xi), \qquad \xi \in \R^n. 
\end{equation*}
 Then the family $(\rho_\epsilon(\xi))_{0<\epsilon\leq 1}$ is bounded in $\stS^m(\Rn;\cA_\theta)$ and, as $\epsilon \rightarrow 0^+$, it converges to $\rho(\xi)$  in $\stS^{m'}(\Rn;\cA_\theta)$ for every $m'>m$. 
\end{proposition}

\begin{remark}\label{rmk:Symbols.reduction-cS}
In the rest of the paper, we will often use Proposition~\ref{prop:Symbols.standard-density} to reduce the proof of equalities for continuous functionals on standard symbols to proving them for maps in $C^\infty_c(\R^n; \cA_\theta)$ or in $\cS(\R^n; \cA_\theta)$.   
\end{remark}

\begin{definition}[\cite{Ba:CRAS88}] \label{def:Symbols.standard-asymptotic}
Let $\rho(\xi)\in\stS^m(\Rn;\cA_\theta)$, $m\in\R$, and, for $j=0, 1, \ldots $, let $\rho_j(\xi)\in\stS^{m-j}(\Rn;\cA_\theta)$. We shall write 
$\rho(\xi)\sim\sum_{j\geq 0}\rho_j(\xi)$ when 
\begin{equation*} 
\rho(\xi)-\sum_{j<N}\rho_j(\xi)\in\stS^{m-N}(\Rn;\cA_\theta) \qquad \text{for all $N\geq 1$}. 
\end{equation*}
\end{definition}

\begin{lemma}[see also~\cite{FW:JPDOA11}] \label{lem:Symbols.standard-construction}
Let $m\in \R$ and, for $j=0,1,\ldots$, let $\rho_{m-j}(\xi)\in\stS^{m-j}(\Rn;\cA_\theta)$. Then there is a symbol  $\rho(\xi)\in\stS^m(\Rn;\cA_\theta)$ such that $\rho(\xi)\sim\sum_{j\geq 0}\rho_{m-j}(\xi)$. Moreover, $\rho(\xi)$ is unique modulo the addition of a symbol in $\cS(\Rn;\cA_\theta)$. 
\end{lemma}\begin{proof}
The proof is similar to the standard proof of the Borel lemma for symbols (see, e.g., \cite{AG:AMS07, Ho:Springer85}). Let $\chi(\xi) \in C^\infty_c (\Rn)$ be such that $\chi(\xi) = 1$ for $|\xi| \leq 1$. 
By Proposition~\ref{prop:Symbols.standard-density} we know that, for all $j\geq 0$, the family $(\chi(\epsilon \xi) \rho_{m-j}(\xi))_{0<\epsilon<1}$ converges to $\rho_{m-j}(\xi)$ in $\stS^{m-j+1}(\R^n; \cA_\theta)$ as $\epsilon \rightarrow 0^+$. Therefore,  we can  recursively construct a sequence $(\epsilon_j)_{j \geq 0}\subset (0,1)$ such that
\begin{equation} \label{eq:Symbols.cutoff-standard-estimates}
\epsilon_{j+1} \leq \frac{1}2 \epsilon_j \qquad \text{and}\qquad  p_{j}^{(m-j+1)}\left( \left[1-\chi(\epsilon_j \xi) \right]\rho_{m-j}\right)\leq 2^{-(j+1)} \qquad \text{for all $j\geq 0$}. 
\end{equation}
Note that the sequence $(\epsilon_j)_{j \geq 0}$ is decreasing and converges to $0$. 

For $j=0,1, \ldots$, set $\tilde{\rho}_{m-j}(\xi)=(1-\chi(\epsilon_j \xi)) \rho_{m-j} (\xi)$, 
$\xi\in \R^n$. Note that $\tilde{\rho}_{m-j}(\xi)=0$ for $|\xi|\leq \epsilon_j^{-1}$. As $\epsilon_j^{-1} \rightarrow \infty$ as $j\rightarrow \infty$ we see that the sum $\rho(\xi):=\sum_{j\geq 0} \tilde{\rho}_{m-j}(\xi)$ is locally finite, and so this gives rise to a map in $C^\infty(\R^n; \cA_\theta)$.  We also observe that 
$\tilde{\rho}_{m-j}(\xi)$ differs from  $\rho_{m-j}(\xi)$  by a map in $C^\infty_c(\R^n;\cA_\theta)$. Thus $\tilde{\rho}_{m-j}(\xi)$ is a symbol in 
$\stS^{m-j}(\R^n; \cA_\theta)$, and, for every integer $N\geq 1$, we have 
\begin{align}
 \rho(\xi) - \sum_{j<N} \rho_{m-j}(\xi) & =  \rho(\xi) - \sum_{j<N} \tilde{\rho}_{m-j}(\xi) \quad \bmod \cS(\R^n; \cA_\theta)\nonumber \\ 
  & = \sum_{j\geq N} \tilde{\rho}_{m-j}(\xi) \quad \bmod \cS(\R^n; \cA_\theta). 
   \label{eq:Symbols.rho-N-estimates}
\end{align}

Let $N\in \N_0$ and set $\rho^{(N)}(\xi) = \sum_{j\geq N}  \tilde{\rho}_{m-j}(\xi)$, $\xi\in \R^n$. Let $\alpha$ and $\beta$ be multi-orders, and set $N_1=\op{max}(N+1,|\alpha|+|\beta|)$. For all $\xi \in \R^n$, we have 
\begin{equation*}
 \left\|\delta^\alpha  \partial_\xi^\beta \rho^{(N)}(\xi)\right\|  \leq \sum_{j\geq N} p^{(m-N)}_{N_1}\left( \tilde{\rho}_{m-j}\right) (1+|\xi|)^{m-N-|\beta|}. 
\end{equation*}
Note that if $j \geq N_1$, then $j\geq N+1$. Therefore, using~(\ref{eq:Symbols.cutoff-standard-estimates}) we get 
\begin{equation*}
 \sum_{j\geq N_1} p^{(m-N)}_{N_1}\left( \tilde{\rho}_{m-j}\right) \leq  \sum_{j\geq N_1} p^{(m-j+1)}_{j}\left( \tilde{\rho}_{m-j}\right)\leq 2^{-N_1}. 
\end{equation*}
It then follows that there is a constant $C_{N\alpha\beta}>0$ such that, for all $\xi \in \R^n$, we have
\begin{equation*}
  \left\|\delta^\alpha  \partial_\xi^\beta \rho^{(N)}(\xi)\right\|  \leq C_{N\alpha\beta}  (1+|\xi|)^{m-N-|\beta|}. 
\end{equation*}
This shows that $ \rho^{(N)}(\xi)\in \stS^{m-N}(\R^n;\cA_\theta)$. In particular, for $N=0$ we see that $\rho(\xi)=\rho^{(0)}(\xi)$ is contained in $\stS^m(\R^n; \cA_\theta)$. In addition, by combining this with~(\ref{eq:Symbols.rho-N-estimates}) we see that $ \rho(\xi) - \sum_{j<N} \rho_{m-j}(\xi)$ is contained in $\stS^{m-N}(\R^n;\cA_\theta)$ for all $N\geq 1$. That is, $\rho(\xi) \sim \sum_{j\geq 0} \rho_{m-j}(\xi)$. 

Finally, if $\sigma(\xi) \in \stS^m(\R^n; \cA_\theta)$ is another symbol such that $\sigma(\xi) \sim \sum_{j\geq 0} \rho_{m-j}(\xi)$, then we have  $\sigma(\xi)\sim \rho(\xi)$, i.e., $\sigma(\xi)- \rho(\xi)$ is contained in $\cS(\R^n; \cA_\theta)$. Thus, the symbol $\rho(\xi)$ is unique modulo the addition of a symbol in $\cS(\Rn;\cA_\theta)$. The proof is complete. 
\end{proof}

\subsection{Homogeneous and classical symbols} 

\begin{definition}[Homogeneous Symbols]
$S_q (\R^n; \cA_\theta )$, $q \in \C$, consists of  maps $\rho(\xi) \in C^\infty(\R^n\backslash 0;\cA_\theta)$ that are homogeneous of degree $q$, i.e., 
\begin{equation*}
\rho( \lambda \xi ) = \lambda^q \rho(\xi) \qquad \text{for all $\xi \in \R^n \backslash 0$ and $\lambda > 0$}. 
\end{equation*}
\end{definition}

\begin{remark} \label{rem:Symbols.homogeneous-differentiation}
 Let $\rho(\xi) \in S_q (\R^n; \cA_\theta )$, $q\in \C$. Then $\delta^\alpha \partial_\xi^\beta \rho(\xi)\in S_{q-|\beta|} (\R^n; \cA_\theta )$ for all $\alpha,\beta \in \N_0^n$. \end{remark}

\begin{remark} \label{rem:Symbols.homogeneous-involution}
 Let $\rho(\xi) \in S_q (\R^n; \cA_\theta )$, $q\in \C$. Then $\rho(\xi)^*$ is homogeneous of degree $\overline{q}$, and so $\rho(\xi)^*\in S_{\overline{q}} (\R^n; \cA_\theta )$. 
\end{remark}

\begin{definition}[Classical Symbols; \emph{cf}.~\cite{Ba:CRAS88}]\label{def:Symbols.classicalsymbols}
$S^q (\R^n; \cA_\theta )$, $q \in \C$, consists of maps $\rho(\xi)\in C^\infty(\R^n;\cA_\theta)$ that admit an asymptotic expansion, 
\begin{equation*}
\rho(\xi) \sim \sum_{j \geq 0} \rho_{q-j} (\xi),  \qquad \rho_{q-j} \in S_{q-j} (\R^n; \cA_\theta ). 
\end{equation*}
Here $\sim$ means that, for all integers $N$ and multi-orders $\alpha$, $\beta$, there exists $C_{N\alpha\beta} >0$ such that, for all $\xi \in \R^n$, $| \xi | \geq 1$, we have
\begin{equation} \label{eq:Symbols.classical-estimates}
\bigg\| \delta^\alpha \partial_\xi^\beta \biggl( \rho - \sum_{j<N} \rho_{q-j} \biggr)(\xi) \biggr\| \leq C_{N\alpha\beta} | \xi |^{\Re{q}-N-| \beta |} .
\end{equation}
\end{definition}

\begin{remark} \label{rem:Symbols.classical-uniqueness}
 The symbol $\rho_{q-j}(\xi)$ in~(\ref{eq:Symbols.classical-estimates}) is called the homogeneous symbol of degree $q-j$ of $\rho(\xi)$. The symbol $\rho_q(\xi)$ is called the \emph{principal symbol} of $\rho(\xi)$. 
 These homogeneous symbols are uniquely determined by~$\rho(\xi)$ since~(\ref{eq:Symbols.classical-estimates}) implies that, for all $\xi \in \R^n\setminus 0$, we have 
\begin{gather*}
 \rho_q(\xi) = \lim_{\lambda \rightarrow \infty} \lambda^{-q} \rho(\lambda \xi), \\
   \rho_{q-j}(\xi) = \lim_{\lambda \rightarrow \infty} \lambda^{-q+j}\biggl( \rho(\lambda \xi)- \sum_{\ell <j} \lambda^{q-\ell} \rho_{q-\ell}(\xi)\biggl), \qquad j\geq 1. 
\end{gather*}
\end{remark}

\begin{example}
 Every polynomial map $\rho(\xi)=\sum_{|\alpha|\leq m} a_\alpha \xi^\alpha$, $a_\alpha\in \cA_\theta$, is a classical symbol of order $m$. Its principal part is 
 $\rho_m(\xi):=\sum_{|\alpha|= m} a_\alpha \xi^\alpha$. 
\end{example}

\begin{example} \label{ex:Symbols.example-symbol}
 For $\xi\in \R^n$ set $\brak{\xi}=(1+|\xi|^2)^{\frac12}$. Given any $s\in \C$, the function $\brak{\xi}^s$ is a classical symbol of order~$s$. This can be seen by using the binomial expansion, 
\begin{equation*}
 \brak{\xi}^s =|\xi|^s\left(1+|\xi|^{-2}\right)^{\frac{s}2} = \sum_{j \geq 0} \binom{\frac{s}2}{j} |\xi|^{s-2j}, \qquad |\xi|>1. 
\end{equation*}
In particular, the principal symbol of $ \brak{\xi}^s$ is equal to $|\xi|^s$. 
\end{example}

\begin{remark} \label{rmk:Symbols.classical-inclusion}
Let $q\in \C$. Then we have an inclusion, 
\begin{equation*}
 S^q(\R^n;\cA_\theta)\subset \stS^{\Re{q}}(\R^n;\cA_\theta). 
\end{equation*}
In addition, if $\rho(\xi)\in S_q(\R^n;\cA_\theta)$ and $\chi(\xi)\in C^\infty_c(\R^n)$ is such that $\chi(\xi)=1$ near $\xi=0$, then $(1-\chi(\xi))\rho(\xi)\in S^q(\R^n;\cA_\theta)$ and $(1-\chi(\xi))\rho(\xi)\sim \rho(\xi)$ in the sense of~(\ref{eq:Symbols.classical-estimates}). In particular, we see that $(1-\chi(\xi))\rho(\xi)\in  \stS^{\Re{q}}(\R^n;\cA_\theta)$. 
\end{remark}

\begin{remark}\label{rmk:Symbols.derivatives-classical}
 Given $q\in\C$, let $\rho(\xi)\in S^q(\R^n;\cA_\theta)$, $\rho(\xi) \sim \sum \rho_{q-j}(\xi)$. Then~(\ref{eq:Symbols.classical-estimates}) and Remark~\ref{rem:Symbols.homogeneous-differentiation} imply that, for all multi-orders $\alpha$ and $\beta$, the partial derivative $\delta^\alpha\partial_\xi^\beta \rho(\xi)$ is a symbol in $S^{q-|\beta|}(\Rn;\cA_\theta)$ and $\delta^\alpha\partial_\xi^\beta \rho(\xi)\sim\sum \delta^\alpha\partial_\xi^\beta \rho_{q-j}(\xi)$.
\end{remark}

\begin{remark} \label{rem:Symbols.classical-involution}
 Given $q\in\C$, let $\rho(\xi)\in S^q(\R^n;\cA_\theta)$, $\rho(\xi) \sim \sum \rho_{q-j}(\xi)$. Then ~(\ref{eq:Symbols.classical-estimates}) and Remark~\ref{rem:Symbols.homogeneous-involution} imply that 
 $ \rho(\xi)^* \in S^{\overline{q}}(\Rn;\cA_\theta)$ and we have $ \rho(\xi)^*\sim\sum \rho_{q-j}(\xi)^*$.
\end{remark}

\begin{remark} \label{rem:Symbols.classical-standard-relation}
 Given $q\in\C$, for $j=0,1,\ldots$ let $\rho_{q-j}(\xi)\in S_{q-j}(\Rn;\cA_\theta)$. We also let 
$\chi(\xi) \in C_c^\infty(\Rn)$ be such that $\chi(\xi)=1$ near $\xi=0$. We know by Remark~\ref{rmk:Symbols.classical-inclusion} that 
$(1-\chi(\xi))\rho_{q-j}(\xi)$ is contained in $\stS^{\Re q-j}(\Rn;\cA_\theta)$ for all $j\geq 0$. Furthermore, given any $\rho(\xi)\in C^\infty(\Rn;\cA_\theta)$, the following are equivalent:
\begin{enumerate}
\item[(i)] $\rho(\xi)\sim\sum_{j\geq 0}\rho_{q-j}(\xi)$ in the sense of~(\ref{eq:Symbols.classical-estimates}).
\item[(ii)] $\rho(\xi)\sim\sum_{j\geq 0}(1-\chi(\xi))\rho_{q-j}(\xi)$ in the sense of Definition~\ref{def:Symbols.standard-asymptotic}. 
\end{enumerate}
In particular, if (ii) is satisfied, then $\rho(\xi)\in S^q(\Rn;\cA_\theta)$ and $\rho(\xi)\sim\sum_{j\geq 0} \rho_{q-j}(\xi)$. 
\end{remark}

Combining Lemma~\ref{lem:Symbols.standard-construction} with Remark~\ref{rem:Symbols.classical-standard-relation} we obtain the following result. 

\begin{proposition} \label{prop:Symbols.classical-construction}
Let $q\in\C$ and, for $j=0,1,\ldots$ let $\rho_{q-j}(\xi) \in S_{q-j}(\Rn;\cA_\theta)$. Then there exists a symbol $\rho(\xi)\in S^q (\Rn ; \cA_\theta)$ such that $\rho(\xi) \sim \sum_{j \geq 0} \rho_{q-j} (\xi)$. Moreover, such a symbol is unique modulo $\cS(\R^n;\cA_\theta)$. 
\end{proposition}

\begin{remark} \label{rem:Symbols.classical-homogeneouspart}
 Let $\rho(\xi) \in C^\infty(\Rn;\cA_\theta)$ be such that $\rho(\xi)\sim \sum_{\ell \geq 0}\rho^{(\ell)}(\xi)$, where $\rho^{(\ell)}(\xi)\in S^{q-\ell}(\Rn;\cA_\theta)$ and $\sim$ is taken in the sense of Definition~\ref{def:Symbols.standard-asymptotic}. Then $\rho(\xi)$ is a symbol in $S^{q}(\Rn;\cA_\theta)$, and we have $\rho(\xi) \sim \sum_{j \geq 0} \rho_{q-j}(\xi)$ in the sense of~(\ref{eq:Symbols.classical-estimates}), where 
\begin{equation*} 
\rho_{q-j}(\xi)= \sum_{\ell \leq j} \rho_{q-j}^{(\ell)}(\xi) , \qquad j \geq 0. 
\end{equation*}
Here $\rho_{q-j}^{(\ell)}(\xi)$ is the symbol of degree $q-j$ of $\rho^{(\ell)}(\xi)$.
\end{remark}

Finally, the following result shows that the product of $\cA_\theta$ gives rise to a (graded) bilinear map on classical symbols. 

\begin{proposition} \label{prop:Symbols.classical-product}
Let $\rho(\xi) \in S^{q}(\Rn;\cA_\theta)$ and $\sigma(\xi) \in S^{q'}(\Rn;\cA_\theta)$, $q,q'\in \C$. Then $\rho(\xi)\sigma(\xi)$ is in $S^{q+q'}(\Rn;\cA_\theta)$, and we have $\rho(\xi)\sigma(\xi)\sim\sum_{j\geq 0} (\rho\sigma)_{q+q'-j}(\xi)$, where
\begin{equation} \label{eq:Symbols.classical-product}
           (\rho\sigma)_{q+q'-j}(\xi)=  \sum_{p+r=j}\rho_{q-p}(\xi)\sigma_{q'-r}(\xi) , \qquad j\geq 0. 
\end{equation}
\end{proposition}\begin{proof}
 Let $\chi(\xi)\in C_c^\infty(\Rn)$ be such that $\chi(\xi)=1$ near $\xi=0$. Given any $N \in \N$, it follows from Remark~\ref{rem:Symbols.classical-standard-relation} that 
 \begin{align*}
\rho(\xi) &= \sum_{p<N}(1-\chi(\xi))\rho_{q-p}(\xi) \quad \bmod \stS^{\Re{q}-N}(\Rn;\cA_\theta), \\
 \sigma(\xi) &= \sum_{r<N}(1-\chi(\xi))\sigma_{q'-r}(\xi) \quad \bmod \stS^{\Re{q'}-N}(\Rn;\cA_\theta) .
\end{align*}
Combining this with Lemma~\ref{lem:Symbols.standard-product} we get 
\begin{equation*}
 \rho(\xi)\sigma(\xi)= \sum_{p,r<N}(1-\chi(\xi))^2\rho_{q-p}(\xi)\sigma_{q'-r}(\xi) \quad \bmod \ \stS^{\Re{(q+q')}-N}(\Rn;\cA_\theta). 
\end{equation*}
For $j=0,1,\ldots$ set
\begin{equation} \label{eq:Symbols.classical-product-j-term}
 (\rho\sigma)_{q+q'-j}(\xi)=  \sum_{p+r=j}\rho_{q-p}(\xi)\sigma_{q'-r}(\xi)\in S_{q+q'-j}(\Rn;\cA_\theta). 
\end{equation}
In addition, set $\widetilde{\chi}(\xi)=1-(1-\chi(\xi))^2$. Note that $\widetilde{\chi}(\xi)\in C_c^\infty(\Rn)$ and $\widetilde{\chi}(\xi)=1$ near $\xi=0$. 

We know by Remark~\ref{rmk:Symbols.classical-inclusion} that $(1-\chi(\xi))\rho_{q-p}(\xi) \in \stS^{\Re{q}-p}(\Rn;\cA_\theta)$ and $(1-\chi(\xi))\sigma_{q'-r}(\xi) \in \stS^{\Re{q'}-r}(\Rn;\cA_\theta)$. Therefore, using Lemma~\ref{lem:Symbols.standard-product} we see that $(1-\chi(\xi))^2\rho_{q-p}(\xi)\sigma_{q'-r}(\xi)$ is a symbol in $\stS^{\Re(q+q')-p-r}(\Rn;\cA_\theta)$, and hence it is contained in $\stS^{\Re(q+q')-N}(\Rn;\cA_\theta)$ when $p+r\geq N$. Combining this with~(\ref{eq:Symbols.classical-product-j-term}) we see that, for all $N\geq 1$, we have
\begin{align*}
  \rho(\xi)\sigma(\xi)& = \sum_{j<N}\sum_{p+r=j}(1-\chi(\xi))^2\rho_{q-p}(\xi)\sigma_{q'-r}(\xi) \quad \bmod \ \stS^{\Re{(q+q')}-N}(\Rn;\cA_\theta)\\ 
  & = \sum_{j<N}(1-\widetilde{\chi}(\xi))(\rho\sigma)_{q+q'-j}(\xi) \quad \bmod \ \stS^{\Re{(q+q')}-N}(\Rn;\cA_\theta).  
\end{align*}
This shows that $ \rho(\xi)\sigma(\xi)\sim \sum (1-\widetilde{\chi}(\xi))(\rho\sigma)_{q+q'-j}(\xi)$ in the sense of Definition~\ref{def:Symbols.standard-asymptotic}. Combining this with Remark~\ref{rem:Symbols.classical-standard-relation} gives the result.
\end{proof}

\section{Amplitudes and Oscillating Integrals} \label{section:Amplitudes}
In this section, we construct the oscillating integral for $\cA_\theta$-valued amplitudes.  We refer to Appendix~\ref{app:LCS-int} for background on the integration of maps with values in locally convex spaces. In Appendix~\ref{app:LCS-int} the extension of Lebesgue's integral is carried out for maps with values in quasi-complete Suslin locally convex spaces. The smooth noncommutative torus $\cA_\theta$ is such a space, since this is a separable Fr\'echet space. 

We refer to~\cite[Chapter 1]{Ri:MAMS93} for an alternative construction of the oscillating integral for zeroth order amplitudes with values in Fr\'echet spaces. The oscillating integral of this section is defined for $\cA_\theta$-valued amplitudes of any order.  

\subsection{Spaces of amplitudes} \label{subsection:Amplitudes.amplitudes} Let us first recall the definition of scalar-valued amplitudes. 

\begin{definition}[{\cite{AG:AMS07}}] 
$A^m ( \Rn \times \Rn )$, $m \in \R$, consists of functions $a(s,\xi)$ in $C^{\infty} (\Rn \times \Rn)$ such that, for all multi-orders $\beta$, $\gamma$, there is $C_{ \beta \gamma} > 0$ such that
\begin{equation*}
\left|\partial_s^\beta \partial_\xi^\gamma  a(s,\xi)\right| \leq C_{ \beta \gamma} \left( 1 + |s| + |\xi| \right)^m \quad \forall (s,\xi) \in \Rn \times \Rn .
\end{equation*}
\end{definition}

Throughout this paper we will make use of the following classes of $\cA_\theta$-valued amplitudes. 

\begin{definition}
$A^m ( \Rn \times \Rn ; \cA_\theta )$, $m \in \R$, consists of maps $a(s,\xi)$ in $C^{\infty} (\Rn \times \Rn ; \cA_\theta )$ such that, for all multi-orders $\alpha$, $\beta$, $\gamma$, there is $C_{\alpha \beta \gamma} > 0$ such that
\begin{equation} \label{eq:Amplitudes.amplitudes-estimates}
\norm{\delta^\alpha \partial_s^\beta \partial_\xi^\gamma  a(s,\xi)} \leq C_{\alpha \beta \gamma} \left( 1 + |s| + |\xi| \right)^m \quad \forall (s,\xi) \in \Rn \times \Rn .
\end{equation}
\end{definition}

\begin{remark}
 In the same way as in Remark~\ref{rmk:Symbols.amplitudes-intersection} we have
 \begin{equation*}
\bigcap_{m\in\R}A^m(\Rn\times\Rn;\cA_\theta)= \cS(\Rn\times\Rn;\cA_\theta).  
\end{equation*}
\end{remark}

We also define 
\begin{equation*} 
 A^{+\infty}(\Rn\times\Rn;\cA_\theta):=\bigcup_{m\in\R} A^m(\Rn\times\Rn;\cA_\theta) .
\end{equation*}

In what follows, we endow the space $A^m(\Rn\times\Rn;\cA_\theta)$, $m\in \R$, with the locally convex topology generated by the semi-norms,
\begin{equation} \label{eq:Amplitudes.amplitudes-semi-norms}
q_N^{(m)} (a) := \sup_{|\alpha|+|\beta|+|\gamma| \leq N}\sup_{(s,\xi) \in \Rn \times \Rn} (1 + |s| + |\xi| )^{-m} \norm{\delta^\alpha \partial_s^\beta \partial_\xi^\gamma a(s,\xi)}, \quad N\in\N_0 .
\end{equation}
In particular, the inclusion of $\C$ into the center of $\cA_\theta$ gives rise to a continuous embedding of $A^m(\Rn\times\Rn)$ into $A^m(\Rn\times\Rn;\cA_\theta)$. In addition, the natural inclusion of $A^m(\Rn\times\Rn;\cA_\theta)$ into $C^\infty (\Rn \times \Rn ; \cA_\theta )$ is continuous. Furthermore, by arguing along similar lines as that of the proof of Proposition~\ref{prop:Symbols.standard-Frechetspace} we obtain the following result. 

\begin{proposition} \label{prop:Amplitudes.amplitudes-Frechetspace} 
$A^m(\Rn\times\Rn;\cA_\theta)$, $m\in \R$,  is a Fr\'{e}chet space.
\end{proposition}

Any map $ \R^n \ni \xi \rightarrow \rho(\xi)\in \cA_\theta$ can be seen as a map $\R^n\times \R^n \ni (s,\xi)\rightarrow \rho(\xi)\in \cA_\theta$ that does not depend on the variable $s$. In particular, this allows us to regard $C^\infty(\R^n; \cA_\theta)$ as a subspace of $C^\infty(\R^n\times \R^n; \cA_\theta)$. Keeping in mind this identification, we have the following relationship between (standard) symbols and amplitudes. 

\begin{lemma} \label{lem:Amplitudes.symbol-inclusion}
 Let $m \in \R$, and set $m_+=\op{max}(m,0)$. Then we have a continuous inclusion, 
 \begin{equation*}
 \stS^m(\R^n; \cA_\theta)\subset A^{m_+}(\R^n\times \R^n; \cA_\theta). 
\end{equation*}
\end{lemma}\begin{proof}
 Let $\rho(\xi) \in \stS^m(\R^n; \cA_\theta)$. Given $N\geq 0$ and multi-orders $\beta$ and $\gamma$ with $|\beta|+|\gamma|\leq N$, we have  
\begin{equation*}
 \left\| \partial_\xi^\beta \delta^\gamma \rho(\xi)\right\|  \leq p^{(m)}_N(\rho) \left( 1+|\xi|\right)^{m-|\beta|} \leq p^{(m)}_N(\rho) \left( 1+|s|+|\xi|\right)^{m_+}. 
\end{equation*}
This shows that $\rho(\xi) \in A^{m_+}(\R^n\times \R^n; \cA_\theta)$, and we have the semi-norm estimates, 
\begin{equation*}
 q_N^{(m_+)} (\rho) \leq p^{(m)}_N(\rho) \qquad \text{for all $N\in \N_0$}. 
\end{equation*}
This shows that the inclusion of $ \stS^m(\R^n; \cA_\theta)$ into $A^{m_+}(\R^n\times \R^n; \cA_\theta)$ is continuous. The result is proved. 
\end{proof}

By arguing along similar lines as that of the proof of Lemma~\ref{lem:Symbols.standard-product} we also get the following result. 

\begin{lemma} \label{lem:Amplitudes.amplitudes-product}
Let  $m_1, m_2\in\R$. Then the product of $\cA_\theta$ gives rise to a continuous bilinear map from $A^{m_1}(\Rn\times\Rn;\cA_\theta)\times A^{m_2}(\Rn\times\Rn;\cA_\theta)$ to $A^{m_1+m_2}(\Rn\times\Rn;\cA_\theta)$. 
\end{lemma}

We also have the following version of Proposition~\ref{prop:Symbols.standard-density}. 

\begin{proposition} \label{prop:Amplitudes.amplitudes-density}
Let $a(s,\xi)\in A^m(\Rn\times\Rn;\cA_\theta)$, $m\in\R$, and let $\chi(s,\xi)\in C^\infty_c(\Rn\times\Rn)$ be such that $\chi(0,0)=1$.  For $0<\epsilon\leq 1$, define $a_\epsilon(s,\xi)\in C^\infty_c(\Rn\times\Rn;\cA_\theta)$ by 
\begin{equation*}
 a_\epsilon(s,\xi)= \chi(\epsilon s,\epsilon \xi) a(s,\xi), \qquad (s,\xi) \in \R^n\times \R^n. 
\end{equation*}
 Then the family $(a_\epsilon(s,\xi))_{0<\epsilon\leq 1}$ is contained in $C^\infty_c(\Rn\times\Rn; \cA_\theta)$, it is bounded in $A^m(\Rn\times\Rn;\cA_\theta)$ and, as $\epsilon \rightarrow 0^+$, it converges to $a(s,\xi)$  in $A^{m'}(\Rn\times\Rn;\cA_\theta)$ for every $m'>m$. 
\end{proposition}
\begin{proof}
 It is immediate that we have a continuous inclusion of $\stS^{m'}(\R^{2n};\cA_\theta)$ into $A^{m'}(\Rn\times\Rn; \cA_\theta)$ for every $m'\geq 0$. Using Lemma~\ref{lem:Symbols.approximate-unit} we then deduce that the family $(\chi(\epsilon s, \epsilon \xi))_{0<\epsilon\leq 1}$ is bounded in $A^0(\Rn\times\Rn; \cA_\theta)$ and, as $\epsilon \rightarrow 0^+$, it converges to $1$  in 
$A^{m'}(\Rn\times\Rn; \cA_\theta)$ for all $m'>0$. Combining this with Lemma~\ref{lem:Amplitudes.amplitudes-product} gives the result. 
\end{proof}

\begin{remark}
Similarly to Remark~\ref{rmk:Symbols.reduction-cS},  Proposition~\ref{prop:Amplitudes.amplitudes-density} allows us to reduce the proof of equalities for continuous functionals on amplitudes to checking them for maps in $C^\infty_c(\Rn\times\Rn; \cA_\theta)$.  
\end{remark}

Let $\varphi \in \cA_\theta'$. By Proposition~\ref{prop:LCS.smooth-Phi}  the composition with $\varphi$ gives rise to a linear map, 
\begin{equation} \label{eq:Amplitudes.amplitudes-evaluation}
 C^\infty(\R^n\times \Rn; \cA_\theta)\ni a(s,\xi) \longrightarrow \varphi\left[ a(s,\xi)\right] \in  C^\infty(\R^n\times \R^n). 
\end{equation}
Moreover, given any $ a(s,\xi)\in C^\infty(\R^n\times \Rn; \cA_\theta)$, for all multi-orders $\beta$ and $\gamma$, we have 
\begin{equation} \label{eq:Amplitudes.amplitudes-evaluation-differentiation}
 \partial_s^\beta \partial_\xi^\gamma \left( \varphi\left[ a(s,\xi)\right] \right)= \varphi\left[   \partial_s^\beta \partial_\xi^\gamma a(s,\xi)\right]. 
\end{equation}

\begin{lemma} \label{lem:Amplitudes.amplitudes-evaluation}
 Let $\varphi \in \cA_\theta'$. Then, for every $m\in \R$, the linear map~(\ref{eq:Amplitudes.amplitudes-evaluation}) induces a continuous linear map from $A^m(\R^n\times \R^n;\cA_\theta)$ to $A^m(\R^n\times \R^n)$. 
\end{lemma}
\begin{proof}
 As $\varphi$ is a continuous linear form on $\cA_\theta$ there are $N_0\in \N_0$ and $C>0$ such that
 \begin{equation*}
 \left| \varphi(u)\right| \leq C \sup_{|\alpha|\leq N_0}\left \| \delta^\alpha(u)\right\|  \qquad \text{for all $u\in \cA_\theta$}. 
\end{equation*}
Let $a(s,\xi)\in A^m(\R^n\times \R^n;\cA_\theta)$, $m\in \R$. Given any multi-orders $\beta$ and $\gamma$, by combining the above estimate with~(\ref{eq:Amplitudes.amplitudes-semi-norms}) and~(\ref{eq:Amplitudes.amplitudes-evaluation-differentiation}) 
 we see that, for all $s,\xi \in \R^n$, we have 
\begin{align} \label{eq:Amplitudes.phi-estimates}
 \left|   \partial_s^\beta \partial_\xi^\gamma \left( \varphi\left[ a(s,\xi)\right] \right)\right| & = \left|  \varphi\left[   \partial_s^\beta \partial_\xi^\gamma a(s,\xi)\right] \right|  \\
 & \leq C  \sup_{|\alpha|\leq N_0}\left \| \delta^\alpha\partial_s^\beta \partial_\xi^\gamma a(s,\xi)\right\|  \nonumber \\
 & \leq C q_{N_0+|\beta|+|\gamma|}^{(m)} (a) \left( 1+|s|+|\xi|\right)^m.  \nonumber
\end{align}
This shows that $ \varphi[ a(s,\xi)]\in A^m(\R^n\times \R^n)$. Moreover,  given any $N\in \N_0$, we have 
\begin{equation*}
 q_N^{(m)} \left( \varphi[a]\right) \leq C q_{N+N_0}^{(m)} (a) \qquad \text{for all $a\in A^m(\R^n\times \R^n;\cA_\theta)$}.  
\end{equation*}
Therefore, we see that the linear map~(\ref{eq:Amplitudes.amplitudes-evaluation}) induces a continuous linear map from $A^m(\R^n\times \R^n;\cA_\theta)$ to $A^m(\R^n\times \R^n)$. The proof is complete. 
\end{proof}

\subsection{$\cA_\theta$-Valued oscillating integrals}
Let $a(s,\xi)\in A^m (\Rn\times\Rn;\cA_\theta)$, $m<-2n$. The estimates~(\ref{eq:Amplitudes.amplitudes-estimates}) imply that, for every multi-order $\alpha$, we have 
\begin{equation} \label{eq:Amplitudes.amplitudes-estimates2}
\iint \left\| \delta^\alpha\left(a(s,\xi)\right)\right\| dsd\xi  \leq \iint  \left(1+|s|+|\xi|\right)^m dsd\xi <\infty. 
\end{equation}
As the semi-norms $u\rightarrow \| \delta^\alpha(u)\|$, 
$\alpha \in \N_0^n$, generate the topology of the separable Fr\'echet space $\cA_\theta$, this shows that the map 
 $\R^n\times \R^n \ni(s,\xi)\rightarrow e^{is\cdot \xi}a(s,\xi)\in \cA_\theta$ is integrable in the sense of Definition~\ref{def:LCS.integrability} in Appendix~\ref{app:LCS-int}. 
 Therefore, we may define the $\cA_\theta$-valued integral, 
\begin{equation} \label{eq:Amplitudes.definition-J0}
J_0(a) := \iint e^{is\cdot\xi} a(s,\xi) ds\dbar\xi , 
\end{equation}
where we have set $\dbar\xi=(2\pi)^{-n}d\xi$.  More precisely, this is the unique element of $\cA_\theta$ such that
\begin{equation}
  \varphi \left( \iint e^{is\cdot \xi} a(s,\xi)ds\dbar \xi \right) =  \iint e^{is\cdot \xi} \varphi[a(s,\xi)]ds\dbar \xi \qquad \text{for all $\varphi \in \cA_\theta'$}. 
  \label{eq:Amplitudes.J_0-vphi}
\end{equation}
We obtain a linear map $J_0:A^m (\Rn\times\Rn;\cA_\theta)\rightarrow \cA_\theta$. Moreover, by combining~(\ref{eq:Amplitudes.amplitudes-estimates2}) with  Proposition~\ref{prop:LCS.Lebesgue-integral} we see that, for all $\alpha \in \N_0^n$, we get 
\begin{equation*}
 \left\| \delta^\alpha J_0(a)\right\|  \leq \iint \left\| \delta^\alpha\left(a(s,\xi)\right) \right\|  ds\dbar \xi \leq C(m) q_{|\alpha|}^{(m)}(a), 
\end{equation*}
where we have set $C(m)=  \iint (1+|s|+|\xi|)^m ds\dbar\xi$. Note that $C(m)<\infty$, since $m<-2n$. Therefore, we arrive at the following statement. 

\begin{lemma} \label{lem:Amplitudes.J0-continuity}
 The linear map $J_0:A^m (\Rn\times\Rn;\cA_\theta)\rightarrow \cA_\theta$ given by~(\ref{eq:Amplitudes.definition-J0}) is continuous for every $m<-2n$. 
\end{lemma}

\begin{remark} \label{rem:Amplitudes.J0-phi-compatibility}
 Let $a(s,\xi)\in A^m (\Rn\times\Rn;\cA_\theta)$, $m<-2n$. We know by Lemma~\ref{lem:Amplitudes.amplitudes-evaluation} that, for every $\varphi \in \cA_\theta'$, the function $\varphi[a(s,\xi)]$ is contained in $A^m (\Rn\times\Rn)$. Then~(\ref{eq:Amplitudes.J_0-vphi}) means that we have
 \begin{equation*}
 \varphi \left[J_0(a)\right] = J_0\left[ \varphi\circ a \right] \qquad  \text{for all $\varphi \in \cA_\theta'$}. 
\end{equation*} 
\end{remark}

In what follows, given any differential operator $P$ on $\R^n\times \R^n$, we denote by $P^t$ its transpose. This is the differential operator on $\R^n\times \R^n$ such that 
\begin{equation} \label{eq:Amplitudes.operator-transpose}
 \iint Pu(s,\xi) v(s,\xi) ds d\xi  =  \iint u(s,\xi) P^tv(s,\xi) ds d\xi \qquad \text{for all $u,v\in C^\infty_c(\R^n\times \R^n)$}. 
\end{equation}
In fact, if we set $P= \sum c_{\beta\gamma}(s,\xi) \partial_s^\beta \partial_\xi^\gamma$, $c_{\beta\gamma}(s,\xi) \in C^\infty(\R^n\times \R^n)$, then we have
\begin{equation*}
 P^tu (s,\xi) = \sum (-1)^{|\beta|+|\gamma|}  \partial_s^\beta \partial_\xi^\gamma \left( c_{\beta\gamma}(s,\xi)u(s,\xi)\right) \qquad  
 \text{for all $u\in C^\infty(\R^n\times \R^n)$}. 
\end{equation*}

\begin{lemma} \label{lem:Amplitudes.J0-P-transpose}
 Let $P$ be a differential operator on $\R^n\times \R^n$. Then, we have
\begin{equation*}
 J_0\left( Pa\right) = J_0\left[ e^{-is\cdot \xi}P^t(e^{is\cdot \xi})a\right] \qquad \forall a \in C^\infty_c(\R^n\times \R^n;\cA_\theta). 
\end{equation*}
\end{lemma}
\begin{proof}
 Let $a(s,\xi)\in C^\infty_c(\R^n\times \R^n; \cA_\theta)$ and $\varphi \in \cA_\theta'$. Then $\varphi\left[a(s,\xi)\right] \in C^\infty_c(\R^n\times \R^n)$ and it follows from~(\ref{eq:Amplitudes.amplitudes-evaluation-differentiation}) that $ \varphi\left( P\left[a(s,\xi)\right]\right)= P\left(\varphi\left[a(s,\xi)\right]\right)$. 
 Combining this with~(\ref{eq:Amplitudes.J_0-vphi}) we get 
 \begin{equation*}
 \varphi\left(  \iint e^{is\cdot \xi} P\left[ a(s,\xi)\right] ds\dbar\xi\right) = \iint \varphi\left(  e^{is\cdot \xi} P\left[ a(s,\xi)\right] \right) ds\dbar\xi = \iint  
 e^{is\cdot \xi} P\left( \varphi\left[a(s,\xi)\right] \right) ds\dbar\xi . 
\end{equation*}
By using~(\ref{eq:Amplitudes.J_0-vphi}) and~(\ref{eq:Amplitudes.operator-transpose}) we also obtain 
\begin{equation*}
 \iint  e^{is\cdot \xi} P\left( \varphi\left[a(s,\xi)\right] \right) ds\dbar\xi =  \iint  P^t\left(e^{is\cdot \xi}\right)  \varphi\left[a(s,\xi)\right]  ds\dbar\xi 
 = \varphi\left(  \iint  P^t\left(e^{is\cdot \xi}\right)  a(s,\xi)  ds\dbar\xi \right). 
\end{equation*}
As $\cA_\theta'$ separates the points of $\cA_\theta$, it then follows that 
\begin{equation*}
 \iint e^{is\cdot \xi} P\left[ a(s,\xi)\right] ds\dbar\xi = \iint  P^t\left(e^{is\cdot \xi}\right)  a(s,\xi)  ds\dbar\xi. 
\end{equation*}
This shows  that $J_0(Pa)= J_0[e^{-is\cdot \xi}P^t(e^{is\cdot \xi})a]$. The proof is complete. 
\end{proof}

We shall now explain how to extend the linear map $J_0$ to the whole class $A^{+\infty}(\Rn\times\Rn;\cA_\theta)$. To reach this end let $\chi(s,\xi) \in C_c^{\infty} (\Rn \times \Rn)$ be such that $\chi(s,\xi)= 1$ near $(s,\xi)=(0,0)$, and set 
\begin{equation*}
L:= \chi(s,\xi) + \frac{1-\chi(s,\xi)}{|s|^2 + |\xi|^2} \sum_{1 \leq j \leq n} (\xi_j D_{s_j} + s_j D_{\xi_j}) , 
\end{equation*}
where we have set $D_{x_j}=\frac{1}{i}\partial_{x_j}$, $j=1, \ldots, n$. We note that
\begin{equation} \label{eq:Amplitudes.L-exponential}
L(e^{is\cdot\xi}) = e^{is\cdot\xi}.
\end{equation}
We also denote by $L^t$ the transpose of $L$. 

\begin{lemma} \label{lem:Amplitudes.L-transpose-continuity} Let $m \in \R$. Then the differential operator $L^t$ gives rise to a continuous linear map,
\begin{equation*}
L^t : A^m (\Rn \times \Rn ; \cA_\theta) \longrightarrow A^{m-1} (\Rn \times \Rn ; \cA_\theta) .
\end{equation*}
\end{lemma}
\begin{proof}
We observe that $L^t$ is of the form,
\begin{equation} \label{eq:Amplitudes.L-transpose}
L^t = -\frac{1-\chi (s,\xi)}{|s|^2 +|\xi|^2} \sum_{1 \leq j \leq n} (\xi_j D_{s_j} + s_j D_{\xi_j} ) - \frac{4is\cdot\xi(1-\chi(s,\xi))}{(|s|^2+|\xi|^2)^2} + \widetilde{\chi}(s,\xi) ,
\end{equation}
where $\widetilde{\chi}(s,\xi) \in C^\infty_c(\R^n\times \R^n)$. As $\widetilde{\chi}(s,\xi)$ is a Schwartz-class function, it follows from Lemma~\ref{lem:Amplitudes.amplitudes-product} that the 
multiplication by $\widetilde{\chi}(s,\xi)$ gives rise to a continuous linear map from $A^m ( \Rn \times \Rn ; \cA_\theta )$ to $A^{m'} ( \Rn \times \Rn ; \cA_\theta )$ for all $m,m'\in\R$. In addition, given any $m \in \R$, we observe that
\begin{itemize}
\item For $j=1,\ldots,n$ the differential operators $D_{s_j}$ and $D_{\xi_j}$ give rise to continuous linear maps from $A^m(\Rn\times\Rn;\cA_\theta)$ to $A^m (\Rn \times \Rn ; \cA_\theta)$.
\item For $j=1,\ldots,n$ the multiplications by $s_j$ and $\xi_j$ give rise to continuous linear maps from $A^m (\Rn\times\Rn;\cA_\theta)$ to $A^{m+1}(\Rn\times\Rn;\cA_\theta)$.
\item For $\ell=1,2$ the multiplication by $(1-\chi(s,\xi))(|s|^2 +|\xi|^2)^{-\ell}$ induces a continuous linear map from $A^m(\Rn\times\Rn;\cA_\theta)$ to $A^{m-2\ell}(\Rn\times\Rn;\cA_\theta)$.
\end{itemize}
The first property is an immediate consequence of the definition of $A^m ( \Rn \times \Rn ; \cA_\theta )$ and its topology. The other two properties follow from Lemma~\ref{lem:Amplitudes.amplitudes-product}. Using these three properties we deduce that
\begin{itemize}
\item The differential operator $(1-\chi(s,\xi))(|s|^2 + |\xi|^2)^{-1} \sum (\xi_j D_{s_j} + s_j D_{\xi_j})$ induces a continuous linear map from $A^m (\Rn \times \Rn ; \cA_\theta)$ to $A^{m-1} (\Rn \times \Rn ; \cA_\theta)$.
\item The multiplication by $s\cdot\xi(1-\chi(s,\xi))(|s|^2+|\xi|^2)^{-2}$ gives rise to a continuous linear map from $A^m(\Rn\times\Rn;\cA_\theta)$ to $A^{m-2}(\Rn\times\Rn;\cA_\theta)$. 
\end{itemize}
Combining all this with~(\ref{eq:Amplitudes.L-transpose}) shows that the transpose $L^t$ induces a continuous linear map from $A^m(\Rn\times\Rn;\cA_\theta)$ to $A^{m-1}(\Rn\times\Rn;\cA_\theta)$. The proof is complete.
\end{proof}

We are now in a position to extend the linear map~(\ref{eq:Amplitudes.definition-J0}) to amplitudes of any order. 

\begin{proposition} \label{prop:Amplitudes.extension-J0}
The linear map~(\ref{eq:Amplitudes.definition-J0}) has a unique extension to a linear map $J:A^{+\infty}(\Rn\times\Rn;\cA_\theta)\rightarrow \cA_\theta$ that is continuous on each space $A^m(\Rn\times\Rn;\cA_\theta)$, $m\in\R$. More precisely, for every $a\in A^m(\Rn\times\Rn;\cA_\theta)$, $m\in\R$, we have
\begin{equation*} 
J(a)=\iint e^{is\cdot\xi}(L^t)^N [a(s,\xi)]ds\dbar\xi ,
\end{equation*}
where $N$ is any non-negative integer $>m+2n$.  
\end{proposition}
\begin{proof}
The proof is based on the following claim. 

\begin{claim*}
 Let $a\in A^{m} (\Rn\times\Rn ;\cA_\theta)$, $m<-2n$. Then  
\begin{equation} \label{eq:Amplitudes.J0-L-transpose}
 J_0(a)= J_0\left(L^t[a]\right). 
\end{equation}
\end{claim*}
\begin{proof}[Proof of the Claim]
 Let $m'\in (m,-2n)$. It follows from Lemma~\ref{lem:Amplitudes.J0-continuity} and Lemma~\ref{lem:Amplitudes.L-transpose-continuity} that both sides of~(\ref{eq:Amplitudes.J0-L-transpose}) define continuous linear maps from $A^{m'}(\R^n\times \R^n; \cA_\theta)$ to $\cA_\theta$. Combining this with Proposition~\ref{prop:Amplitudes.amplitudes-density} we then deduce that it is enough to prove~(\ref{eq:Amplitudes.J0-L-transpose}) when $a(s,\xi)\in C^\infty_c(\R^n\times \R^n; \cA_\theta)$. Now, if $a(s,\xi)\in C^\infty_c(\R^n\times \R^n; \cA_\theta)$, then by using~(\ref{eq:Amplitudes.L-exponential}) and Lemma~\ref{lem:Amplitudes.J0-P-transpose} we get
\begin{equation*}
 J_0\left(L^t[a]\right)= J_0\left( e^{-is\cdot \xi}L(e^{is\cdot \xi})a\right)= J_0(a). 
\end{equation*}
This completes the proof of the claim. 
\end{proof}

 Let $a\in A^{m} (\Rn\times\Rn ;\cA_\theta)$, $m<-2n$. Repeated use of~(\ref{eq:Amplitudes.J0-L-transpose}) shows that
\begin{equation} \label{eq:Amplitudes.J0-L-transpose-N}
J_0(a) = J_0\left((L^t)^N [a]\right) \qquad \text{for all $N \geq 0$}.
\end{equation}
It follows from Lemma~\ref{lem:Amplitudes.L-transpose-continuity} that $(L^t)^N$ induces a continuous linear map from $A^{m}(\Rn\times\Rn ;\cA_\theta)$ to $A^{m-N}(\Rn\times\Rn ;\cA_\theta)$. Therefore, the right-hand side of~(\ref{eq:Amplitudes.J0-L-transpose-N}) actually makes sense for any amplitude in  $A^{m}(\Rn\times\Rn ;\cA_\theta)$ with $m<-2n+N$. 

Given $m\in \R$, let $N$ and $N'$ be non-negative integers such that $N'>N>m+2n$. Then $(L^t)^N [a(s,\xi)]$ and $(L^t)^{N'} [a(s,\xi)]$ are both amplitudes of order~$<-2n$. Moreover, using~(\ref{eq:Amplitudes.J0-L-transpose-N}) we see that 
\begin{equation*}
 J_0\left((L^t)^{N'} [a]\right)= J_0\left((L^t)^{N'-N}\left[(L^t)^N [a]\right]\right)= J_0\left((L^t)^N [a]\right) . 
\end{equation*}
Therefore, we see that the value of $J_0\left((L^t)^{N} [a]\right)$ is independent of the choice of the non-negative integer~$N>m+2n$. 

All this allows us to define a linear map $J:A^{+\infty}(\Rn\times\Rn;\cA_\theta)\rightarrow\cA_\theta$ such that, for every $a(s,\xi)\in A^{m}(\Rn\times\Rn;\cA_\theta)$, $m\in \R$, we have 
\begin{equation*}
 J(a)= J_0\left((L^t)^N [a]\right) = \iint e^{is\cdot\xi}(L^t)^N [a(s,\xi)]ds\dbar\xi ,
\end{equation*}
where $N$ is any non-negative integer $>m+2n$, the value of which is irrelevant. In particular, when $m<-2n$ we may take $N=0$; this allows us to recover the linear map $J_0$. 

Let $m\in \R$ and $N\in \N_0$ be such that $N>m+2n$. Then $J=J_0 \circ (L^t)^N$ on $A^{m}(\Rn\times\Rn;\cA_\theta)$. As mentioned above, $(L^t)^N$ maps continuously $A^{m}(\Rn\times\Rn ;\cA_\theta)$ to $A^{m-N}(\Rn\times\Rn ;\cA_\theta)$. As $m-N<-2n$ we also know by Lemma~\ref{lem:Amplitudes.J0-continuity} that $J_0$ is a continuous linear map from $A^{m-N}(\Rn\times\Rn ;\cA_\theta)$ to $\cA_\theta$. Thus, $J$ induces a continuous linear map from $A^{m}(\Rn\times\Rn;\cA_\theta)$ to $\cA_\theta$ for every $m\in \R$.

To complete the proof it remains to show that $J$ is the only linear extension of $J_0$ to $A^{+\infty}(\Rn\times\Rn;\cA_\theta)$ that is continuous on each space $A^{m}(\Rn\times\Rn;\cA_\theta)$, $m\in \R$. Let $\tilde{J}: A^{+\infty}(\Rn\times\Rn;\cA_\theta)\rightarrow \cA_\theta$ be another such extension. Let $a(s,\xi)\in A^{m}(\Rn\times\Rn;\cA_\theta)$, $m\in \R$. We know by Proposition~\ref{prop:Amplitudes.amplitudes-density} that there is a family $(a_\epsilon(s,\xi))_{0<\epsilon\leq 1}\subset C_c^{\infty}(\Rn\times\Rn;\cA_\theta)$ such that $a_\epsilon(s,\xi)\rightarrow a(s,\xi)$ as $\epsilon \rightarrow 0^+$ in $A^{m'}(\Rn\times\Rn;\cA_\theta)$ for every $m'>m$. The continuity of $J$ on $A^{m'}(\Rn\times\Rn;\cA_\theta)$ then implies that
\begin{equation*}
 J(a)= \lim_{\epsilon\rightarrow 0^+}J(a_\epsilon)= \lim_{\epsilon\rightarrow 0^+}J_0(a_\epsilon) .
\end{equation*}
Likewise, we have $\tilde{J}(a)= \lim_{\epsilon\rightarrow 0^+}J_0(a_\epsilon)=J(a)$, and so the linear maps $\tilde{J}$ and $J$ agree on 
$A^{+\infty}(\Rn\times\Rn;\cA_\theta)$. This shows that  $J$ is the unique linear extension of $J_0$ to $A^{+\infty}(\Rn\times\Rn;\cA_\theta)$ that is continuous on each space $A^{m}(\Rn\times\Rn;\cA_\theta)$, $m\in \R$. The proof is complete. 
\end{proof}

\begin{remark}
 A proof of Proposition~\ref{prop:Amplitudes.extension-J0} for scalar-valued amplitudes is given in~\cite{AG:AMS07} by using a dyadic partition of unity. 
\end{remark}

As with~(\ref{eq:Amplitudes.amplitudes-evaluation}), any continuous $\R$-linear map  $\Phi:\cA_\theta \rightarrow \cA_\theta$ gives rise to an $\R$-linear map, 
\begin{equation} \label{eq:Amplitudes.amplitudes-composition-Phi}
 C^\infty(\Rn\times\Rn;\cA_\theta)\ni a(s,\xi) \longrightarrow \Phi\left[a(s,\xi)\right] \in C^\infty(\Rn\times\Rn;\cA_\theta). 
\end{equation}
Moreover, in the same way as in~(\ref{eq:Amplitudes.amplitudes-evaluation-differentiation}) this map commutes with partial differentiations with respect to $s$ and $\xi$ of any order.

\begin{lemma} \label{lem:Amplitudes.J-Phi-compatibility}
 Let $\Phi:\cA_\theta \rightarrow \cA_\theta$ be a continuous $\R$-linear map.
 \begin{enumerate}
 \item[(i)] The linear map~(\ref{eq:Amplitudes.amplitudes-composition-Phi}) induces a continuous $\R$-linear map from $A^m(\Rn\times\Rn;\cA_\theta)$ to itself for every $m\in \R$. 
 
 \item[(ii)] If $\Phi$ is $\C$-linear, then, for  all $a(s,\xi)\in A^{+\infty}(\Rn\times\Rn;\cA_\theta)$, we have 
\begin{equation} \label{eq:Amplitudes.J-Phi-compatibility}
 J\left( \Phi(a)\right) = \Phi\left(J(a)\right).
\end{equation}

\item[(iii)] If $\Phi$ is anti-linear, then, for  all $a(s,\xi)\in A^{+\infty}(\Rn\times\Rn;\cA_\theta)$, we have 
\begin{equation} \label{eq:Amplitudes.J-antiPhi-compatibility}
 J\left( \Phi(a)\right) = \Phi\left[J(a\left(-s,\xi)\right)\right].
\end{equation}
\end{enumerate}
\end{lemma}
\begin{proof}
Let $N\in \N_0$. As $\Phi:\cA_\theta \rightarrow \cA_\theta$ is a continuous $\R$-linear map,  there are $N'\in \N_0$ and $C_{NN'}>0$ such that
\begin{equation*}
 \sup_{|\alpha|\leq N}\left\|\delta^\alpha\left(\Phi(u)\right)\right\|  \leq C_{NN'} \sup_{|\alpha|\leq N'}\left\|\delta^\alpha(u)\right\|  \qquad \text{for all $u\in \cA_\theta$}. 
\end{equation*}
Let $\alpha$, $\beta$ and $\gamma$ be multi-orders such that $|\alpha|+|\beta|+|\gamma|\leq N$. In addition, let $a(s,\xi)$ be an amplitude in $A^m(\Rn\times\Rn;\cA_\theta)$, $m\in \R$. Then, in the same way as in~(\ref{eq:Amplitudes.phi-estimates}), it can be shown that, for all $s,\xi\in \R^n$, we have 
\begin{equation*}
 \left\| \delta^\alpha \partial_s^\beta \partial_\xi^\gamma \Phi\left(a(s,\xi)\right)\right\|  \leq C_{NN'} q_{N+N'}^{(m)}(a) (1+|s|+|\xi|)^m. 
\end{equation*}
This shows that $\Phi(a(s,\xi))\in A^m(\Rn\times\Rn;\cA_\theta)$. Moreover, we have the semi-norm estimate,
\begin{equation*}
 q_N^{(m)}\left[ \Phi(a)\right]  \leq C_{NN'} q_{N+N'}^{(m)}(a) \qquad \text{for all $a\in A^m(\Rn\times\Rn;\cA_\theta)$}. 
\end{equation*}
Therefore, the linear map~(\ref{eq:Amplitudes.amplitudes-composition-Phi}) induces a continuous $\R$-linear endomorphism on $A^m(\Rn\times\Rn;\cA_\theta)$ for every $m\in \R$. 

Let us prove~(ii). Suppose that $\Phi$ is $\C$-linear. Using (i) and Proposition~\ref{prop:Amplitudes.extension-J0} we see that both sides of~(\ref{eq:Amplitudes.J-Phi-compatibility}) define continuous $\C$-linear maps on $A^{m}(\Rn\times\Rn;\cA_\theta)$ for every $m\in \R$. Combining this with Proposition~\ref{prop:Amplitudes.amplitudes-density} we deduce that it is enough to prove~(\ref{eq:Amplitudes.J-Phi-compatibility}) when $a(s,\xi)$ is in $C_c^{\infty}(\Rn\times\Rn;\cA_\theta)$. Now, if $a(s,\xi)$ is in $C_c^{\infty}(\Rn\times\Rn;\cA_\theta)$, then by using Proposition~\ref{prop:LCS.Phi-integral}  and the $\C$-linearity of $\Phi$ we get
\begin{equation*}
\Phi\left[ J(a)\right] =\Phi\left( \iint e^{is\cdot\xi} a(s,\xi) ds \dbar\xi \right) =  \iint e^{is\cdot\xi} \Phi\left[a(s,\xi)\right] ds \dbar\xi = J\left( \Phi(a)\right). 
\end{equation*}
This gives~(\ref{eq:Amplitudes.J-Phi-compatibility}) when $a(s,\xi)$ is in $C_c^{\infty}(\Rn\times\Rn;\cA_\theta)$. The proof of (ii) is complete.

It remains to prove~(iii). Assume that $\Phi$ is anti-linear. We observe that $a(s,\xi) \rightarrow a(-s,\xi)$ is a continuous linear map from $A^m(\Rn\times\Rn;\cA_\theta)$ to itself for every $m\in \R$. Therefore, by using~(i) we also see that both sides of~(\ref{eq:Amplitudes.J-antiPhi-compatibility}) define continuous anti-linear maps on $A^{m}(\Rn\times\Rn;\cA_\theta)$ for every $m\in \R$. In the same way as with the proof of~(\ref{eq:Amplitudes.J-Phi-compatibility}) above, combining this with Proposition~\ref{prop:Amplitudes.amplitudes-density} allows us to reduce the proof of~(\ref{eq:Amplitudes.J-antiPhi-compatibility}) to the case where $a(s,\xi)\in C_c^{\infty}(\Rn\times\Rn;\cA_\theta)$. 

Let $a(s,\xi)\in C_c^{\infty}(\Rn\times\Rn;\cA_\theta)$. In the same way as above, by using the anti-linearity of $\Phi$ and Proposition~\ref{prop:LCS.Phi-integral} we obtain
\begin{equation*}
\Phi\left[ J(a)\right] =\Phi\left( \iint e^{is\cdot\xi} a(s,\xi) ds \dbar\xi \right) =  \iint e^{-is\cdot\xi} \Phi\left[a(s,\xi)\right] ds \dbar\xi. 
\end{equation*}
Thanks to Proposition~\ref{prop:LCS.change-variable} we can make the change of variable $s \rightarrow -s$ to get
\begin{equation*}
 \Phi\left[ J(a)\right] = \iint e^{is\cdot\xi} \Phi\left[a(-s,\xi)\right] ds \dbar\xi = J\left(\Phi\left[a(-s,\xi)\right]\right). 
\end{equation*}
This establishes~(\ref{eq:Amplitudes.J-antiPhi-compatibility}) when $a(s,\xi)\in C_c^{\infty}(\Rn\times\Rn;\cA_\theta)$. This proves (iii) and completes the proof. 
\end{proof}

We gather the main properties of the linear map $J$ in the following statement. 

\begin{proposition} \label{prop:Amplitudes.J-properties}
Let $a(s,\xi)\in A^m(\Rn\times\Rn;\cA_\theta)$, $m\in\R$. The following holds.
\begin{enumerate}
\item[(i)] For all $b_1, b_2\in\cA_\theta$, we have 
\begin{equation*} 
J(b_1a b_2)=b_1J(a)b_2. 
\end{equation*}

\item[(ii)] Set $a^*(s,\xi)= a(-s,\xi)^*$, $s,\xi\in \R^n$. Then $a^*(s,\xi)\in A^m(\Rn\times\Rn;\cA_\theta)$, and we have
\begin{equation*} 
J(a)^* = J(a^*) .
\end{equation*}

\item[(iii)]  For every multi-order $\alpha$, we have
\begin{equation*} 
\label{eq:Amplitudes.J-delta-compatibility}
\delta^\alpha J(a) = J(\delta^\alpha a) .
\end{equation*}

\item[(iv)] For all multi-orders $\alpha$, $\beta$, we have
\begin{equation} \label{eq:Amplitudes.J-Delta-property}
J\left(D_s^\alpha D_\xi^\beta a\right) =  (-1)^{|\alpha|+|\beta|} J\left(s^\beta \xi^\alpha a\right).  
\end{equation}
\end{enumerate}
\end{proposition}
\begin{proof}
The properties (i)--(iii) are immediate consequences of Lemma~\ref{lem:Amplitudes.J-Phi-compatibility}. Therefore, we only have to prove (iv). Let $\alpha$ and $\beta$ be multi-orders. As mentioned in the proof of Lemma~\ref{lem:Amplitudes.L-transpose-continuity}, for $j=1, \ldots, n$ and every $m\in \R$, multiplication by $s_j$ and $\xi_j$ give rise to continuous linear maps from $A^m(\Rn\times\Rn;\cA_\theta)$ to   $A^{m+1}(\Rn\times\Rn;\cA_\theta)$, and differentiation with respect to $s_j$ and $\xi_j$ give rise to continuous linear maps from $A^m(\Rn\times\Rn;\cA_\theta)$ to itself. Therefore, both sides of~(\ref{eq:Amplitudes.J-Delta-property}) defines continuous linear maps on $A^m(\Rn\times\Rn;\cA_\theta)$ for every $m\in \R$. As in the proof of Lemma~\ref{lem:Amplitudes.J-Phi-compatibility} this reduces the proof of~(\ref{eq:Amplitudes.J-Delta-property}) to the case where $a(s,\xi)\in C^{\infty}_c(\Rn\times\Rn;\cA_\theta)$.

Let $a(s,\xi)\in C^{\infty}_c(\Rn\times\Rn;\cA_\theta)$. Note that if we set $P=D_s^\alpha D_\xi^\beta$, then we have 
\begin{equation*}
 P^t\left( e^{i s\cdot \xi} \right) = (-1)^{|\alpha|+|\beta|} D_s^\alpha D_\xi^\beta\left( e^{i s\cdot \xi} \right)= (-1)^{|\alpha|+|\beta|} s^\beta \xi^\alpha e^{i s\cdot \xi} . 
\end{equation*}
Combining this with Lemma~\ref{lem:Amplitudes.J0-P-transpose} we obtain
\begin{equation*}
 J\left(D_s^\alpha D_\xi^\beta a\right) = J_0\left( P[a]\right) = J_0\left( e^{-is\cdot \xi}P^t(e^{is\cdot \xi})a\right) = (-1)^{|\alpha|+|\beta|} J\left(s^\beta \xi^\alpha a\right).
\end{equation*}
This proves~(\ref{eq:Amplitudes.J-Delta-property}) when $a(s,\xi)\in C^{\infty}_c(\Rn\times\Rn;\cA_\theta)$. The proof is complete. 
\end{proof}

We conclude this section with the following result on oscillating integrals associated with families of amplitudes. 

\begin{proposition} \label{prop:Amplitudes.J-partial-compatibility-family}
Suppose $U$ is an open subset of $\R^d$, $d\geq 1$. Given $m\in\R$, let $a(x;s,\xi)\in C^\infty (U\times\Rn\times\Rn ;\cA_\theta)$ be such that, for all compact sets $K\subset U$ and for all multi-orders $\alpha\in\N_0^d$ and $ \beta, \gamma, \lambda \in \N_0^n$, there is $C_{K\alpha\beta\gamma\lambda}>0$ such that, for all $(x,s,\xi)\in K\times\Rn\times\Rn $, we have
\begin{equation} \label{eq:Amplitudes.amplitudes-family-estimates}
\norm{\partial_x^\alpha \partial_s^\beta \partial_\xi^\gamma \delta^\lambda a(x;s,\xi)}\leq C_{K\alpha\beta\gamma\lambda} (1+|s|+|\xi|)^m .
\end{equation}
Then $x\rightarrow J(a(x;\cdot,\cdot))$ is a smooth map from $U$ to $\cA_\theta$, and, for every multi-order $\alpha$, we have 
\begin{equation*}
\partial_x^\alpha J\left(a(x;\cdot,\cdot)\right) = J\left[(\partial_x^\alpha a)(x;\cdot,\cdot)\right]  \qquad  \forall x\in U .
\end{equation*}
\end{proposition}
\begin{proof}
The smoothness of $a(x;s,\xi)$ and the estimates~(\ref{eq:Amplitudes.amplitudes-family-estimates}) ensures us that $\partial_x^\alpha a(x; \cdot, \cdot)$ is an element of 
$A^m (\Rn\times\Rn;\cA_\theta)$ for all $x\in U$ and $\alpha \in \N_0^d$. 

\begin{claim*}
 The map $x \rightarrow a(x; \cdot, \cdot)$ is a smooth map from $U$ to $A^m (\Rn\times\Rn;\cA_\theta)$. Moreover, for every multi-order $\alpha \in \N_0^d$, we have 
 \begin{equation}
 \partial_x^\alpha \left[ a(x; \cdot, \cdot)\right] = ( \partial_x^\alpha a)(x; \cdot, \cdot) \qquad \forall x\in U. 
 \label{eq:Amplitudes.derivatives-family}
\end{equation}
\end{claim*}
\begin{proof}[Proof of the Claim]
 Let $x\in U$ and $\delta>0$ be such that $\overline{B}(x,\delta)\subset U$. Let $ \beta, \gamma, \lambda \in \N_0^n$ and $s,\xi \in \R^n$. By Lemma~\ref{lem:LCS.C1-differentiable} for $|h|<\delta$ we have 
 \begin{equation}
\partial_s^\beta \partial_\xi^\gamma \delta^\lambda a(x+h;s,\xi) - \partial_s^\beta \partial_\xi^\gamma \delta^\lambda a(x;s,\xi) = 
\sum_{1\leq j \leq d} h_j \int_0^1 \partial_{x_j} \partial_s^\beta \partial_\xi^\gamma \delta^\lambda a(x+th;s,\xi)dt.  
\label{eq:Amplitudes.rate-change} 
\end{equation}
Thus,
  \begin{align*}
\norm{\partial_s^\beta \partial_\xi^\gamma \delta^\lambda a(x+h;s,\xi) - \partial_s^\beta \partial_\xi^\gamma \delta^\lambda a(x;s,\xi)} & \leq 
\sum_{1\leq j \leq d} |h_j| \int_0^1 \norm{\partial_{x_j} \partial_s^\beta \partial_\xi^\gamma \delta^\lambda a(x+th;s,\xi)}dt\\
& \leq \sqrt{d} |h| \sup_{1\leq j \leq d} \sup_{|y-x|\leq \delta} \norm{\partial_{x_j} \partial_s^\beta \partial_\xi^\gamma \delta^\lambda a(y;s,\xi)}.  
\end{align*}
Combining this with~(\ref{eq:Amplitudes.amplitudes-family-estimates}) we see that, for all  $N\in \N_0$, there is a constant $C_N>0$ such that, 
\begin{equation*}
 q_N^{(m)}\left[ a(x+h;\cdot, \cdot)- a(x;\cdot, \cdot)\right] \leq C_N |h| \qquad \text{for all $h \in \overline{B}(0,\delta)$}.    
\end{equation*}
It then follows that $ a(x+h;\cdot, \cdot) \rightarrow a(x;\cdot, \cdot)$ in $A^m (\Rn\times\Rn;\cA_\theta)$ as $h\rightarrow 0$. This shows that $x \rightarrow a(x; \cdot, \cdot)$ is a continuous map from $U$ to $A^m (\Rn\times\Rn;\cA_\theta)$. It can be similarly shown that we have a continuous map $U\ni x \rightarrow \partial_x^\alpha a(x; \cdot, \cdot)\in A^m (\Rn\times\Rn;\cA_\theta)$ for every $\alpha \in \N_0^d$. 
 
Let  $x\in U$ and $\delta>0$ be such that $\overline{B}(x,\delta)\subset U$. In addition, let $(e_1, \ldots, e_d)$ be the canonical basis of $\R^d$. Then~(\ref{eq:Amplitudes.rate-change}) implies that, for $0<|t|\leq \delta$ and $j=1,\ldots, d$, we have 
\begin{equation}
 \frac{1}{t} \left[ a(x+te_j;\cdot,\cdot) - a(x;\cdot,\cdot) \right] = \int_0^1  \partial_{x_j} a(x+ste_j;\cdot,\cdot)ds \qquad \text{in $A^m (\Rn\times\Rn;\cA_\theta)$}.
 \label{eq:Amplitudes.rate-change-ej}
\end{equation}
As mentioned above  $y\rightarrow  \partial_{x_j} a(y;\cdot,\cdot)$ is a continuous map from $U$ to  $A^m (\Rn\times\Rn;\cA_\theta)$, and so it is uniformly continuous on the compact set $\overline{B}(x,\delta)$. Therefore, by arguing as in the proof of Lemma~\ref{lem:LCS.convergence-C1}  it can be shown that as $t\rightarrow 0$ the right-hand side of~(\ref{eq:Amplitudes.rate-change-ej}) converges to $ \partial_{x_j} a(x; \cdot, \cdot)$ in $A^m (\Rn\times\Rn;\cA_\theta)$. It then follows that $x\rightarrow  a(x;\cdot,\cdot)$ is a $C^1$-map from $U$ to  $A^m (\Rn\times\Rn;\cA_\theta)$. An induction further shows that, for every integer $N\geq 1$, this map is $C^N$ and its partial derivatives of order~$\leq N$ are given by~(\ref{eq:Amplitudes.derivatives-family}). Incidentally, this is a smooth map. The claim is thus proved. 
 \end{proof}

We know by Proposition~\ref{prop:Amplitudes.extension-J0} that $J$ is a continuous linear map from $A^m (\Rn\times\Rn;\cA_\theta)$ to $\cA_\theta$. Therefore, by combining the above claim with Proposition~\ref{prop:LCS.smooth-Phi} 
shows that $x\rightarrow J(a(x;~\cdot,\cdot))$ is a smooth map from $U$ to $\cA_\theta$ and, for every $\alpha \in \N_0^d$ and all $x\in U$, we have 
\begin{equation*}
 \partial_x^\alpha J\left(a(x;\cdot,\cdot)\right) = J\left(\partial_x^\alpha \left[a(x;\cdot,\cdot)\right]\right) = J\left[(\partial_x^\alpha a)(x;\cdot,\cdot)\right]. 
\end{equation*}
The proof is complete. 
\end{proof}

\section{Pseudodifferential Operators on Noncommutative Tori}\label{sec:PsiDOs}
In this section, we give a precise definition of the pseudodifferential operators (\psidos) on noncommutative tori associated with symbols and amplitudes. We shall also derive a few properties of these operators, including a characterization of smoothing operators. 

\subsection{$\mathbf{\Psi}$DOs associated with amplitudes} 
The results of the previous section allow us to give sense to the integral appearing at the end of Section~\ref{section:NCtori}. More generally, we can give sense to integrals of the form, 
\begin{equation*}
 \iint e^{is\cdot \xi} a(s,\xi) \alpha_{-s}(u) ds \dbar \xi,
\end{equation*}
where $a(s,\xi)\in A^m(\Rn\times\Rn;\cA_\theta)$, $m\in\R$, and $u\in \cA_\theta$. Namely, such an integral is the oscillating integral associated with $a(s,\xi) \alpha_{-s}(u)$. Thus, we only have to justify that $a(s,\xi) \alpha_{-s}(u)$ is an amplitude. In fact, we have the following result. 

\begin{lemma} \label{lem:PsiDOs.amplitudes-alpha-product-continuity}
 Let $m\in \R$. Then the map $(a,u)\rightarrow a(s,\xi) \alpha_{-s}(u)$ is a continuous bilinear map from $A^m(\Rn\times\Rn;\cA_\theta)\times \cA_\theta$ to  $A^m(\Rn\times\Rn;\cA_\theta)$. 
\end{lemma}
\begin{proof}
 Given $u \in \cA_\theta$, let $\tilde{u}(s,\xi)\in C^\infty(\R^n\times \R^n;\cA_\theta)$ be defined by 
\begin{equation*}
 \tilde{u}(s,\xi)= \alpha_{-s}(u), \qquad s,\xi \in \R^n. 
\end{equation*}
 Let $N\in \N_0$ and $\alpha, \gamma \in \N_0^n$ be such that $|\alpha|+|\gamma|\leq N$. Then we have  
\begin{equation*}
\left\| \partial^\alpha_s \delta^\gamma \tilde{u}(s,\xi)\right \|= 
 \left\| \partial^\alpha_s \delta^\gamma\left( \alpha_{-s}(u)\right) \right\|  = 
  \left\| \partial^{\alpha+\gamma}_s \left( \alpha_{-s}( u)\right) \right\|  = 
\left \|\alpha_{-s}\left( \delta^{\alpha+\gamma}(u)\right)\right\|   =\left\| \delta^{\alpha+\gamma}(u)\right\| . 
\end{equation*}
 This shows that $\tilde{u}(s,\xi)\in A^0(\Rn\times\Rn;\cA_\theta)$, and we have the semi-norm equalities, 
\begin{equation*}
 q_N^{(0)} (\tilde{u})= \sup_{|\alpha|\leq N} \|\delta^\alpha(u)\| \qquad \text{for all $u \in \cA_\theta$}. 
\end{equation*}
 Therefore, we see that $u \rightarrow \tilde{u}(s,\xi)$ is a continuous linear map from $\cA_\theta$ to $A^0(\Rn\times\Rn;\cA_\theta)$. Combining this with Lemma~\ref{lem:Amplitudes.amplitudes-product} gives the result. 
 \end{proof}
 
 Given any amplitude $a(s,\xi)\in A^m(\Rn\times\Rn;\cA_\theta)$, $m\in\R$, the above lemma allows us to define a linear operator $P_a:\cA_\theta \rightarrow \cA_\theta$ by
\begin{align*}
P_a u &= J\left(a(s,\xi)\alpha_{-s} (u)\right) \\\nonumber
&= \iint e^{is\cdot\xi} a(s,\xi) \alpha_{-s} (u) ds\dbar\xi, \qquad u\in\cA_\theta .
\end{align*}
 Thanks to the continuity contents of Proposition~\ref{prop:Amplitudes.extension-J0} and Lemma~\ref{lem:PsiDOs.amplitudes-alpha-product-continuity} we have the following result.

\begin{proposition} \label{prop:PsiDOs.pdos-continuity}
Let $m\in\R$. Then the map $(a,u)\rightarrow P_a u$ is a continuous bilinear map from $A^m(\Rn\times\Rn;\cA_\theta)\times\cA_\theta$ to $\cA_\theta$. In particular, 
for every $a(s,\xi)\in A^m(\Rn\times\Rn;\cA_\theta)$, the linear operator $P_a:\cA_\theta \rightarrow \cA_\theta$ is continuous. 
\end{proposition}

In what follows we denote by $\cL(\cA_\theta)$ the algebra of continuous linear maps $T:\cA_\theta \rightarrow \cA_\theta$. We equip it with its strong dual topology (a.k.a.\ uniform bounded convergence topology). This is the locally convex topology generated  by the semi-norms,
\begin{equation*}
T\longrightarrow\sup_{u\in B}\norm{\delta^\alpha Tu} , \qquad \alpha \in \N_0^n, \quad \text{$B\subset \cA_\theta$ bounded}. 
\end{equation*}
 
\begin{corollary} \label{cor:PsiDOs.amplitudes-to-pdos-continuity}
 Let $m \in \R$. Then the map $a(s,\xi)\rightarrow P_a$ is a continuous linear map from $A^m(\Rn\times\Rn;\cA_\theta)$ to $\cL(\cA_\theta)$. 
\end{corollary}
 \begin{proof}
 This is a standard semi-norm estimate argument. Let $\alpha \in \N_0^n$. It follows from Proposition~\ref{prop:PsiDOs.pdos-continuity} that there are non-negative integers $N$, $N'$ and a constant $C_{\alpha NN'}$ such that, for all $a(s,\xi)\in A^m(\Rn\times\Rn;\cA_\theta)$ and $u\in \cA_\theta$, we have 
 \begin{equation*} 
\left\| \delta^\alpha ( P_a u) \right\|  \leq C_{\alpha NN'} q_{N}^{(m)}(a)\sup_{|\beta|\leq N'}\norm{\delta^\beta u}.
\end{equation*}
Thus, for every bounded set $B\subset \cA_\theta$, we have
\begin{equation*}
 \sup_{u\in B}\left\| \delta^\alpha ( P_a u) \right\|  \leq C_{\alpha NN'} C_{N'}(B) q_{N}^{(m)}(a)  \qquad \text{for all $a(s,\xi)\in A^m(\Rn\times\Rn;\cA_\theta)$}, 
\end{equation*}
where we have set $C_{N'}(B)= \sup_{|\alpha|\leq N'} \sup_{u\in B}\norm{\delta^\alpha u}$. This proves the result. 
\end{proof}

\subsection{$\mathbf{\Psi}$DOs associated with symbols} 
As mentioned above, any symbol $\rho(\xi)\in \stS^m(\Rn;\cA_\theta)$, $m\in\R$, can be regarded as an amplitude in $A^{m_+}(\Rn\times\Rn;\cA_\theta)$. We thus can define a continuous linear operator $P_\rho:\cA_\theta \rightarrow \cA_\theta$ as in~(\ref{eq:NCtori.diff-op-integral}). We thus obtain the formula given in~\cite{Co:CRAS80}, 
\begin{equation*}
 P_\rho u= \iint e^{is\cdot\xi} \rho (\xi) \alpha_{-s} (u) ds\dbar\xi, \qquad u\in\cA_\theta ,
\end{equation*}
where the integral is meant as an oscillating integral, i.e., this is $J[\rho(\xi) \alpha_{-s} (u)]$. 

Combining Lemma~\ref{lem:Amplitudes.symbol-inclusion} with Proposition~\ref{prop:PsiDOs.pdos-continuity} and Corollary~\ref{cor:PsiDOs.amplitudes-to-pdos-continuity} proves the following result.

\begin{proposition} \label{prop:PsiDOs.symbols-to-pdos-continuity}
Let $m\in\R$. The following holds. 
\begin{enumerate}
\item The map $(\rho,u)\rightarrow P_\rho u$ is a continuous bilinear map from $\stS^m(\Rn;\cA_\theta)\times\cA_\theta$ to $\cA_\theta$. 

\item The map $\rho\rightarrow P_\rho$ is  a continuous linear map from $\stS^m(\Rn;\cA_\theta)$ to $\cL(\cA_\theta)$.
\end{enumerate}
\end{proposition}

\begin{proposition}[\cite{Ba:CRAS88}]
 Let $\rho(\xi)\in \stS^m(\R^n; \cA_\theta)$, $m\in \R$. For $j=1,\ldots, n$, we have 
 \begin{equation*}
 [\delta_j, P_\rho]= P_{\delta_j \rho}. 
\end{equation*}
\end{proposition}
\begin{proof}
 Let $u\in \cA_\theta$. By using Proposition~\ref{prop:Amplitudes.J-properties} we see that, for $j=1,\ldots, n$, we have 
\begin{equation*}
 \delta_j P_\rho u = \delta_j J\left(\rho(\xi)\alpha_{-s}(u)\right)= J\left( \delta_j\left[\rho(\xi)\alpha_{-s}(u)\right]\right). 
\end{equation*}
As $ \delta_j[\rho(\xi)\alpha_{-s}(u)]=[\delta_j\rho(\xi)]\alpha_{-s}(u) + \rho(\xi)\alpha_{-s}(\delta_ju)$, we get 
\begin{equation*}
 \delta_j P_\rho u = J\left(\delta_j\left[\rho(\xi)\right]\alpha_{-s}(u)\right) + J\left( \rho(\xi)\alpha_{-s}(\delta_ju)\right)= P_{\delta_j \rho}u+P_\rho (\delta_j u). 
\end{equation*}
This shows that $[\delta_j, P_\rho]= P_{\delta_j \rho}$. The result is proved. 
\end{proof}

\begin{definition}
 $\Psi^q(\cA_\theta)$,  $q\in \C$, consists of all linear operators $P:\cA_\theta\rightarrow \cA_\theta$ that are of the form $P=P_\rho$ for some symbol $\rho(\xi)\in S^q(\R^n; \cA_\theta)$. 
\end{definition}

\begin{remark}
Given $P\in \Psi^q(\cA_\theta)$ there is not a unique symbol $\rho(\xi)\in S^q(\R^n; \cA_\theta)$ such that $P=P_\rho$. However, the symbol is unique up to the addition of an element of $\cS(\R^n;\cA_\theta)$ (\emph{cf}.\ Corollary~\ref{cor:PsiDOs.uniqueness-symbol} \emph{infra}). As a result, the homogeneous symbols $\rho_{q-j}(\xi)\in S_{q-j}(\R^n;\cA_\theta)$, $j=0,1,\ldots$, are uniquely determined by $P$. Therefore, it makes sense to call $\rho_{q-j}(\xi)$ the \emph{symbol of degree $q-j$} of $P$. In particular, we shall call $\rho_q(\xi)$ the \emph{principal symbol} of $P$.
\end{remark}

\begin{lemma}[\cite{Co:CRAS80}]\label{lem:PsiDO.Prho-cS}
 Let $\rho(\xi)\in \cS(\R^n; \cA_\theta)$. Then 
 \begin{equation*}
 P_{\rho}u = \int \check{\rho}(s)\alpha_{-s}(u) ds \qquad \text{for all $u\in \cA_\theta$}, 
\end{equation*}
where $\check{\rho}(s)$ is the inverse Fourier transform of $\rho(\xi)$ (\cf\ Appendix~\ref{app:LCS-diff}). 
\end{lemma}
\begin{proof}
Let $\chi(s)\in C_c^\infty(\R^n)$ be such that $\chi(0)=1$. For $0\leq \epsilon \leq 1$ set $\chi_\epsilon(s)=\chi(\epsilon s)$, $s\in \R^n$. For every $m>0$, the family $(\chi_\epsilon(s))_{0\leq \epsilon \leq 1}$ converges to $1$ in $A^{m}(\R^n\times \R^n)$ as $\epsilon \rightarrow 0^+$ (\cf~\cite[Proposition~18.1.2]{Ho:Springer85}). Moreover, it follows from Lemma~\ref{lem:PsiDOs.amplitudes-alpha-product-continuity} that $\rho(\xi)\alpha_{-s}(u)\in A^0(\R^n\times \R^n; \cA_\theta)$. Therefore, by using Lemma~\ref{lem:Amplitudes.amplitudes-product} we see that, given any $m>0$, the family 
$(\chi_\epsilon(s) \rho(\xi)\alpha_{-s}(u))_{0\leq \epsilon \leq 1}$ converges to $\rho(\xi)\alpha_{-s}(u)$ in $A^{m}(\R^n\times \R^n; \cA_\theta)$  as $\epsilon \rightarrow 0^{+}$. Combining this with continuity of $J$ on $A^{m}(\R^n\times \R^n; \cA_\theta)$ then gives
\begin{equation}
 P_\rho u=J\left[\rho(\xi)\alpha_{-s}(u)\right] = \lim_{\epsilon \rightarrow 0^{+}} J\left[\chi_\epsilon(s)  \rho(\xi)\alpha_{-s}(u)\right]. 
 \label{eq:PsiDO.Prho-lim}
 \end{equation}

Let $\epsilon \in (0,1]$. Note that  $\rho(\xi) \in \cS(\R^n;\cA_\theta)$ and $\chi_{\epsilon}(s) \alpha_{-s}(u) \in C_c^\infty(\R^n; \cA_\theta)$. Therefore, in a similar way as in the proof of Lemma~\ref{lem:Symbols.standard-product} it can be shown that $\rho(\xi) \chi_{\epsilon}(s) \alpha_{-s}(u)\in \cS(\R^n\times \R^n; \cA_\theta)$. Combining this with the version of Fubini's theorem provided by Proposition~\ref{prop:LCS.Fubini}  then gives 
\begin{align*}
 J\left[\chi_\epsilon(s)  \rho(\xi)\alpha_{-s}(u)\right] & = \iint e^{is\cdot \xi} \chi_\epsilon(s)  \rho(\xi)\alpha_{-s}(u)ds\dbar \xi\\ 
 &  = \int  \chi_\epsilon(s) \left( \int e^{is\cdot \xi} \rho(\xi)\alpha_{-s}(u)\dbar \xi\right) ds. 
\end{align*}
Given any $s\in \R^n$, the right-multiplication by $\alpha_{-s}(u)$ induces a continuous endomorphism on $\cA_\theta$, and so by Proposition~\ref{prop:LCS.Phi-integral} it commutes with the integration of integrable $\cA_\theta$-maps. Thus, 
 \begin{equation*}
 \int e^{is\cdot \xi} \rho(\xi)\alpha_{-s}(u)\dbar \xi = \left( \int e^{is\cdot \xi} \rho(\xi)\dbar \xi\right) \alpha_{-s}(u) = \check{\rho}(s) \alpha_{-s}(u). 
\end{equation*}
where $\check{\rho}(s)$ is the inverse Fourier transform of $\rho(\xi)$. Note that Proposition~\ref{prop:LCS.Fubini} implies that $\check{\rho}(s) \alpha_{-s}(u)\in L^1(\R^n; \cA_\theta)$. In fact, as $\check{\rho}(s)\in \cS(\R^n; \cA_\theta)$ (\cf\ Appendix~\ref{app:LCS-diff}), in a similar way as in the proof of Lemma~\ref{lem:Symbols.standard-product} it can be shown that  
$\check{\rho}(s) \alpha_{-s}(u)\in \cS(\R^n; \cA_\theta)$. In any case, it follows from all this that we have 
\begin{equation}
 J\left[\chi_\epsilon(s)  \rho(\xi)\alpha_{-s}(u)\right] =  \int  \chi_\epsilon(s) \check{\rho}(s) \alpha_{-s}(u)ds. 
 \label{eq:PsiDO.Jepsi-int}
\end{equation}

As $\chi(0)=1$, for every $s\in \R^n$, the family $\chi_\epsilon(s) \check{\rho}(s)\alpha_{-s}(u)$ converges to $\check{\rho}(s)\alpha_{-s}(u)$ in $\cA_\theta$ as $\epsilon \rightarrow 0^+$. Moreover, given any multi-order $\alpha$, for all $s\in \R^n$ and $\epsilon \in [0,1]$, we have 
\begin{equation*}
 \norm{\delta^\alpha [\chi_\epsilon(s) \check{\rho}(s)\alpha_{-s}(u)]} =|\chi_\epsilon(s)| \norm{\delta^\alpha [\check{\rho}(s)\alpha_{-s}(u)]} \leq C \norm{\delta^\alpha [\check{\rho}(s)\alpha_{-s}(u)]}, 
\end{equation*}
where we have set $C= \sup_{s \in \R^n}|\chi(s)|$. Note that $\norm{\delta^\alpha [\check{\rho}(s)\alpha_{-s}(u)]} \in L^1(\R^n)$, since $
\check{\rho}(s) \alpha_{-s}(u)$ is in $L^1(\R^n; \cA_\theta)$. Therefore, we may use the version of the dominated convergence theorem provided by Proposition~\ref{prop:LCS.DCT} to get 
\begin{equation*}
 \lim_{\epsilon \rightarrow 0^+}  \int  \chi_\epsilon(s) \check{\rho}(s) \alpha_{-s}(u)ds = \int  \check{\rho}(s) \alpha_{-s}(u)ds. 
\end{equation*}
Combining this with~(\ref{eq:PsiDO.Prho-lim}) and~(\ref{eq:PsiDO.Jepsi-int}) shows that $P_{\rho}u = \int \check{\rho}(s)\alpha_{-s}(u) ds$. The proof is complete. 
 \end{proof}

\begin{proposition}[\cite{CT:Baltimore11}] \label{prop:PsiDOs.Prhou-equation}
Let $\rho(\xi)\in\stS^m(\Rn;\cA_\theta)$, $m\in\R$. Then, for every $u=\sum_{k\in\Z^n} u_k U^k\in\cA_\theta$, we have
\begin{equation*}
P_\rho u = \sum_{k\in\Z^n} u_k \rho(k)U^k .
\end{equation*}
\end{proposition}
\begin{proof} Let $u=\sum_{k\in\Z^n} u_k U^k\in\cA_\theta$. As we have convergence in $\cA_\theta$ and $P_\rho$ is a continuous endomorphism on $\cA_\theta$, we have 
\begin{equation*}
P_\rho u = P\biggl( \sum_{k\in\Z^n} u_k U^k \biggr) =\sum_{k\in\Z^n} u_k P_\rho (U^k) .
\end{equation*}
Therefore, it is enough to prove that 
\begin{equation} \label{eq:PsiDOs.Prhou-basis}
P_\rho(U^k) = \rho(k)U^k  \qquad \text{for all $k\in \Z^n$}. 
\end{equation}
The right-hand side~is a continuous linear map on $\stS^{m'}(\R^n;\cA_\theta)$ for all $m'>m$. The same is true for the left-hand side~thanks to Proposition~\ref{prop:PsiDOs.symbols-to-pdos-continuity}. Combining this with Proposition~\ref{prop:Symbols.standard-density} we then deduce it is sufficient to prove~(\ref{eq:PsiDOs.Prhou-basis}) when $\rho(\xi)\in \cS(\R^n; \cA_\theta)$. 

Suppose that $\rho(\xi)\in\cS(\R^n; \cA_\theta)$.  Then by Lemma~\ref{lem:PsiDO.Prho-cS} we have 
\begin{equation*}
 P_\rho (U^k)=  \int  \check{\rho}(s) \alpha_{-s}\left( U^k\right)ds=   \int  e^{-is\cdot k} \check{\rho}(s) U^k ds. 
\end{equation*}
In the same way as in the proof of Lemma~\ref{lem:PsiDO.Prho-cS}, the right-multiplication by $U^k$ commutes with the integration of integrable $\cA_\theta$-valued maps, and so we have 
\begin{equation*}
P_\rho (U^k)=   \int  e^{-is\cdot k} \check{\rho}(s) U^k ds =  \left( \int  e^{-is\cdot k} \check{\rho}(s)  ds\right) U^k = \left( \check{\rho}\right)^\wedge\!\! (k) U^k, 
\end{equation*}
where $( \check{\rho})^\wedge(\xi)$ is the Fourier transform of $\check{\rho}(s)$ (\cf\ Appendix~\ref{app:LCS-diff}). As by the Fourier inversion formulas provided by Proposition~\ref{prop:LCS.Fourier-inversion} we have $( \check{\rho})^\wedge(k)=\rho(k)$, we see that $P_\rho (U^k)=\rho(k)U^k$. This proves~(\ref{eq:PsiDOs.Prhou-basis}) when $\rho(\xi)\in \cS(\R^n; \cA_\theta)$. The proof is complete. 
\end{proof}

\begin{corollary}\label{cor:PsiDOs.P-Prho-relation}
 Let $\rho_j(\xi)\in \stS^m(\R^n;\cA_\theta)$, $j=1,2$. Then $P_{\rho_1}=P_{\rho_2}$ if and only if $\rho_1(k)=\rho_2(k)$ for all $k\in \Z^n$. 
\end{corollary}
\begin{proof}
 If $\rho_1$ and $\rho_2$ agree on $\Z^n$, then it follows from Proposition~\ref{prop:PsiDOs.Prhou-equation} that $P_{\rho_1}=P_{\rho_2}$ on all $\cA_\theta$. Conversely, suppose that $P_{\rho_1}=P_{\rho_2}$. Then Proposition~\ref{prop:PsiDOs.Prhou-equation} ensures us that, for all $k\in \Z^n$, we have 
 $\rho_1(k)U^k= P_{\rho_1}(U^k)= P_{\rho_2}(U^k)= \rho_2(k)U^k$, and hence $\rho_1=\rho_2$ on $\Z^n$. The proof is complete. 
\end{proof}

We are now in a position to justify the formula~(\ref{eq:NCtori.diff-op-integral}).  Namely, we have the following result. 

\begin{corollary}
 Let $P=\sum_{|\alpha|\leq m} a_\alpha \delta^\alpha$, $a_\alpha\in \cA_\theta$, be a differential operator of order~$m$. Then  $P=P_\rho$, where
 $\rho(\xi)=\sum_{|\alpha|\leq m} a_\alpha \xi^\alpha$ is the symbol of $P$. In particular, $P\in \Psi^m(\cA_\theta)$. 
\end{corollary}
\begin{proof}
As $P$ and $P_\rho$ are both continuous endomorphisms of $\cA_\theta$, it is enough to show that 
\begin{equation} \label{eq:PsiDOs.equality-check-basis}
 P(U^k)=P_\rho(U^k) \qquad \text{for all $k\in \Z^n$}. 
\end{equation}
Let $k \in \Z^n$. We know by Proposition~\ref{prop:PsiDOs.Prhou-equation} that $P_\rho(U^k)=\rho(k)U^k$. Moreover, we have 
\begin{equation*}
 P(U^k)= \sum_{|\alpha|\leq m} a_\alpha \delta^\alpha(U^k)= \sum_{|\alpha|\leq m} a_\alpha k^\alpha U^k =\rho(k)U^k. 
\end{equation*}
Therefore, we see that $ P(U^k)=P_\rho(U^k)$. This proves~(\ref{eq:PsiDOs.equality-check-basis}). As $\rho(\xi)\in S^m(\R^n; \cA_\theta)$ (\emph{cf}.~Remark~\ref{rem:NCtori.differential-op}), it then follows that $P\in \Psi^m(\cA_\theta)$. 
The proof is complete. 
\end{proof}

As it can be inferred from Corollary~\ref{cor:PsiDOs.P-Prho-relation} the symbol of \psido\ is by no means unique. Nevertheless, we have the following result. 

\begin{proposition}\label{prop:PsiDOs.vanishing-Prho}
 Let $\rho(\xi)\in \stS^m(\R^n;\cA_\theta)$, $m\in \R$, be such that $P_\rho=0$. Then $\rho(\xi)\in \cS(\R^n;\cA_\theta)$.  
\end{proposition}
\begin{proof}
If $P_\rho=0$, then it follows from Corollary~\ref{cor:PsiDOs.P-Prho-relation} that $\rho(k)=0$ for all $k\in \Z^n$. It will be shown later 
(see Corollary~\ref{cor:toroidal.rhok0-cS}) that this implies that $\rho(\xi)\in \cS(\R^n; \cA_\theta)$. Thus, modulo proving Corollary~\ref{cor:toroidal.rhok0-cS}, we obtain the result.   
\end{proof}

\begin{corollary}\label{cor:PsiDOs.uniqueness-symbol}
Let $P\in \Psi^q(\cA_\theta)$, $q\in \C$. Then
\begin{enumerate}
 \item[(i)] The symbol of $P$ is unique modulo $\cS(\R^n;\cA_\theta)$. 
 
 \item[(ii)] The homogeneous components of the symbol of $P$ are uniquely determined by $P$. 
\end{enumerate}
\end{corollary}
\begin{proof}
 Let $\rho_1(\xi)$ and $\rho_2(\xi)$ be symbols in $S^q(\R^n;\cA_\theta)$ such that $P_{\rho_1}=P_{\rho_2}=P$. Then $P_{\rho_1-\rho_2}=0$, and so by Proposition~\ref{prop:PsiDOs.vanishing-Prho} the symbols $\rho_1(\xi)$ and $\rho_2(\xi)$ differ by an element of $\cS(\R^n; \cA_\theta)$. In particular, they have the same homogeneous components. This proves the result.
\end{proof}

Let us now give some examples of pseudodifferential operators that are not differential operators. Let $\Delta= \delta_1^2 + \cdots + \delta_n^2$ be the flat Laplacian of $\cA_\theta$.  This operator is isospectral to the ordinary Laplacian on the usual torus 
$\mathbb{T}^n=\R^n\slash 2\pi \Z^n$. More precisely, the family $(U^k)_{k\in \Z^n}$ forms an orthonormal eigenbasis of $\cH_\theta$ such that
\begin{equation*}
 \Delta \left(U^k\right)= |k |^2 U^k \qquad \text{for all $k \in \Z^n$}. 
\end{equation*}
 We have a positive selfadjoint operator on $\cH_\theta$ with domain, 
 \begin{equation*}
 \op{Dom}(\Delta)=\biggl\{ u=\sum_{k\in \Z^n} u_kU^k\in \cH_\theta; \ \sum_{k\in \Z^n} |k|^4|u_k|^2<\infty\biggr\}. 
\end{equation*}
For any $s\in \C$ we denote by $\Lambda^s$ the operator $(1+\Delta)^{\frac{s}2}$. Thus, 
\begin{equation} \label{eq:PsiDOs.Lambda-s-equation}
 \Lambda^s \left(U^k\right)=\left(1+ |k |^2\right)^{\frac{s}{2}} U^k \qquad \text{for all $k \in \Z^n$}.
\end{equation}
For $\Re s\leq 0$ we obtain a bounded operator. For $\Re s>0$ we get a closed operator with domain, 
 \begin{equation*} 
\op{Dom}(\Lambda^s)= \biggl\{ u=\sum_{k\in \Z^n} u_kU^k\in \cH_\theta; \ \sum_{k\in \Z^n} |k|^{2 \Re s} |u_k|^2<\infty\biggr\}. 
\end{equation*}
In particular, the domain of $\Lambda^s$ always contains $\cA_\theta$. We obtain a selfadjoint operator when $s\in \R$. We also have the property, 
\begin{equation*}
\Lambda^{s_1+s_2}=\Lambda^{s_1}\Lambda^{s_2}, \qquad s_1, s_2\in \C.
\end{equation*}
In addition, when $\Re s<0$ we have $(1+ |k |^2)^{\frac{s}2} \rightarrow 0$ as $|k|\rightarrow \infty$, and so $\Lambda^s$ is a compact operator on $\cH_\theta$. 

Recall  that $\brak{\xi}^s:= (1+ |\xi|^2)^{\frac{s}2}$ is a classical symbol of order $s$ (\emph{cf}.~Example~\ref{ex:Symbols.example-symbol}).  Moreover, given any $u = \sum u_kU^k$ in $\cA_\theta$, using~(\ref{eq:PsiDOs.Lambda-s-equation}) and Proposition~\ref{prop:PsiDOs.Prhou-equation} we get
\begin{equation*}
 \Lambda^s u = \sum_{k\in \Z^n} u_k \brak{k}^{\frac{s}2} U^k =P_{\brak{\xi}^s}u. 
\end{equation*}
Therefore, we arrive at the following statement. 

\begin{proposition} \label{prop:PsiDOs.Lambdas-properties}
 The family $(\Lambda^s)_{s\in \C}$ is a 1-parameter group of classical \psidos. For every $s\in \C$, the operator $\Lambda^s$ is the \psido\ associated with the symbol $\brak{\xi}^s$. In particular, $\Lambda^s\in \Psi^{s}(\cA_\theta)$ and the principal symbol of $\Lambda^s$ is $|\xi|^s$. 
\end{proposition}

\section{Toroidal \psidos\ and Smoothing Operators}\label{sec:toroidal} 
In this section, we clarify the relationship between standard \psidos\ as defined in the previous sections and the toroidal \psidos\ considered in~\cite{GJP:MAMS17, LNP:TAMS16}. In particular, we will see that we get the exact same class of operators. This will extend to the noncommutative case the equality between toroidal and standard \psidos\ on $\T^n$ given in~\cite{RT:Birkhauser10}. 

The class of toroidal symbols specifically pertains to the Fourier series decomposition in $\cA_\theta$. Therefore, toroidal symbols do not make sense for more general $C^*$-dynamical systems. Nevertheless, they allows us to make use of Fourier series techniques to simplify the derivation of a few important properties of \psidos\ on noncommutative tori. This fact will be illustrated by the characterization of smoothing operators at the end of this section. 

\subsection{Toroidal symbols} 
Given any symbol $\rho(\xi)\in \stS^m(\R^n;\cA_\theta)$ we know by Proposition~\ref{prop:PsiDOs.Prhou-equation} that 
\begin{equation*}
 P_\rho u = \sum_{k\in \Z^n} u_k\rho(k)U^k, \qquad u=\sum u_kU^k\in \cA_\theta. 
\end{equation*}
Therefore, it stems for reason to consider more generally operators $P:\cA_\theta \rightarrow \cA_\theta$ of the above form where $(\rho(k))_{k\in \Z^n}$ is replaced by sequences of a suitable class. This class of sequences is the class of toroidal symbols. They are defined as follows. 

Let $\cA_\theta^{\Z^n}$ denote the space of sequences with values in $\cA_\theta$ that are indexed by $\Z^n$. In addition, let $(e_1, \ldots, e_n)$ be the canonical basis of $\R^n$. For $i=1,\ldots, n$, the difference operator $\Delta_i: \cA_\theta^{\Z^n} \rightarrow \cA_\theta^{\Z^n}$ is defined by
\begin{equation*}
 \Delta_i u_k= u_{k+e_i}-u_k, \qquad (u_k)_{k\in \Z^n} \in \cA_\theta^{\Z^n}. 
\end{equation*}
The operators $\Delta_1, \ldots, \Delta_n$ pairwise commute. For $\alpha =(\alpha_1, \ldots, \alpha_n)\in \N_0^n$, we set $\Delta^\alpha =\Delta_1^{\alpha_1} \cdots \Delta_n^{\alpha_n}$. We similarly define backward difference operators $\overline{\Delta}=\overline{\Delta}_1^{\alpha_1} \cdots \overline{\Delta}_n^{\alpha_n}$, where $\overline{\Delta}_i:  \cA_\theta^{\Z^n} \rightarrow \cA_\theta^{\Z^n}$ is defined by
\begin{equation*}
 \Delta_i u_k= u_{k}-u_{k-e_i}, \qquad (u_k)_{k\in \Z^n} \in \cA_\theta^{\Z^n}. 
\end{equation*}
We refer to~\cite{RT:Birkhauser10} for the main properties of difference operators.  

\begin{definition}[Toroidal Symbols~\cite{LNP:TAMS16, RT:Birkhauser10}] $\stS^m(\Z^n; \cA_\theta)$, $m\in \R$, consists of sequences $(\rho_k)_{k\in \Z^n} \subset\cA_\theta$ such that, for all multi-orders $\alpha$ and $\beta$, there is $C_{\alpha\beta}>0$ such that 
\begin{equation*}
 \big\| \delta^\alpha \Delta^\beta \rho_k \| \leq C_{\alpha\beta} (1+|k|)^{m-|\beta|} \qquad \forall k \in \Z^n. 
\end{equation*}
\end{definition}

\begin{remark}
 When $\theta=0$ we recover the toroidal symbols on $\T^n$ of~\cite{RT:Birkhauser10}. 
\end{remark}

\begin{remark}
If $(\rho_k)_{k\in \Z^n} \in  \stS^m(\Z^n; \cA_\theta)$, then $(\delta^\alpha \Delta^\beta \rho_k)_{k\in \Z^n} \in  \stS^{m-|\beta|}(\Z^n; \cA_\theta)$  for all $\alpha,\beta\in \N_0^n$. 
\end{remark}

\begin{remark}
 $ \stS^m(\Z^n; \cA_\theta)$ is a Fr\'echet space with respect to the locally convex topology generated by the norms, 
\begin{equation*}
 (\rho_k)_{k\in \Z^n} \longrightarrow \sup_{|\alpha|+|\beta|\leq N} \sup_{k\in \Z^n} (1+|k|)^{-m+|\beta|}\big\| \delta^\alpha \Delta^\beta \rho_k  \big\|, \qquad N\geq 0.  
\end{equation*}
\end{remark}

It is also convenient to introduce the following class of sequences. 

\begin{definition}
 $\cS(\Z^n; \cA_\theta)$ consists of sequences $(\rho_k)_{k\in \Z^n} \subset \cA_\theta$ such that, for all $N\in \N_0$ and $\alpha\in \N_0^n$, there is $C_{N\alpha}>0$ such that 
 \begin{equation*}
 \|\delta^\alpha \rho_k\|\leq C_{N\alpha}(1+|k|)^{-N}\qquad  \forall k \in \Z^n.  
\end{equation*}
\end{definition}
 
\begin{remark}\label{rmk:toroidal.symbols-intersection}
 If $(\rho_k)_{k\in \Z^n}\in \cS(\Z^n; \cA_\theta)$, then $(\delta^\alpha \Delta^\beta \rho_k)_{k\in \Z^n}\in \cS(\Z^n; \cA_\theta)$ for all $\alpha, \beta \in \N_0^n$. It then follows that
\begin{equation*}
 \cS(\Z^n; \cA_\theta) = \bigcap_{m\in \R} \stS^m(\Z^n; \cA_\theta). 
\end{equation*}
\end{remark}
 
 \begin{remark}
 A sequence $(\rho_k)_{k\in \Z^n} \subset \cA_\theta$ is in  $\cS(\Z^n; \cA_\theta)$ if and only if $((1+|k|)^{N}\rho_k)_{k\in \Z^n}$ is bounded in $\cA_\theta$ for all $N\in \N_0$. Therefore, if $(\rho_k)_{k\in \Z^n}\in \cS(\Z^n; \cA_\theta)$, then $((1+|k|)^{N}\rho_k)_{k\in \Z^n}$ is contained in $\cS(\Z^n; \cA_\theta)$ for every $N\in \N_0$. 
\end{remark}
 
\begin{remark}
 $\cS(\Z^n; \cA_\theta)$ is a Fr\'echet space with respect to the topology generated by the semi-norms, 
\begin{equation*}
  (\rho_k)_{k\in \Z^n} \longrightarrow \sup_{k\in \Z^n} (1+|k|)^{-N}\big\| \delta^\alpha \rho_k\big\|, \qquad N\geq 0, \quad \alpha \in \N_0^n.
\end{equation*}
\end{remark}

Following~\cite{LNP:TAMS16} (see also~\cite{GJP:MAMS17}) the \psido\ associated with a  toroidal symbol $(\rho_k)_{k\in \Z^n}\in \stS^m(\Z^n; \cA_\theta)$, $m\in \R$, is the linear operator $P:\cA_\theta \rightarrow \cA_\theta$ given by 
\begin{equation*}
 P u = \sum_{k\in \Z^n} u_k\rho(k)U^k, \qquad u=\sum u_kU^k\in \cA_\theta. 
\end{equation*}
We refer to~\cite{LNP:TAMS16} for the verification that this operator is well defined and continuous.

\subsection{From standard symbols to toroidal symbols}
We also can define difference operators $\Delta^\alpha = \Delta_1^{\alpha_1} \cdots \Delta_n^{\alpha_n}$ on $C^\infty(\R^n; \cA_\theta)$, where  $\Delta_i:C^\infty(\R^n; \cA_\theta) \rightarrow C^\infty(\R^n; \cA_\theta)$ is given by 
\begin{equation*}
 (\Delta_i u)(\xi): = u(\xi+e_i)-u(\xi), \qquad u\in C^\infty(\R^n; \cA_\theta), \ \xi \in \R^n. 
\end{equation*}

We record the following version of Peetre's inequality.

\begin{lemma}[see, e.g., {\cite[Lemma I.8.2]{AG:AMS07}}] \label{lem:Peetre}
Let $m\in\R$. Then, for all $\xi,\eta\in \Rn$, we have 
\begin{equation}
(1+|\xi+\eta|)^m\leq (1+|\xi|)^m (1+|\eta|)^{|m|}.
 \label{eq:Peetre}
\end{equation}
\end{lemma}

The following lemma relates for standard symbols difference operators to partial derivatives.

\begin{lemma}\label{lem:toroidal.finite-diff-partial-derivatives}
 Let $\rho(\xi)\in \stS^m(\R^n; \cA_\theta)$. Then, for all $\alpha \in \N_0^n$, we have
\begin{equation*}
\Delta^\alpha \rho(\xi)\in   \stS^{m-|\alpha|}(\R^n; \cA_\theta), \qquad  \Delta^\alpha \rho(\xi)-\partial^\alpha_\xi \rho(\xi) \in   \stS^{m-|\alpha|-1}(\R^n; \cA_\theta). 
\end{equation*}
\end{lemma}
\begin{proof}
 We proceed by induction on $|\alpha|$.  Let us first prove the result when $|\alpha|=1$. By Proposition~\ref{prop:LCS-Taylor}, for $i=1, \ldots, n$, we have 
\begin{equation}
 \Delta_i \rho(\xi)-\partial_{\xi_i}\rho(\xi) = \rho(\xi+e_i)-\rho(\xi)-\partial_{\xi_i}\rho(\xi)= \int_0^1 (1-t)\partial_{\xi_i}^2\rho(\xi+te_i)dt.
 \label{eq:toroidal.Taylor-order2}  
\end{equation}
Using the fact that $\partial_{\xi_i}^2\rho(\xi)\in \stS^{m-2}(\R^n;\cA_\theta)$ together with Peetre's inequality~(\ref{eq:Peetre}) we see there is $C>0$ such that, for all $\xi\in \R^n$ and $t\in [0,1]$, we have 
\begin{align*}
 \| \partial_{\xi_i}^2\rho(\xi+te_i)\| &\leq C(1+|\xi+te_i|)^{m-2} \\ 
& \leq  C(1+t)^{|m-2|}(1+|\xi|)^{m-2}\\ 
 &\leq 2^{|m-2|}C(1+|\xi|)^{m-2}. 
\end{align*}
Combining this with~(\ref{eq:toroidal.Taylor-order2}) shows that there is $C>0$ such that, for all $\xi\in \R^n$, we have
\begin{equation*}
 \| \Delta_i \rho(\xi)-\partial_{\xi_i}\rho(\xi)\| \leq \int_0^1 (1-t)\|\partial_{\xi_i}^2\rho(\xi+te_i)\|dt \leq C(1+|\xi|)^{m-2}. 
\end{equation*}
Likewise, as $\delta^\alpha\partial_\xi^\beta\rho\in \stS^{m-|\beta|}(\R^n; \cA_\theta)$,  given any multi-orders $\alpha$ and $\beta$, there is $C_{\alpha\beta}>0$ such that, for all $\xi\in \R^n$, we have
\begin{equation*}
 \big\| \delta^\alpha\partial_\xi^\beta(\Delta_i \rho-\partial_{\xi_i}\rho)(\xi)\big\| = \big\| \Delta_i (\delta^\alpha\partial_\xi^\beta\rho)(\xi)-\partial_{\xi_i}(\delta^\alpha\partial_\xi^\beta\rho)(\xi)\big\| \leq C_{\alpha\beta}(1+|\xi|)^{m-|\beta|-2}. 
 \end{equation*}
This shows that  $\Delta_i \rho(\xi)-\partial_{\xi_i}\rho(\xi) \in \stS^{m-2}(\R^n; \cA_\theta)$. As $\partial_{\xi_i}\rho(\xi) \in 
 \stS^{m-1}(\R^n; \cA_\theta)$, we deduce that $\Delta_i \rho(\xi)  \in 
 \stS^{m-1}(\R^n; \cA_\theta)$. This proves the result when $|\alpha|=1$. 

Suppose that the result is true for $|\alpha|\leq N$. Let $\alpha\in \N_0^n$, $|\alpha|\leq N$. For $i=1, \ldots, n$, we have 
\begin{equation*}
  \Delta_i \Delta^\alpha \rho  - \partial_{\xi_i} \partial_{\xi}^\alpha\rho = \Delta_i (\Delta^\alpha \rho -\partial_{\xi}^\alpha\rho) +\partial_{\xi}^\alpha (\Delta_i\rho -\partial_{\xi_i}\rho). 
\end{equation*}
By assumption $\Delta^\alpha \rho -\partial_{\xi}^\alpha\rho\in \stS^{m-|\alpha|-1}(\Rn;\cA_\theta)$, and so by the first part of the proof the finite difference $\Delta_i (\Delta^\alpha \rho -\partial_{\xi}^\alpha\rho)$ is a symbol in $\stS^{m-|\alpha|-2}(\Rn;\cA_\theta)$. The first part of the proof also ensures us that $\Delta_i\rho -\partial_{\xi_i}\rho\in \stS^{m-2}(\Rn;\cA_\theta)$, and so $\partial_{\xi}^\alpha (\Delta_i\rho -\partial_{\xi_i}\rho)$ is in $\stS^{m-|\alpha|-2}(\R^n; \cA_\theta)$ as well. It then follows that $ \Delta_i \Delta^\alpha \rho  - \partial_{\xi_i} \partial_{\xi}^\alpha\rho\in \stS^{m-|\alpha|-2}(\R^n; \cA_\theta)$. As $\partial_{\xi_i} \partial_{\xi}^\alpha\rho \in \stS^{m-|\alpha|-1}(\R^n; \cA_\theta)$ this implies that $\Delta_i\Delta^\alpha \rho \in \stS^{m-|\alpha|-1}(\R^n; \cA_\theta)$. This proves the result for multi-orders $\alpha$ with $|\alpha|\leq N+1$. The proof is complete. 
\end{proof}

\begin{remark}
 Even when $\theta=0$ Lemma~\ref{lem:toroidal.finite-diff-partial-derivatives} seems to be new. 
\end{remark}

\begin{corollary}[\cite{RT:Birkhauser10}]\label{cor:toroidal.SR-SZ-m}
If $\rho(\xi)\in \stS^m(\R^n; \cA_\theta)$, $m\in \R$, then $(\rho(k))_{k\in \Z^n} \in  \stS^m(\Z^n; \cA_\theta)$. 
\end{corollary}
\begin{proof}
 Let $\alpha, \beta \in \N_0^n$. As $\delta^\alpha \rho(\xi)\in \stS^m(\R^n; \cA_\theta)$, Lemma~\ref{lem:toroidal.finite-diff-partial-derivatives} ensures us that $\delta^\alpha \Delta^\beta \rho(\xi)= \Delta^\beta \delta^\alpha\rho(\xi)$ is contained in $\stS^{m-|\beta|}(\R^n;\cA_\theta)$, and so $\delta^\alpha \Delta^\beta \rho(k)=\op{O}(|k|^{m-|\beta|})$ as $|k|\rightarrow \infty$. This shows that  $(\rho(k))_{k\in \Z^n} \in  \stS^m(\Z^n; \cA_\theta)$. The result is proved.
 \end{proof}

The above result asserts that the restriction to $\Z^n$ of a standard symbol is a toroidal symbol. Combining this with Proposition~\ref{prop:PsiDOs.Prhou-equation} shows that every standard \psido\ is a toroidal \psido. 

Set $\stS^{+\infty}(\R^n;\cA_\theta):=  \bigcup_{\mu \in \R} \stS^{\mu}(\R^n;\cA_\theta)$. We have  the following converse of Corollary~\ref{cor:toroidal.SR-SZ-m}. 

\begin{proposition}\label{prop:toroidal.rhok-rohxim}
 Let $\rho(\xi)\in \stS^{+\infty}(\R^n;\cA_\theta)$ be such that $(\rho(k))_{k\in \Z^n}\in \stS^m(\Z^n;\cA_\theta)$, $m\in \R$. Then $\rho(\xi)$ is contained in 
 $\stS^m(\R^n;\cA_\theta)$.  
\end{proposition}
\begin{proof}
 With a view toward contradiction suppose that $\rho(\xi)\not \in \stS^m(\R^n;\cA_\theta)$. Set
\begin{equation*}
 m_0= \inf \left\{ \mu \in \R; \ \rho(\xi)\in \stS^\mu(\R^n; \cA_\theta)\right\}. 
\end{equation*}
The definition of $m_0$ implies that $\rho (\xi)\in \stS^{m'}(\R^n;\cA_\theta)$ for all $m'>m_0$. Thus, the assumption that 
$\rho(\xi)\not \in \stS^m(\R^n;\cA_\theta)$ implies that $m\leq m_0$. 

Let $\mu \in (m_0-1,m_0)$. Thus, $\rho(\xi)$ is not in $\stS^\mu(\R^n;\cA_\theta)$, and so if we set $\overline{m}= \max(m,\mu)$, then $\rho(\xi)\not\in \stS^{\overline{m}}(\R^n;\cA_\theta)$. Furthermore, as $\mu>m_0-1$, we have $\mu+1>m_0$, and so $\rho(\xi)\in \stS^{\mu+1}(\R^n;\cA_\theta)$. Combining this with Lemma~\ref{lem:toroidal.finite-diff-partial-derivatives} shows that, given any $\beta \in \N_0^n$, the difference $\partial_\xi^\beta\rho(\xi) - \Delta^\beta \rho(\xi)$ is a symbol in $\stS^{\mu-|\beta|}(\R^n;\cA_\theta)$, and hence by Corollary~\ref{cor:toroidal.SR-SZ-m} the sequence $(\partial_\xi^\beta\rho(k) - \Delta^\beta \rho(k))_{k\in \Z^n}$ is contained in $\stS^{\mu-|\beta|}(\Z^n;\cA_\theta)$. Thus, given any $\alpha \in \N_0^n$,  we have 
\begin{equation}
 \|\delta^\alpha\partial_\xi^\beta\rho(k) - \delta^\alpha\Delta^\beta \rho(k)\| = \op{O}(|k|^{\mu-|\beta|})=  \op{O}(|k|^{\overline{m}-|\beta|}) \qquad \text{as $|k|\rightarrow \infty$}.
 \label{eq:toroidal.deltapart-deltaDeltak} 
\end{equation}
By assumption $(\rho(k))_{k \in \Z^n}\in \stS^{m}(\Z^n; \cA_\theta)$, and so $\|\delta^\alpha\Delta^\beta \rho(k)\| = \op{O}(|k|^{m-|\beta|})=  \op{O}(|k|^{\overline{m}-|\beta|})$ as $|k|\rightarrow \infty$. Combining this with~(\ref{eq:toroidal.deltapart-deltaDeltak}) shows that, for all $\alpha, \beta \in \N_0^n$, there is $C_{\alpha\beta}>0$  such that
\begin{equation}
  \|\delta^\alpha\partial_\xi^\beta\rho(k) \| \leq C_{\alpha\beta} (1+|k|)^{\overline{m}-|\beta|} \qquad \forall k \in \Z^n.
  \label{eq:toroidal.deltapartphok} 
\end{equation}

Bearing this in mind, let $\xi\in \R^n$ and $k\in \Z^n$ be such that $|\xi_i-k_i|\leq 1$ for $i=1,\ldots, n$. In particular, we have $|\xi-k|\leq \sqrt{n}$. By Lemma~\ref{lem:LCS.C1-differentiable} we have 
\begin{equation}
\rho(\xi)-\rho(k) =\sum_{1\leq j \leq n} (\xi_j -k_j) \int_0^1 \partial_{\xi_j} \rho\big(\xi+t(k-\xi)\big) dt. 
 \label{eq:toroidal.Taylor-order1} 
\end{equation}
As  $\rho(\xi)\in \stS^{\mu+1}(\R^n;\cA_\theta)$ each partial derivative $ \partial_{\xi_j} \rho(\xi)$ is a symbol in 
$\stS^{\mu}(\R^n;\cA_\theta)\subset \stS^{\overline{m}}(\R^n;\cA_\theta)$. By combining this with Peetre's inequality~(\ref{eq:Peetre}) and using the inequality $|\xi-k|\leq \sqrt{n}$, we see there is a constant $C>0$ independent of $\xi$ and $k$ such that, for all $t\in [0,1]$, we have
\begin{align*}
\big\|\partial_{\xi_j} \rho\big(\xi+t(k-\xi)\big)\big\| & \leq C\big( 1+ |\xi +t (k-\xi)|\big)^{\overline{m}} \\
& \leq C \big( 1+ t|k-\xi|\big)^{|\overline{m}|} ( 1+ |\xi|)^{\overline{m}} \\
& \leq C\big(1+  \sqrt{n}\big)^{|\overline{m}|} ( 1+ |\xi|)^{\overline{m}} . 
\end{align*}
Combining this with~(\ref{eq:toroidal.Taylor-order1}) we deduce there is $C>0$ independent of $\xi$ and $k$ such that
\begin{equation*}
 \big\|  \rho(\xi)-\rho(k)\| \leq C ( 1+ |\xi|)^{\overline{m}} . 
\end{equation*}
Likewise, given any multi-orders $\alpha$ and $\beta$, there is a constant $C_{\alpha\beta}>0$ independent of $\xi$ and $k$, such that 
\begin{equation}
  \big\|  \delta^\alpha\partial_\xi^\beta\rho(\xi)- \delta^\alpha\partial_\xi^\beta\rho(k)\| \leq C_{\alpha\beta} (1+ |\xi|)^{\overline{m}-|\beta|}.
  \label{eq:toroidal.deltapart-xi-k}
\end{equation}

In addition, in view of~(\ref{eq:toroidal.deltapartphok}), by using Peetre's inequality~(\ref{eq:Peetre})  and arguing as above we get 
 \begin{align*}
  \|\delta^\alpha\partial_\xi^\beta\rho(k) \| & \leq C_{\alpha\beta} (1+|k|)^{\overline{m}-|\beta|}\\ 
  & \leq C_{\alpha\beta} \big( 1+ |k-\xi|\big)^{|\overline{m}-|\beta||} ( 1+ |\xi|)^{\overline{m}-|\beta|}\\
  & \leq C_{\alpha\beta} \big(1+  \sqrt{n}\big)^{|\overline{m}-|\beta||}  (1+ |\xi|)^{\overline{m}-|\beta|}.
\end{align*}
Combining this with~(\ref{eq:toroidal.deltapart-xi-k}) we deduce that there is $C_{\alpha\beta}>0$ such that
\begin{equation*}
  \big\|  \delta^\alpha\partial_\xi^\beta\rho(\xi)\| \leq C_{\alpha\beta} ( 1+ |\xi|)^{\overline{m}-|\beta|} \qquad \forall \xi \in \R^n.
\end{equation*}
This shows that $\rho(\xi)\in\stS^{\overline{m}}(\R^n; \cA_\theta)$. This gives a contradiction, since  $\rho(\xi)$ is not contained in $\stS^{\overline{m}}(\R^n; \cA_\theta)$. Thus, it is impossible for $\rho(\xi)$ not to be contained in $\stS^{m}(\R^n; \cA_\theta)$, and so it must be a symbol of this class. The proof is complete. 
\end{proof}

\begin{remark}
 Even when $\theta=0$ Proposition~\ref{prop:toroidal.rhok-rohxim} seems to be new. 
\end{remark}

Recall that  $\cS(\R^n; \cA_\theta)=\bigcap \stS^m(\R^n;\cA_\theta)$ and $\cS(\Z^n; \cA_\theta) = \bigcap \stS^m(\Z^n;\cA_\theta)$ (\emph{cf}.\ Remark~\ref{rmk:Symbols.amplitudes-intersection} and Remark~\ref{rmk:toroidal.symbols-intersection}). Therefore, specializing Corollary~\ref{cor:toroidal.SR-SZ-m} and Proposition~\ref{prop:toroidal.rhok-rohxim} to symbols of order~$-\infty$ leads us to the following characterization of these symbols.

\begin{corollary}\label{cor:toroidal.rhok-rho-cS}
 Let $\rho(\xi)\in\stS^{+\infty}(\R^n;\cA_\theta)$. Then $\rho(\xi)\in \cS(\R^n; \cA_\theta)$ if and only if $(\rho(k))_{k\in \Z^n}\in \cS(\Z^n; \cA_\theta)$. 
\end{corollary}

Further specializing Corollary~\ref{cor:toroidal.rhok-rho-cS} to the case where $\rho=0$ on $\Z^n$ we get the following statement, which was used in the proof of Proposition~\ref{prop:PsiDOs.vanishing-Prho}. 

\begin{corollary}\label{cor:toroidal.rhok0-cS}
  Let $\rho(\xi)\in \stS^{+\infty}(\R^n;\cA_\theta)$ be such that $\rho(k)=0$ for all $k\in \Z^n$. Then $\rho(\xi)\in \cS(\R^n; \cA_\theta)$. 
\end{corollary}

\subsection{From toroidal symbols to standard symbols} 
As mentioned above, it follows from Proposition~\ref{prop:PsiDOs.Prhou-equation} and Corollary~\ref{cor:toroidal.SR-SZ-m} that any standard \psido\ is a toroidal \psido. We shall now establish that, conversely, any toroidal \psido\ is a standard \psido, and so we do not get new operators by considering toroidal \psidos. The bulk of the proof is showing that any toroidal symbol seen as a function on $\Z^n$ can be extended into a standard symbol defined on all $\R^n$. When $\theta=0$ this result is part 
of~\cite[Theorem~4.5.3]{RT:Birkhauser10}. Our approach is a mere elaboration of the arguments of~\cite{RT:Birkhauser10}. The result is also mentioned in~\cite{LNP:TAMS16}. We shall however give a detailed account for reader's convenience.

\begin{lemma}[{\cite[Lemma 4.5.1]{RT:Birkhauser10}}] \label{lem:toroidal.phi}There exists a function $\phi(\xi)\in \cS(\Rn)$ such that
\begin{enumerate}
\item[(i)] $\phi(0)=1$ and $\phi(k)=0$ for all $k\in \Z^n\setminus 0$. 

\item[(ii)] For every multi-order $\alpha$, there is $\phi_\alpha(\xi)\in \cS(\Rn)$ such that $\partial_\xi^\alpha \phi(\xi) =\overline{\Delta}^\alpha \phi_\alpha(\xi)$. 
\end{enumerate}
\end{lemma}

\begin{remark}
 The idea of this lemma is due to Yves Meyer (see~\cite[page 4]{Da:Springer92}).
\end{remark}

\begin{remark}
Our convention for the Fourier transform differs from that in~\cite{RT:Birkhauser10}. Thus, we get a function $\phi(\xi)$ satisfying the conditions of the lemma as the Fourier transform 
$\phi(\xi)=\hat{\theta}(\xi)$, where $\theta(x)=\theta_1(x_1) \cdots \theta_1(x_n)$ and $\theta_1(t)$ is a smooth even function on $\R$ with support in $(-2\pi, 2\pi)$ such that $\theta_1(t)+\theta_1(2\pi-t)=(2\pi)^{-1}$ on $[0,2\pi]$. 
\end{remark}

We will also need the following summation by parts formula. 

\begin{lemma}
 Let $(u_k)_{k\in \Z^n} \in \cS(\Z^n; \cA_\theta)$ and $(\rho_k)_{k\in \Z^n} \in \stS^m(\Z^n; \cA_\theta)$, $m\in \R$. Then, for all $\alpha \in \N_0^n$, we have 
\begin{equation}
\sum_{k\in \Z^n}(\overline{\Delta}^\alpha u_k) \rho_k = (-1)^{|\alpha|} \sum_{k\in \Z^n}  u_k  \Delta^\alpha\rho_k.
\label{eq:toroidal.sum-parts} 
\end{equation}
\end{lemma}
\begin{proof}
The proof is the same as in the scalar case (see~\cite[Lemma 3.3.10]{RT:Birkhauser10}).  For $i=1,\ldots, n$, we have 
\begin{equation*}
 \sum_{k\in \Z^n}  (\overline{\Delta}_i u_k) \rho_k  =   \sum_{k\in \Z^n} (u_k-u_{k-e_i})  \rho_k =  \sum_{k\in \Z^n} u_k \rho_k -  \sum_{k\in \Z^n} u_{k-e_i}\rho_k.  
\end{equation*}
The change of index $k\rightarrow k+e_i$ gives
\begin{equation*}
 \sum_{k\in \Z^n}  (\overline{\Delta}_i u_k) \rho_k  =  \sum_{k\in \Z^n} u_k \rho_k -  \sum_{k\in \Z^n} u_{k}\rho_{k+e_i}= - \sum_{k\in \Z^n} u_k \Delta_i\rho_k
\end{equation*}
This proves the result when $|\alpha|=1$. The general result follows by induction. 
\end{proof}

Given some toroidal symbol $(\rho_k)_{k\in \Z^n}\in \stS^m(\Z^n; \cA_\theta)$, $m\in \R$, set 
\begin{equation}
\tilde{\rho}(\xi) = \sum_{k\in \Z^n} \phi(\xi-k) \rho_k, \qquad \xi \in \R^n, 
 \label{eq:toroidal.trho} 
\end{equation}
where $\phi(\xi)\in \cS(\R^n)$ satisfies the properties (i)--(ii) of Lemma~\ref{lem:toroidal.phi}.  Note that the above series converges pointwise, since, for each $\xi\in \R^n$, this is the sum of a series with general term in $\cS(\Z^n;\cA_\theta)$. In addition, given any $k\in \Z^n$, thanks to the property~(i) of Lemma~\ref{lem:toroidal.phi} we have
 \begin{equation}
 \tilde{\rho}(k)= \sum_{\ell \in \Z^n} \phi(k-\ell) \rho_{\ell}=  \sum_{\ell \in \Z^n} \delta_{\ell,k} \rho_{\ell}=\rho_k.
 \label{eq:toroidal.trhoZn} 
\end{equation}
 
 \begin{lemma}[see also~\cite{LNP:TAMS16, RT:Birkhauser10}] \label{lem:toroidal.trho-symbol}
The map $\tilde{\rho}(\xi)$ defined by~(\ref{eq:toroidal.trho}) is a symbol in $\stS^m(\R^n; \cA_\theta)$. 
\end{lemma}
\begin{proof}
Let us first show that $\tilde{\rho}(\xi)\in C^\infty(\R^n; \cA_\theta)$.  Let $\alpha,\beta\in \N_0^n$. By the property~(ii) of Lemma~\ref{lem:toroidal.phi} there is $\phi_\beta(\xi)\in \cS(\R^n)$ such that 
\begin{equation}
 \delta^\alpha \partial_\xi^\beta \big( \phi(\xi-k) \rho_k\big)  = (\partial_\xi^\beta \phi)(\xi-k) \delta^\alpha \rho_k \\
  = (\overline{\Delta}^\beta \phi_\beta)(\xi-k)\delta^\alpha \rho_k. 
  \label{eq:toroidal.delta-part-phirho}
\end{equation}

As $\overline{\Delta}^\beta \phi_\beta\in \cS(\R^n)$ given any $N>0$ there is a constant $C_{N\beta}>0$ such that
\begin{equation*}
 |\overline{\Delta}^\beta \phi_\beta (\eta)|\leq C_{N\beta} (1+|\eta|)^{-N} \qquad \forall \eta \in \R^n. 
\end{equation*}
Combining this with Peetre's inequality~(\ref{eq:Peetre}) we see that, for all $\xi \in \R^n$ and $k\in \Z^n$, we have 
\begin{equation*}
 \big|(\overline{\Delta}^\beta \phi_\beta)(\xi-k)\big| \leq C_{N\beta} (1+|\xi-k|)^{-N} \leq C_{N\beta} (1+|\xi|)^{N} (1+|k|)^{-N}. 
\end{equation*}
In addition, the fact that $(\rho_k)_{k\in \Z^n}\in \stS^m(\Z^n; \cA_\theta)$ implies that $ \|\delta^\alpha \rho_k\|=\op{O}(|k|^m)$ as $|k|\rightarrow \infty$. Combining this with~(\ref{eq:toroidal.delta-part-phirho}) we deduce there is a constant $C_{N\alpha \beta}>0$ such that, for all $\xi \in \R^n$ and $k\in \Z^n$, we have
\begin{equation*}
\big\|  \delta^\alpha \partial_\xi^\beta \big( \phi(\xi-k) \rho_k\big) \big\| = \big| (\overline{\Delta}^\beta \phi_\beta)(\xi-k)\big| \big\|\delta^\alpha \rho_k\big\| \leq 
C_{N\alpha\beta} (1+|\xi|)^N (1+|k|)^{m-N}. 
\end{equation*}
If we choose $N$ large enough then we see that the series $\sum  \delta^\alpha \partial_\xi^\beta ( \phi(\xi-k) \rho_k)$ converges normally with respect to the norm $\|\cdot\|$ and uniformly on bounded sets in $\R^n$. It then follows that the series $\sum  \phi(\xi-k) \rho_k$ converges in $C^\infty(\R^n;\cA_\theta)$, and so its sum $\tilde{\rho}(\xi)$ is contained in $C^\infty(\R^n; \cA_\theta)$. 

The convergence in $C^\infty(\R^n;\cA_\theta)$ allows us to term-wise differentiate the series  $\sum  \phi(\xi-k) \rho_k$. Thus, by using~(\ref{eq:toroidal.delta-part-phirho}) and the summation by parts~(\ref{eq:toroidal.sum-parts}) we see that  $\delta^\alpha \partial_\xi^\beta \tilde{\rho}(\xi) $ is equal to
\begin{equation}
  \sum_{k\in \Z^n} \delta^\alpha \partial_\xi^\beta \big( \phi(\xi-k) \rho_k\big) 
 =  \sum_{k\in \Z^n} (\overline{\Delta}^\beta \phi_\beta)(\xi-k)\delta^\alpha \rho_k=  \sum_{k\in \Z^n}(-1)^{|\beta|} \phi_\beta(\xi-k)\delta^\alpha \Delta^\beta \rho_k. 
 \label{eq:toroidal.delta-part-trho}
 \end{equation}
As $(\rho_k)_{k\in \Z^n}\in \stS^m(\Z^n;\cA_\theta)$, by using Peetre's inequality~(\ref{eq:Peetre}) once again we see there is a constant $C_{\alpha\beta}>0$ such that, for all $\xi \in \R^n$ and $k \in \Z^n$, we have 
\begin{equation*}
 \big\| \delta^\alpha \Delta^\beta \rho_k\big\|  \leq C_{\alpha\beta}(1+|k|)^{m-|\beta|} \leq  C_{\alpha\beta}(1+|\xi-k|)^{|m-|\beta||}(1+|\xi|)^{m-|\beta|}. 
\end{equation*}
 Combining this with~(\ref{eq:toroidal.delta-part-trho}) we get
\begin{equation*}
\big\|  \delta^\alpha \partial_\xi^\beta \tilde{\rho}(\xi)\big\| \leq \sum_{k\in \Z^n} \big|  \phi_\beta(\xi-k)\big| \big\| \delta^\alpha\Delta^\beta \rho_k\big\|  \leq C_{\alpha\beta} 
\sum_{k\in \Z^n} \big|  \phi_\beta(\xi-k)\big| (1+|\xi-k|)^{|m-|\beta||}(1+|\xi|)^{m-|\beta|}.  
\end{equation*}
Note that $\sum_{k\in \Z^n} \big|  \phi_\beta(\xi-k)\big| (1+|\xi-k|)^{|m-|\beta||}$ is the periodization of the rapid decay continuous function 
$\phi(\eta)(1+|\eta|)^{|m-|\beta||}$, 
and so this is a bounded function. Thus,  there is a constant $C_{\alpha\beta}>0$ such that
\begin{equation*}
\big\|  \delta^\alpha \partial_\xi^\beta \tilde{\rho}(\xi)\big\| \leq C_{\alpha \beta} (1+|\xi|)^{m-|\beta|} \qquad \forall \xi\in \R^n. 
\end{equation*}
This shows that $\tilde{\rho}(\xi)$ is a symbol in $\stS^m(\R^n;\cA_\theta)$. The proof is complete. 
\end{proof}

\begin{remark}
The construction of $\tilde{\rho}(\xi)$ is given in the proof of Theorem~4.5.3 of~\cite{RT:Birkhauser10} in the case $\theta=0$.  This construction is extended to the case $\theta\neq 0$ in~\cite{LNP:TAMS16} (see also~\cite{GJP:MAMS17}). However, we note that in~\cite{GJP:MAMS17, LNP:TAMS16, RT:Birkhauser10} there are no justifications of the term-wise differentiation of the series $\sum \phi(\xi-k)\rho_k$. This justification is necessary in order to perform the key step~(\ref{eq:toroidal.delta-part-trho}). The proof above fixes this oversight. 
\end{remark}

\begin{remark}\label{rmk:toroidal.continuity-extension}
 The extension map $(\rho_k)_{k\in \Z^n}\rightarrow \tilde{\rho}(\xi)$ is a continuous linear map from $\stS^m(\Z^n; \cA_\theta)$ to $\stS^m(\R^n; \cA_\theta)$. This can be deduced from a close examination of the arguments of the proof above. Alternatively, given any $\xi_0\in \R^n$, the map 
 $(\rho_k)_{k\in \Z^n}\rightarrow \tilde{\rho}(\xi_0)$ is a continuous linear map from $\stS^m(\Z^n; \cA_\theta)$ to $\cA_\theta$. This allows us to show that the graph of the extension map is a closed set of $\stS^m(\Z^n; \cA_\theta)\times \stS^m(\R^n; \cA_\theta)$, and hence we have a continuous linear map by the closed graph theorem. \end{remark}

We are now in a position to prove the following result. 

\begin{proposition}[\cite{LNP:TAMS16, RT:Birkhauser10}] \label{prop:toroidal.extension-toroidal}
 Let $(\rho_k)_{k\in \Z^n}$ be a toroidal symbol in $\stS^m(\Z^n; \cA_\theta)$. Then there is a standard symbol $\tilde{\rho}(\xi)\in \stS^m(\R^n;\cA_\theta)$ such that 
 $\tilde{\rho}(k)=\rho_k$ for all $k\in \Z^n$. Moreover, this symbol is unique modulo $\cS(\R^n; \cA_\theta)$. 
\end{proposition}
\begin{proof}
 The existence follows from~(\ref{eq:toroidal.trhoZn}) and Lemma~\ref{lem:toroidal.trho-symbol}. The unicity is a consequence of Corollary~\ref{cor:toroidal.rhok0-cS}.\end{proof}

In the special of symbols of order~$-\infty$ we actually have the following statement. 

\begin{proposition}\label{prop:toroidal.extension-cS}
  Let $(\rho_k)_{k\in \Z^n}\in \cS(\Z^n; \cA_\theta)$. Then there is $\tilde{\rho}(\xi)\in \cS(\R^n;\cA_\theta)$ such that 
 $\tilde{\rho}(k)=\rho_k$ for all $k\in \Z^n$.
\end{proposition}
\begin{proof}
Let $\tilde{\rho}(\xi)$ be the symbol defined by~(\ref{eq:toroidal.trho}). We know by~(\ref{eq:toroidal.trhoZn}) that $\tilde{\rho}(k)=\rho_k$ for all $k\in \Z^n$. Moreover, 
as $(\rho_k)_{k\in \Z^n}$ is in $\cS(\Z^n; \cA_\theta)$,  by Remark~\ref{rmk:toroidal.symbols-intersection} it is contained in each space 
$\stS^m(\Z^n; \cA_\theta)$, $m\in \R$. Therefore, by Lemma~\ref{lem:toroidal.trho-symbol} the symbol $\tilde{\rho}(\xi)$ is contained in $\bigcap \stS^m(\R^n; \cA_\theta)= \cS(\R^n;\cA_\theta)$. This proves the result. 
\end{proof}

\begin{remark}
 By arguing as in Remark~\ref{rmk:toroidal.continuity-extension} it can be shown that the graph of the map $\cS(\Z^n;\cA_\theta)\ni (\rho_k)\rightarrow \tilde{\rho}(\xi)\in \cS(\R^n;\cA_\theta)$ is closed in $\cS(\Z^n;\cA_\theta)\times \cS(\R^n;\cA_\theta)$, and hence we get a continuous linear map by the closed graph theorem. 
\end{remark}

When $\theta=0$ it is shown in~\cite{RT:Birkhauser10} that the classes of toroidal and standard \psidos\ on $\T^n$ agree. We are now ready to obtain the analogue of this result when $\theta\neq 0$. Namely, we have the following statement. 

\begin{proposition}\label{prop:toroidal.toroidal=standard}
An operator on $\cA_\theta$ is the (toroidal) \psido\ associated with a symbol in $\stS^m(\Z^n; \cA_\theta)$ if and only if it is a (standard) \psido\ associated with a symbol in $\stS^m(\R^n; \cA_\theta)$. 
\end{proposition}
\begin{proof}
As mentioned above any standard \psido\ associated a symbol in $\stS^m(\R^n; \cA_\theta)$ is a (toroidal) \psido\ associated with a symbol in $\stS^m(\Z^n; \cA_\theta)$. 

Conversely, let $P:\cA_\theta \rightarrow \cA_\theta$ be the \psido\ associated with some toroidal symbol $(\rho_k)_{k\in \Z^n}\in \stS^m(\Z^n; \cA_\theta)$, $m\in \R$. By Proposition~\ref{prop:toroidal.extension-toroidal} there is a standard symbol $\tilde{\rho}(\xi)\in \stS^m(\R^n; \cA_\theta)$ such that $\tilde{\rho}(k)=\rho_k$ for all $k\in \Z^n$. Therefore, by using 
Proposition~\ref{prop:PsiDOs.Prhou-equation} we see that, for all $u=\sum u_k U^k$ in $\cA_\theta$, we have 
\begin{equation*}
  P u = \sum_{k\in \Z^n} u_k\rho_kU^k= \sum_{k\in \Z^n} u_k\tilde{\rho}(k)U^k =P_{\rho}u. 
\end{equation*}
This shows that $P=P_{\tilde{\rho}}$, and so $P$ is the \psido\ associated with a standard symbol of order $m$.  This completes the proof. 
\end{proof}

\subsection{Smoothing operators} As an application of the results of this section we shall now give a characterization of smoothing operators. 

As mentioned in Section~\ref{subsection:NCtori.Distributions} the inclusion of $\cA_\theta$ into $\cA_\theta'$ is dense. This leads us to the following notion of smoothing operators. 

\begin{definition}
 A linear operator $R:\cA_\theta \rightarrow \cA_\theta'$ is called \emph{smoothing} when it extends to a continuous linear operator $R:\cA_\theta'\rightarrow \cA_\theta$.
\end{definition}
 
We will need the following lemma. 

\begin{lemma} \label{lem:PsiDOs.aU-rapid-decaying}
 Let $(a_k)_{k\in \Z^n} \in \cS(\Z^n; \cA_\theta)$. Then the family $(a_kU^k)_{k\in \Z^n}$ is contained in $\cS(\Z^n; \cA_\theta)$. In particular, it is bounded in $\cA_\theta$. 
\end{lemma}
\begin{proof}
 Given any multi-order $\alpha$, we have 
\begin{equation*}
 \delta^\alpha(a_k U^k) = \sum_{\beta +\gamma=\alpha} \binom\alpha\beta \delta^\beta(a_k) \delta^\gamma(U^k)= \sum_{\beta +\gamma=\alpha} \binom\alpha\beta  k^\gamma \delta^\beta(a_k) U^k.
\end{equation*}
As the families $(U^k)_{k\in \Z^n}$ and $(k^\gamma \delta^\beta(a_k))_{k\in \Z^n}$ are bounded with respect to the norm $\|\cdot \|$, we deduce that  
$ (\delta^\alpha(a_k U^k))_{k\in \Z^n}$ is bounded with respect to the norm $\|\cdot \|$ for all $\alpha \in \N_0^n$. This shows that the family $(a_kU^k)_{k\in \Z^n}$ is bounded in $\cA_\theta$. 
Applying this result to the family $((1+|k|)^{N}a_k)_{k\in \Z^n}$ shows that the family $((1+|k|)^Na_kU^k)_{k\in \Z^n}$ is bounded in $\cA_\theta$ for all $N\in \N_0$, and hence $(a_kU^k)_{k\in \Z^n}\in \cS(\Z^n; \cA_\theta)$. The lemma is proved. 
\end{proof}
 
We have the following characterization of smoothing operators.

\begin{proposition} \label{prop:PsiDos.smoothing-condition}
 Let $R:\cA_\theta \rightarrow \cA_\theta$ be a linear operator. Then the following are equivalent:
\begin{enumerate}
 \item[(i)] $R$ is a smoothing operator. 
 
 \item[(ii)] $R$ is a toroidal \psido\ associated with a symbol in $\cS(\Z^n; \cA_\theta)$. 
 
 \item[(iii)] $R$ is standard \psido\ associated with a symbol in $\cS(\R^n; \cA_\theta)$. 
\end{enumerate}
\end{proposition}
\begin{proof} 
It follows from Proposition~\ref{prop:PsiDOs.Prhou-equation} and Proposition~\ref{prop:toroidal.extension-cS} that (ii) implies (iii). Therefore, we only have to establish that (i) implies (ii) and (iii) implies (i). 

Let us first show that that (i) implies (ii). Suppose that $R:\cA_\theta'\rightarrow\cA_\theta$ is a continuous linear operator. Let $u=\sum_k u_k U^k\in \cA_\theta$. This Fourier series converges in $\cA_\theta$, and so it converges in $\cA_\theta'$. As $R$ is continuous, we get 
\begin{equation} \label{eq:PsiDOs.Ru}
 Ru= \sum_{k \in \Z^n} u_k R(U^k)= \sum_{k \in \Z^n} u_k \rho_k U^k,
\end{equation}
where the series converges in $\cA_\theta$ and we have set $\rho_k=R(U^k)(U^k)^{-1}=R(U^k)(U^k)^{*}$, $k \in \Z^n$. 

Bearing this in mind, for every $v\in \cA_\theta$, the family $(\langle U^k,v\rangle)_{k\in \Z^n}=(\overline{(v^*|U^k)})_{k\in \Z^n}$ is  in $\cS(\Z^n)$, and so, given any integer $N\geq 0$, the family $((1+|k|)^N\langle U^k,v\rangle)_{k\in \Z^n}$ is bounded. As $\cA_\theta$ is a Fr\'echet space, the Banach-Steinhaus theorem ensures us that the sequence $((1+|k|)^NU^k)_{k\in \Z^n}$ is bounded in $\cA_\theta'$. 
The continuity of the operator $R$ then implies that the sequence $((1+|k|)^NR(U^k))_{k\in \Z^n}$ is bounded in $\cA_\theta$ for every $N\in \N$, i.e., $(R(U^k))_{k\in \Z^n} \in \cS(\Z^n; \cA_\theta)$. By arguing along the same lines as that of the proof of Lemma~\ref{lem:PsiDOs.aU-rapid-decaying}, it can be shown that the sequence $(\rho_k)_{k\in \Z^n}=(R(U^k)(U^k)^*)_{k\in \Z^n}$ is contained in  $\cS(\Z^n; \cA_\theta)$. Together with~(\ref{eq:PsiDOs.Ru})  this shows that $R$ is a toroidal \psido\ assiociated with a symbol in $\cS(\Z^n; \cA_\theta)$. This proves that (i) implies that (ii). 

It remains to show that (iii) implies (i).   Suppose that $R=P_\rho$ with $\rho(\xi)\in\cS(\Rn;\cA_\theta)$. Given any $u=\sum u_kU^k$ in $\cA_\theta$, we know by~(\ref{eq:NCtori.distrb-innerproduct-eq}) that, for all $k\in \Z^n$, we have $u_k=\acoup{u}{U^k}=\acou{u}{(U^k)^*}=\acou{(U^k)^*}{u}$. Therefore, by  
using Proposition \ref{prop:PsiDOs.Prhou-equation} we get
\begin{equation*} 
Ru=P_\rho u = \sum_{k\in\Z^n}u_k\rho(k)U^k = \sum_{k\in\Z^n}\acou{u}{(U^k)^*} \rho(k)U^k,
\end{equation*}
where the above series converge in $\cA_\theta$. Thus, given any multi-order $\alpha$ we have
\begin{equation} \label{eq:PsiDOs.delta-Prhou}
\delta^\alpha\left( R u \right) = \sum_{k\in\Z^n} \acou{u}{(U^k)^*} \delta^\alpha (\rho(k)U^k). 
\end{equation}

As $\rho(\xi)\in \cS(\R^n;\cA_\theta)$ it follows from Lemma~\ref{lem:PsiDOs.aU-rapid-decaying} that the family $ (\delta^\alpha [\rho(k)U^k])_{k\in \Z^n}$ is in $\cS(\Z^n; \cA_\theta)$, and so the family $(\|\delta^\alpha [\rho(k)U^k]\|)_{k\in \Z^n}$ is in $\cS(\Z^n)\subset \cS(\Z^n; \cA_\theta)$. Using again Lemma~\ref{lem:PsiDOs.aU-rapid-decaying} we see that the family $(\|\delta^\alpha [\rho(k)U^k]\| (U^k)^*)_{k\in \Z^n}$ is contained in $\cS(\Z^n; \cA_\theta)$. Thus, the family $\cB_\alpha:=\left\{ (1+|k|)^{n+1} \|\delta^\alpha (\rho(k)U^k)\|(U^k)^*; k \in \Z^n\right\}$ is bounded in $\cA_\theta$. Bearing this in mind we observe that, for all $k \in \Z^n$, we have 
\begin{align*}
  \left| \acou{u}{(U^k)^*}\right| \left\|\delta^\alpha (\rho(k)U^k)\right\|   \leq (1+|k|)^{-(n+1)} 
  \sup_{v \in \cB_\alpha} |\acou{u}{v}| .
\end{align*}
Combining this with~(\ref{eq:PsiDOs.delta-Prhou}) we deduce that, for all $u \in \cA_\theta$, we have 
\begin{equation*}
 \left\|\delta^\alpha\left( R u \right) \right\|  \leq  \sum_{k\in\Z^n}\left| \acou{u}{(U^k)^*}\right| \left\|\delta^\alpha (\rho(k)U^k)\right\|  \leq C \sup_{v \in \cB_\alpha} |\acou{u}{v}| ,
 \end{equation*}
where we have set $C= \sum  (1+|k|)^{-(n+1)}$. As the families $\cB_\alpha$ are bounded this shows that $R$ satisfies on $\cA_\theta$ the semi-norm estimates required for the continuity of a linear operator from $\cA_\theta'$ to $\cA_\theta$. As $\cA_\theta$ is dense in $\cA_\theta'$ it then follows that $R$ uniquely extends to a continuous linear operator from $\cA_\theta'$ to $\cA_\theta$, i.e., $R$ is a smoothing operator.  This shows that~(iii) implies~(i). The proof is complete. 
\end{proof}

\appendix

\section{Some Basic Properties of $A_\theta$ and $\cA_\theta$} \label{appendix:cAtheta-basic-properties} 
In this appendix, for the reader's convenience, we include proofs of Proposition~\ref{prop:NCTori.GNS-representation}, Proposition~\ref{prop:NCtori.cAtheta-Frechet}, Proposition~\ref{prop:NCtori.condition-cAtheta} and Proposition~\ref{prop:NCtori.invertibility-cAtheta}. 

As in Section~\ref{section:NCtori} we denote by $\cA_\theta^0$ the subalgebra generated by the unitary operators $U_1, \ldots, U_n$. 

\begin{proof}[Proof of Proposition~\ref{prop:NCTori.GNS-representation}]
 The proof is an elementary instance of the GNS construction~(see, e.g., \cite{Ar:Springer81}).  Let $u,v\in A_\theta$. Using the inequality $u^*u \leq \|u\|^2$ and the positivity of the functional $w\rightarrow \tau(v^*wv)$ we get 
 \begin{equation*}
 \acoup{uv}{uv} = \tau\left(uvv^*u^*\right) = \tau\left(v^*u^*uv\right) \leq \tau\left( v^*\|u\|^2 v\right) = \|u\|^2 \acoup{v}{v}.
\end{equation*}
In particular, we see that $\|uv\|_0\leq \|u\| \|v\|_0$ for all $u,v\in \cA_\theta^0$. Combining this with the density of $\cA_\theta^0$ in $A_\theta$ and $\cH_\theta$ we deduce that the multiplication of $\cA_\theta^0$ uniquely extends to a continuous bilinear map $A_\theta\times \cH_\theta \rightarrow \cH_\theta$. This gives rise to a representation of the algebra $A_\theta$ by bounded operators on $\cH_\theta$. This representation is unital. In addition, let $u\in A_\theta$. For all $v,w\in \cA_\theta^0$ we have 
\begin{equation*}
 \acoup{u^*v}{w}=\tau(u^*vw^*)=\tau(vw^*u^*)=\acoup{v}{uw}. 
\end{equation*}
 It then follows that $\acoup{u^*v}{w}=\acoup{v}{uw}$ for all $v,w\in \cH_\theta$, i.e, the action $u^*$ on $\cH_\theta$ is the adjoint of the action of $u$. Therefore, we see that we have a unital $*$-representation of the $C^*$-algebra  $A_\theta$ in $\cH_\theta$. As unital $*$-representations of $C^*$-algebras are isometries we obtain~(\ref{eq:NCtori.nrom0<norm-u}). This proves the first part of Proposition~\ref{prop:NCTori.GNS-representation}. 
 
Let $u\in \cA_\theta^0$. Then~(\ref{eq:NCtori.nrom0<norm-u}) implies that
\begin{equation*}
 \|u\|_0 = \|u 1\|_0\leq \|u\|. 
\end{equation*}
 Combining this with the density of $\cA_\theta^0$ in $A_\theta$ we then deduce that the inclusion of $\cA_\theta^0$ into $\cH_\theta$ uniquely extends to a continuous map from $A_\theta$ to $\cH_\theta$. This map assigns to each $u\in A_\theta$ the sum of its Fourier series $\sum u_k U^k$ in $\cH_\theta$. To complete the proof it remains to show that this map is one-to-one. 
 
Let  $u\in A_\theta$ be such that $u_k=0$ for all $k\in \Z^n$. Thus, for all $k\in \Z^n$, we have 
\begin{equation*}
\tau\left(u^*U^k\right) = \acoup{U^k}{u} = \overline{\acoup{u}{U^k}} = \overline{u_k} = 0 .
\end{equation*}
By linearity it then follows that $\tau[u^*v]=0$ for all $v\in \cA_\theta^0$. Therefore, for all $v,w\in \cA_\theta^0$, we have 
\begin{equation*}
 \acoup{w}{uv}=\tau(wv^*u^*)=\tau(u^*wv^*)=0.
\end{equation*}
Combining this with the density of $\cA_\theta^0$ in $\cH_\theta$ we deduce that $\acoup{w}{uv}=0$ for all $v,w\in \cH_\theta$. Using~(\ref{eq:NCtori.nrom0<norm-u}) we then obtain
\begin{equation*}
 \|u\|= \sup_{\|v\|_0=1}\|uv\|_0 =  \sup_{\|v\|_0=1}\sup_{\|w\|_0=1}\left| \acoup{w}{uv}\right| =0.
\end{equation*}
It then follows that we have a one-to-one map from $A_\theta$ into $\cH_\theta$. The proof is complete. 
\end{proof}

In what follows we equip each subalgebra $A_\theta^{(N)}$, $N\geq 1$, with the norm, 
\begin{equation*}
 \NormN{u}= \sup_{|\beta|\leq N} \|\delta^\beta(u)\|, \qquad u\in A_\theta^{(N)}. 
\end{equation*}
These norms generate the topology of $\cA_\theta$. Moreover, it follows from~(\ref{eq:NCtori-uv-Leibniz})--(\ref{eq:NCtori-involution-delta-compatibility}) that we have 
\begin{gather}
 \NormN{uv} \leq 2^N \NormN{u} \NormN{v} ,\qquad u,v\in A^{(N)}_\theta,\label{eq:cAtheta-appendix.submultiplicative}\\
 \NormN{u^*}=\NormN{u}, \qquad u\in A^{(N)}_\theta. \label{eq:cAtheta-appendix.isometric-involution}
\end{gather}

\begin{lemma} \label{lem:cATheta-appendix.A^N-Banach}
 $A^{(N)}_\theta$ is a (unital) Banach $*$-algebra.
\end{lemma}
\begin{proof}
 It follows from~(\ref{eq:cAtheta-appendix.submultiplicative})--(\ref{eq:cAtheta-appendix.isometric-involution}) that we have a continuous product and an isometric involution on $A_\theta^{(N)}$. Therefore, we only need to check that $\NormN{\cdot}$ is  a Banach space norm. Let $(u_\ell)_{\ell\geq 0}$ be a Cauchy sequence in $A_\theta^{(N)}$. This yields a sequence $(\alpha_s(u_\ell))_{\ell \geq 0}$ in the Banach algebra $C^N(\T^n; A_\theta)$ such that
\begin{align*}
 \sup_{s\in \R^n}\sup_{|\beta|\leq N} \left\| D_s^\beta \left( \alpha_s(u_{\ell'}) -\alpha_s(u_\ell)\right) \right\|  = 
 \sup_{|\beta|\leq N}\sup_{s\in \R^n}\left\|\alpha_s\left[ \delta^\beta(u_{\ell '}-u_\ell)\right] \right\|   = \NormN{u_{\ell '}-u_\ell}. 
\end{align*}
Thus, we obtain a Cauchy sequence in $C^N(\T^n; A_\theta)$, and so there is $v(s)\in C^N(\T^n; A_\theta)$ such that $\alpha_s(u_\ell)\rightarrow v(s)$ in $ C^N(\T^n; A_\theta)$. In particular, if we set $u=v(0)$, then $u_\ell=\alpha_s(u_\ell)|_{s=0}\rightarrow u$ in $A_\theta$. Therefore, $\alpha_s(u_\ell)\rightarrow \alpha_s(u)$ in $A_\theta$ for every $s\in \R^n$. This implies that $v(s)=\alpha_s(u)$ for all $s\in \R^n$. Incidentally, $\alpha_s(u)=v(s)\in C^N(\R^n; A_\theta)$, and so $u\in A_\theta^{(N)}$. More generally, as in~(\ref{eq:NCtori.convergence-u_l}), for every $\beta\in \N_0^n$, $|\beta|\leq N$, we have 
\begin{equation*}
 \|\delta^\beta(u_\ell -u)\| = \left\| D_s^\beta \left( \alpha_s(u_{\ell}) -\alpha_s(u)\right)|_{s=0} \right\|  =  
 \left\| D_s^\beta \left( \alpha_s(u_{\ell}) -v(s)\right)|_{s=0} \right\| \longrightarrow 0. 
\end{equation*}
Therefore, we see that $\NormN{u_\ell-u} \rightarrow 0$. This shows that $\NormN{\cdot}$ is  a Banach space norm. The proof is complete. 
\end{proof}

We are now in a position to prove Proposition~\ref{prop:NCtori.cAtheta-Frechet}. 

\begin{proof}[Proof of Proposition~\ref{prop:NCtori.cAtheta-Frechet}] 
As the topology of $\cA_\theta$ is generated by the norms $\NormN{\cdot}$, $N\geq 1$,  it follows from Lemma~\ref{lem:cATheta-appendix.A^N-Banach} that the product and involution of $\cA_\theta$ are continuous. 
Moreover, any Cauchy sequence in $\cA_\theta$ is a Cauchy sequence in  each of the Banach spaces $A_\theta^{(N)}$. Thus, it converges with respect to the all norms $\NormN{\cdot}$, and so it converges in $\cA_\theta$. Therefore, we see that $\cA_\theta$ is a Fr\'echet $*$-algebra. 

Given any $\beta\in \N_0^n$, the continuity of $\delta^\beta:\cA_\theta \rightarrow \cA_\theta$ and the continuity of the action of $\R^n$ on $A_\theta$ at least imply that $(s,u)\rightarrow \alpha_s[\delta^\beta(u)]$ is a continuous map from $\R^n\times \cA_\theta$ to $A_\theta$. Thus, given any multi-order $\beta$, as $(t,v)$ approaches $(s,u)$ in $\R^n\times \cA_\theta$ we have 
\begin{equation*}
 \left\|  \delta^\beta\left( \alpha_t(v)-\alpha_{s}(u)\right)\right\|  = \left\| \alpha_t\left[\delta^\beta(v)\right]-  \alpha_{s}\left[\delta^\beta(u)\right]\right\|  \longrightarrow 0.
\end{equation*}
It then follows that $\alpha_t(v)$ converges to $\alpha_{s}(u)$ in $\cA_\theta$ as $(t,v)$ approaches $(s,u)$ in $\R^n\times \cA_\theta$. This shows that $(s,u)\rightarrow \alpha_s(u)$ is a continuous map from $\R^n\times \cA_\theta$ to $\cA_\theta$. 

Given any $u\in A_\theta^{(1)}$ and $s,h\in \R^n$, set 
 $\Delta_h \alpha_s(u) =\alpha_{s+h}(u)-\alpha_s(u) -\sum_{j=1}^n ih_j\alpha_s\left[\delta_j(u)\right]$. As any $C^1$-map is differentiable (\cf\ Appendix~\ref{app:LCS-diff}), 
we see that $\|\Delta_h \alpha_s(u)\|=\op{o}(|h|)$ as $h\rightarrow 0$. Suppose that $u\in \cA_\theta$, and let $\beta \in \N_0^n$. Note that $\delta^\beta[ \Delta_h \alpha_s(u)]= \Delta_h \alpha_s[\delta^\beta(u)]$. As $\delta^\beta(u)\in A_\theta^{(1)}$ we see that, for all $\beta\in \N_0^n$, we have $\left\| \delta^\beta\left[ \Delta_h \alpha_s(u)\right]\right\| = \left\| \Delta_h \alpha_s\left[\delta^\beta(u)\right]\right\|  =\op{o}(|h|)$ as $h\rightarrow 0$. That is, $\Delta_h \alpha_s(u)=\op{o}(|h|)$ in $\cA_\theta$. It then follows that $s\rightarrow \alpha_s(u)$ is a differentiable map from $\R^n$ to $\cA_\theta$, and we have 
\begin{equation*}
D_{s_j}\alpha_s(u) = \alpha_s\left[\delta_j(u)\right], \qquad j=1,\ldots, n. 
\end{equation*}
Combining this with the continuity of the action of  $\R^n$ on $\cA_\theta$ shows that  $\alpha_s(u)\in C^1(\R^n;\cA_\theta)$. 
An induction then shows that $\alpha_s(u)\in C^N(\R^n;\cA_\theta)$ for every $N\geq 1$, and we have 
\begin{equation*}
 D_{s}^\beta \alpha_s(u) = \alpha_s\left[\delta^\beta (u)\right], \qquad |\beta|\leq N. 
\end{equation*}
It then follows that $\alpha_s(u)\in C^\infty(\R^n; \cA_\theta)$. This completes the proof of Proposition~\ref{prop:NCtori.cAtheta-Frechet}. 
\end{proof}

Let us now prove Proposition~\ref{prop:NCtori.condition-cAtheta}. 

\begin{proof}[Proof of Proposition~\ref{prop:NCtori.condition-cAtheta}] Let $u\in \cH_\theta$ have Fourier series $\sum u_k U^k$. Suppose that $(u_k)_{k \in \Z^n}\in \cS(\Z^n)$. Then, for all $\beta\in \N_0^n$, we have 
\begin{equation}\label{eq:cAtheta-appendix.Fourier-series-estimates}
 \sum_{k\in \Z^n} \left\| \delta^\beta(u_kU^k)\right\| =  \sum_{k\in \Z^n} |k^\beta u_k|\left\| U^k\right\| =  \sum_{k\in \Z^n} |k^\beta u_k|<\infty. 
\end{equation}
This implies that the Fourier series $\sum u_k U^k$ is normally convergent in each of the Banach spaces $A_\theta^{(N)}$. Therefore, it converges in each of those spaces, and so it converges in $\cA_\theta$. Note that its sum must be $u$ since Proposition~\ref{prop:NCTori.GNS-representation} implies that we have a continuous inclusion of $\cA_\theta$ into $\cH_\theta$, and hence $u\in\cA_\theta$. 

Conversely, suppose that $u\in \cA_\theta$. Let $k\in \Z^n$. Using~(\ref{eq:NCtori.nrom0<norm-u}) we get
\begin{equation} \label{eq:cAtheta.u_k-estimates}
|u_k|= \left| \acoup{u}{U^k}\right| \leq \| u\|_0 \| U^k\|_0  \leq \| u\|. 
\end{equation}
Moreover, using~(\ref{eq:NCtori.integration-by-parts}) we see that, for $j=1, \ldots, n$, we have
\begin{equation*}
 \acoup{\delta_j(u)}{U^k}= \tau\left[\delta_j(u) (U^k)^*\right]=- \tau\left[u \delta_j\left((U^k)^*\right)\right]. 
\end{equation*}
It follows from~(\ref{eq:NCtori.derivation-involution}) and~(\ref{eq:NCtori.delta-U^k}) that $ \delta_j((U^k)^*)= -(\delta_j(U^k))^*=-k_j (U^k)^*$. Thus, 
\begin{equation*}
 \acoup{\delta_j(u)}{U^k}=k_j\tau\left[u (U^k)^*\right]=k_j \acoup{u}{U^k}=k_ju_k. 
\end{equation*}
 An induction then shows that, for all $\beta\in \N_0^n$, we have $\acoup{\delta^\beta(u)}{U^k}=k^\beta u_k$. Combining this with~(\ref{eq:cAtheta.u_k-estimates}) we deduce that, for all $\beta\in \N_0^n$ and $k\in \Z^n$, we have 
 \begin{equation}
 |k^\beta u_k| = |\acoup{\delta^\beta(u)}{U^k}| \leq \|\delta^\beta(u)\|. 
 \label{eq:cAtheta-appendix.Fourier-series-estimates2}
\end{equation}
It then follows that  $(u_k)_{k \in \Z^n}\in \cS(\Z^n)$. Therefore, we see that $u$ belongs to $\cA_\theta$ if and only if the sequence $(u_k)_{k \in \Z^n}$ is contained in $\cS(\Z^n)$. 

It follows from all this that we have a linear isomorphism $\Phi: (u_k)\rightarrow \sum u_k U^k$ from $\cS(\Z^n)$ onto $\cA_\theta$. Given any $\beta \in \N_0^n$ it follows from~(\ref{eq:cAtheta-appendix.Fourier-series-estimates}) that, if $u=\sum u_k U^k$ with $(u_k)\in \cS(\Z^n)$, then we have 
\begin{equation*}
  \left\| \delta^\beta(u)\right\| \leq \sum_{k\in \Z^n} \left\| \delta^\beta(u_kU^k)\right\|=  \sum_{k\in \Z^n} |k^\beta u_k|. 
\end{equation*}
As $(a_k)\rightarrow  \sum |k^\beta a_k|$ is a continuous semi-norm on $\cS(\Z^n)$ we see that $\Phi$ is continuous. The continuity of $\Phi^{-1}$ is an immediate consequence of~(\ref{eq:cAtheta-appendix.Fourier-series-estimates2}). Therefore, we see that $\Phi$ is a linear homeomorphism. The proof of Proposition~\ref{prop:NCtori.condition-cAtheta} is complete.  
\end{proof}

Finally, we  prove Proposition~\ref{prop:NCtori.invertibility-cAtheta}. 

\begin{proof}[Proof of Proposition~\ref{prop:NCtori.invertibility-cAtheta}] 
 If  $u\in \cA_\theta$ is invertible in $A_\theta$, then $\alpha_s(u^{-1})=\alpha_s(u)^{-1}$ is a smooth map from $\R^n$ to $A_\theta$, and so $u^{-1}\in \cA_\theta$. Therefore, we see that $A_\theta^{-1} \cap \cA_\theta =\cA_\theta^{-1}$. As $A_\theta^{-1}$ is an open set of $A_\theta$ and the inclusion of $\cA_\theta$ into $A_\theta$ is continuous, it then follows that $\cA_\theta^{-1}$ is an open set of $\cA_\theta$.
 
Given any $N\geq 1$, the inverse map of the invertible group of $A_\theta^{(N)}$ is continuous, since $A_\theta^{(N)}$ is a Banach algebra. 
This implies that the inverse map of $\cA_\theta^{-1}$ is continuous with respect to each of the norms $\NormN{\cdot }$, $N\geq 1$. As these norms generate the topology of $\cA_\theta$, we obtain the continuity with respect to the $\cA_\theta$-topology. The proof is complete. 
 \end{proof}

\section{Integration in Locally Convex Spaces} \label{app:LCS-int}
In this appendix, we review the integration of maps with values in locally convex spaces. 

\subsection{Riemann integration}\label{subsec:Riemann} 
There is no major difficulty to extend Riemann's integral to maps with values in locally convex spaces (see, e.g., \cite{FJ:JDE68, Ha:BAMS82}). In what follows we assume that $E$ is a (Hausdorff) locally convex space, and we let $I$ be a closed bounded cube in $\R^d$, $d\geq 1$. 

A \emph{step map} $f:I\rightarrow E$ is of the form, 
\begin{equation}
 f(t)= \sum_{1\leq j \leq m} \mathds{1}_{I_j}(t)\xi_j, \qquad t\in I,
 \label{eq:LCS.step-maps}
\end{equation}
where $\xi_1, \ldots, \xi_m$ are elements of $E$ and $I_1, \ldots, I_m$ are mutually disjoint cubes such that $I=\bigcup I_j$. We denote by $\cR_0(I; E)$ the vector space of step maps. In addition, we endow the space of maps $f:I\rightarrow E$ with the semi-norms, 
\begin{equation}
 f \longrightarrow \sup_{t\in I} p\left[ f(t)\right],
 \label{eq:LCS.semi-norms-regulated-maps}
\end{equation}
where $p$ ranges over continuous semi-norms on $E$. 

\begin{definition}
 $\cR(I;E)$ is the closure of $\cR_0(I; E)$ with respect to the topology defined by the semi-norms~(\ref{eq:LCS.semi-norms-regulated-maps}). The maps in $\cR(I;E)$ are called Riemann-integrable maps. 
\end{definition}

\begin{remark}
Suppose that $f:I\rightarrow E$ is a Riemann-integrable map. Then
\begin{itemize}
 \item[-] For every continuous semi-norm $p$ on $E$, the function $p\circ f:I\rightarrow \R$ is Riemann-integrable (i.e., it is the uniform limit of step functions). 
 
 \item[-] Given any continuous $\R$-linear map $\Phi$ from $E$ to some locally convex space $F$, the map $\Phi\circ f:I\rightarrow F$ is Riemann-integrable. 
\end{itemize}
\end{remark}

As $I$ is a compact cube, any continuous map $f:I\rightarrow E$ is uniformly continuous. This fact allows us to prove the following result. 

\begin{proposition}
 Every continuous map $f:I\rightarrow E$ is Riemann-integrable. 
\end{proposition}

If $f(t)=\sum_{j=1}^m  \xi_j\mathds{1}_{I_j}(t)$ is a step map, then we define its integral by 
\begin{equation}
 \int_I f(t)dt= \sum_{1\leq j \leq m} \op{Vol}(I_j) \xi_j.
 \label{eq:LCS.Riemann integral-step-map}
\end{equation}
Note this does not depend on the representation~(\ref{eq:LCS.step-maps}) of $f$. In addition, for every continuous semi-norm $p$ on $E$, we have 
\begin{equation*}
 p\left( \int_I f(t)dt\right) \leq \sum_{1\leq j \leq m} \op{Vol}(I_j) p(\xi_j)=\op{Vol}(I) \sup_{t\in I} p\left[ f(t)\right]. 
\end{equation*}
This shows that the above notion of integral defines a linear map from $\cR_0(I;E)$ to $E$. As $\cR_0(I;E)$ is dense in $\cR(I;E)$ we arrive at the following result. 

\begin{proposition}\label{prop:LCS.Riemann integral}
 The integral~(\ref{eq:LCS.Riemann integral-step-map}) uniquely extends to a continuous linear map $f\rightarrow \int_I f(t)dt$ from $\cR(I;E)$ to $E$. 
\end{proposition}

We mention the following properties of the Riemann integral. 

\begin{proposition}\label{prop:LCS.properties-Riemann-int}
 Let $f:I\rightarrow E$ be a Riemann-integrable map. 
 \begin{enumerate}
    \item For every continuous semi-norm $p$ on $E$, we have 
             \begin{equation}
                    p\left( \int_I f(t) dt\right) \leq \int_I p\left[ f(t)\right] dt.
                    \label{eq:LCS.semi-norm-Riemann integral}
             \end{equation}
             
     \item Let $\Phi$ be a continuous $\R$-linear map from  $E$ to some locally convex space $F$. Then we have 
           \begin{equation}
                      \Phi\left( \int_I f(t)dt\right) =  \int_I \Phi \circ f(t)dt. 
                       \label{eq:LCS.Phi-Riemann integral}
           \end{equation}
 \end{enumerate}
\end{proposition}
\begin{proof}
It is straightforward to check~(\ref{eq:LCS.semi-norm-Riemann integral})--(\ref{eq:LCS.Phi-Riemann integral}) when $f$ is a step map. The extension of these results to all Riemann-integrable maps then follows by using the density of step maps among Riemann-integrable maps. 
\end{proof}

It follows from the Hahn-Banach theorem that $E'$ separates the points of $E$. Therefore, by specializing~(\ref{eq:LCS.Phi-Riemann integral}) to elements in $E'$ we obtain the following characterization of the Riemann integral. 

\begin{corollary}\label{cor:LCS.varphi-Riemann integral}
 Let $f:I\rightarrow E$ be a Riemann-integrable map. Then $\int_I f(t)dt$ is the unique element of $E$ such that
 \begin{equation*}
   \varphi\left( \int_I f(t)dt\right) =  \int_I \varphi \circ f(t)dt \qquad \forall \varphi \in E'. 
\end{equation*}
\end{corollary}

Let $C^0_c(\R^d;E)$ be the space of compactly supported continuous maps $f:\R^d \rightarrow E$. Given any map $f\in C^0_c(\R^d;E)$ its integral is defined by 
\begin{equation*}
 \int f(x) dx := \int_I f(x) dx,
\end{equation*}
where $I$ is any bounded closed cube whose interior contains the support of $f$. The value of $ \int_I f(x) dx$ is independent of the choice of $I$. Bearing this in mind we have the following change of variable formula. 

\begin{proposition}\label{prop:LCS.change-variable}
 Suppose that $\phi: \R^d \rightarrow \R^d$ is a $C^1$-diffeomorphism. Then, for every map $f\in C^0_c(\R^d;E)$, we have 
 \begin{equation*}
 \int f(x) dx = \int f\left[ \phi(x) \right] \left| \det\left[ \phi'(x)\right] \right| dx. 
\end{equation*}
\end{proposition}
\begin{proof}
 Let $f\in  C^0_c(\R^d;E)$. As $\phi$ is a continuous proper map, the composition $f\circ \phi:\R^d \rightarrow E$ is continuous and has compact support. Let $\psi \in E'$. The function $\psi \circ f$ is continuous and has compact support. As $\phi$ is a $C^1$-diffeomorphism, we get 
 \begin{equation*}
 \int \psi\circ f (x) dx=  \int \psi \circ f\left[ \phi(x) \right] \left| \det\left[ \phi'(x)\right] \right| dx. 
\end{equation*}
Combining this with Corollary~\ref{cor:LCS.varphi-Riemann integral} gives the result.  
\end{proof}

\subsection{Lebesgue integration}\label{subsec:Lebesgue}
The integrals of Bochner~\cite{Bo:FM35} and Gel'fand-Pettis~\cite{Ge:CISMK, Pe:TAMS38} are natural extensions of Lebesgue's integral to maps with values in Banach spaces. We refer to~\cite{Gr:MAMS55} for an extension of Bochner's integral to maps with values in Fr\'echet spaces. Following Thomas~\cite{Th:TAMS75} a natural setting for the integration of maps with values in locally convex spaces is provided by quasi-complete Suslin locally convex spaces.

Recall that a locally convex space is \emph{quasi-complete} when all closed bounded sets are complete (i.e., every bounded Cauchy net is convergent). A \emph{Suslin locally convex space} is a Hausdorff locally convex space which is the image by a continuous map of a separable complete metric space. The following types of locally convex spaces are quasi-complete Suslin spaces (see~\cite{DeW:MSRSLC69, Sc:Tata73, Va:TLCS82}):
\begin{itemize}
  \item Separable Fr\'echet spaces (including nuclear Fr\'echet  spaces) and their weak duals. 
  
  \item  Inductive limits of sequences of separable Fr\'echet spaces and weak duals of such spaces. 
  
  \item Strong duals of separable Fr\'echet-Montel spaces and of inductive limits of sequences of such spaces.
  
   \item The space $\cL(E,F)$ equipped with the compact convergence topology (or any weaker topology), when $F$ is a separable Fr\'echet space and $
   E$ is a separable Fr\'echet space or a countable inductive limit of a sequence of such spaces. 
 
   \item The space $\cL(E,F)$ equipped with the bounded convergence topology (a.k.a.\ strong dual topology), when $F$ is a separable Fr\'echet space and $E$ is a separable Fr\'echet-Montel space or an inductive limit of a sequence of such spaces. 
 \end{itemize}
The most common examples of locally convex spaces are quasi-complete Suslin spaces. In particular, the smooth noncommutative torus $\cA_\theta$, its strong dual $\cA_\theta'$, and the space $\cL(\cA_\theta)$ are such spaces. 

\begin{remark}
 In~\cite{DeW:MSRSLC69} Suslin spaces are defined in terms of existence of a sieve. This is equivalent to the usual definition of a Suslin space. In fact, as every separable complete metric space admits a sieve (see~\cite[Appendix]{Do:Springer84}) and the image of a sieve by a continuous map is a sieve,  any Suslin space admits a sieve. Conversely, let  $\N^\N$ be the null Baire space equipped with its standard metric. This is a complete separable metric space (see, e.g.,~\cite[Appendix]{Do:Springer84}). Then the datum of a sieve on a topological space $E$ precisely allows us to construct a surjective continuous map from $\N^\N$ onto $E$, and so $E$ is a Suslin space (see~\cite[Appendix]{Do:Springer84}). 
\end{remark}

In what follows, we let $(X, \fS, \mu)$ be a $\sigma$-finite measured space, and we assume that $E$ is a quasi-complete Suslin locally convex space. In this setting a map $f:X\rightarrow E$ is measurable when $f^{-1}(B)\in \fS$ for every Borel set $B\subset E$.

\begin{proposition}[\cite{Th:TAMS75}]\label{prop:LCS.measurability} 
 A map $f:X\rightarrow E$ is measurable if and only if, for every $\varphi \in E'$, the function $\varphi \circ f:X\rightarrow \C$ is measurable. 
\end{proposition}

When $E$ is metrizable it is well known that any pointwise limit of measurable maps from $X$ to $E$ is measurable. As a consequence of  Proposition~\ref{prop:LCS.measurability} this fact continues to hold for non-metrizable quasi-complete Suslin  locally convex spaces. 

\begin{corollary}\label{cor:LCS.measurability-pointwise-limit}
 Any pointwise limit of measurable maps from $X$ to $E$ is measurable. 
\end{corollary}
\begin{proof}
 Let $(f_\ell)_{\ell \geq 0}$ be a sequence of measurable maps from $X$ to $E$ that converges pointwise to $f(x)$. For every $\varphi\in E'$ the function $\varphi\circ f(x)$ is the pointwise limit of the measurable functions $\varphi\circ f_\ell(x)$, $\ell\geq 0$, and hence it is measurable. Combining this with Proposition~\ref{prop:LCS.measurability} then shows that $f:X\rightarrow E$ is a measurable map. The proof is complete. 
\end{proof}

\begin{definition}\label{def:LCS.integrability}
 A map $f:X\rightarrow E$ is integrable if and only if, for every continuous semi-norm $p$ on $E$, we have 
 \begin{equation*}
 \int_X p\left[ f(x)\right] d\mu(x)<\infty. 
\end{equation*}
\end{definition}

\begin{remark}
 If $\mu(X)<\infty$, then every measurable map with bounded range is integrable. In particular, when $X$ is a compact (Hausdorff) topological space and $\mu$ is a (finite) Borel measure, then every continuous map $f:X\rightarrow E$ is integrable. 
\end{remark}

\begin{remark}
 When $E$ is nuclear, it can be shown that a measurable map $f:X\rightarrow E$ is integrable if and only if it is weakly integrable,  i.e., $\int_X |\varphi \circ f(x)|d\mu(x)<\infty$ for every $\varphi \in E'$ (see~\cite{Gr:MAMS55, Th:TAMS75}). 
\end{remark}

We shall denote by $L^1_\mu(X;E)$ the space of integrable maps $f:X\rightarrow E$ modulo the relation $f=g$ a.e.. We will often identify elements of $L^1_\mu(X;E)$ with their representatives. In addition, we equip $L^1_\mu (X; E)$ with the topology defined by the semi-norms, 
\begin{equation*}
 f \longrightarrow \int_X p\left[ f(x)\right] d\mu(x), 
\end{equation*}
where $p$ ranges over continuous semi-norms on $E$. This turns $L^1_\mu (X; E)$  into a Hausdorff locally convex space. 

The integrals of integrable maps  $f:X\rightarrow E$ is defined as follows. 

\begin{proposition}[\cite{Th:TAMS75}]\label{prop:LCS.Lebesgue-integral} 
Let $f:X\rightarrow E$ be an integrable map. Then there is a unique element $\int_X f(x) d\mu(x)\in E$ such that, for every $\varphi \in E'$, we have 
\begin{equation}
 \varphi \left[ \int_X f(x) d\mu (x) \right] = \int_X \varphi \circ f(x) d\mu (x). 
 \label{eq:LCS.int-vphi}
\end{equation}In addition, for every continuous semi-norm $p$ on $E$, we have 
\begin{equation}
 p\left[ \int_X f(x)d\mu (x)\right] \leq  \int_X p\left[ f(x)\right] d\mu(x). 
 \label{eq:LCS.semi-norm-integral}
\end{equation}
In particular, the integration $f \rightarrow \int_X f(x) d\mu(x)$ defines a continuous linear map from $L^1_\mu(X;E)$ to $E$.
\end{proposition}

\begin{remark}
 When $X$ is a closed bounded cube $I\subset \R^d$ and $\mu$ is the Lebesgue measure on $I$, it follows from Corollary~\ref{cor:LCS.varphi-Riemann integral} and Corollary~\ref{cor:LCS.measurability-pointwise-limit} that every Riemann-integrable map $f:I\rightarrow E$ is integrable and its Riemann integral agrees with the integral provided by Proposition~\ref{prop:LCS.Lebesgue-integral}.   
\end{remark}

\begin{remark}
 Let $f:X\rightarrow E$ be an integrable map. It follows from~(\ref{eq:LCS.int-vphi}) that, for all $\varphi \in E'$, we have 
 \begin{equation*}
 \Re \varphi \left[ \int_X f(x) d\mu (x) \right] = \Re \left( \int_X \varphi \circ f(x) d\mu (x)\right) =  \int_X \Re \left[\varphi \circ f(x)\right] d\mu (x).  
\end{equation*}
Let $E'_\R$ be the space of continuous $\R$-linear forms on $E$. Any $\psi \in E'_\R$ is of the form $\psi =\Re \varphi$ for a unique $\varphi \in E'$ (see, e.g., \cite{Co:Springer90}). Therefore, we see that
\begin{equation}
 \psi \left[ \int_X f(x) d\mu (x) \right] = \int_X \psi \circ f(x) d\mu (x) \qquad \text{for all $\psi \in E'_\R$}. 
 \label{eq:LCS.int-psi} 
\end{equation}
\end{remark}

\begin{example}
 A map $f:X\rightarrow E$ is called \emph{simple}  when it takes the form $f=\sum_{j=1}^m \xi_j \mathds{1}_{A_j}$ with $\xi_j\in E$ and $A_j \in \fS$ such that $\mu(A_j)<\infty$. Such a map is integrable and we have
 \begin{equation*}
  \int_X f(x) d\mu(x)= \sum_{1\leq j \leq m} \mu(A_j) \xi_j.  
\end{equation*}
 \end{example}

\begin{remark}
 It was proved by Grothendieck~\cite{Gr:MAMS55} that, when $E$ is a separable Fr\'echet space, for every integrable map $f:X\rightarrow E$ there is a sequence of simple 
 maps $(f_\ell)_{\ell \geq 0}$ such that $f_\ell(x)\rightarrow f(x)$ a.e.\ and $\int_X f_\ell(x)d\mu(x) \rightarrow  \int_X f(x)d\mu(x)$. This result also holds for Suslin locally convex spaces with quasi-complete nuclear barrelled preduals (see~\cite{Th:TAMS75}). 
\end{remark}

\begin{proposition} \label{prop:LCS.Phi-integral}
Suppose that $\Phi:E\rightarrow F$ is a continuous $\R$-linear map from $E$ to some quasi-complete Suslin locally convex space $F$. Let $f:X\rightarrow E$ be an integrable map.  Then the map $\Phi \circ f:X\rightarrow F$ is integrable, and we have 
\begin{equation}
 \Phi \left( \int_X f(x) d\mu(x) \right) =  \int_X \Phi \circ f(x) d\mu(x). 
 \label{eq:LCS.Phi-Lebesgue-integral}
\end{equation}
\end{proposition}
\begin{proof}
As $\Phi$ is continuous, the composition $\Phi \circ f:X\rightarrow F$ is a measurable map.   Let $q$ be a continuous semi-norm on $F$. As $\Phi$ is a continuous $\R$-linear map, there is a continuous semi-norm $p$ on $E$ such that $q[\Phi(\xi)]\leq p(\xi)$ for all $\xi\in E$. Combining this with the integrability of $f$ gives 
 \begin{equation*}
 \int_X q\left[ \Phi\circ f(x)\right] d\mu(x) \leq  \int_X p\left[f(x)\right] d\mu(x) <\infty.
\end{equation*}
It then follows that the map $\Phi\circ f:X\rightarrow F$ is integrable. 

Let $\psi \in F'_\R$. Then $\psi\circ \Phi \in E'_\R$, and so by using~(\ref{eq:LCS.int-psi}) we get 
\begin{equation*}
 \psi \circ \Phi  \left( \int_X f(x) d\mu(x) \right)=  \int_X \psi \left[ \Phi\circ f (x)\right] d\mu(x)= 
 \psi \left( \int_X \Phi\circ f(x) d\mu(x) \right). 
\end{equation*}
As $F'_\R$ separates the points of $F$ by the real version of the Hahn-Banach theorem, we deduce that 
$ \Phi ( \int_X f(x) d\mu(x)) =  \int_X \Phi \circ f(x) d\mu(x)$. The proof is complete. 
\end{proof}

We also have the following version of the dominated convergence theorem. 

\begin{proposition}\label{prop:LCS.DCT}
 Let $(f_\ell)_{\ell\geq 0}\subset L^1_\mu(X;E)$ be a sequence such that
 \begin{enumerate}
\item[(i)] $f_\ell(x) \rightarrow f(x)$ for almost every $x\in X$.  

\item[(ii)] For every continuous semi-norm $p$ on $E$, there is a function $g_p\in L^1_\mu(X)$ such that, for almost every $x\in X$, we have
\begin{equation*}
 p[f_\ell(x)]\leq g_p(x) \qquad \forall \ell \geq 0. 
\end{equation*}
\end{enumerate}
Then, we have
\begin{equation*}
f_\ell \longrightarrow f \ \text{in $L^1_\mu(X;E)$} \qquad  \text{and} \qquad  \int_X f_\ell(x) d\mu(x)\longrightarrow \int_X f(x) d\mu(x). 
\end{equation*}
\end{proposition}
\begin{proof}
It follows from (i) and Corollary~\ref{cor:LCS.measurability-pointwise-limit} that $f(x)$ is a measurable map. Let $p$ be a continuous semi-norm on $E$. Using (i)--(ii) and the continuity of $p$ we see that  $p[f(x)]\leq g_p(x)$ a.e., and so $\int_X p[f(x)]d\mu(x) \leq \int_X g_p(x) d\mu(x)<\infty$.
In addition, we have $p[f_\ell(x)-f(x)]\rightarrow 0$ a.e.\ and $p[f_\ell(x)-f(x)]\leq 2g_p(x)$ a.e.. Therefore, the dominated convergence theorem ensures us that $\int_X p[f_\ell(x)-f(x)]d\mu(x) \rightarrow 0$. This shows that $f(x)$ is in $L^1_\mu(X;E)$ and $f_\ell \rightarrow f$ in $L^1_\mu(X;E)$. The continuity of the integral on $L^1_\mu(X;E)$ then implies that $ \int_X f_\ell(x) d\mu(x)\rightarrow \int_X f(x) d\mu(x)$. The proof is complete. 
\end{proof}

Suppose that there are $\sigma$-finite measured spaces $(X_1,\fS_1, \mu_1)$ and $(X_2,\fS_2, \mu_2)$ such that $X=X_1\times X_2$ and $(\fS,\mu)$ is the product measure $(\fS_1\otimes \fS_2, \mu_1\otimes \mu_2)$ or is its completion when $(\fS_1, \mu_1)$ and $(\fS_2, \mu_2)$ are complete measures. We have the following version of Fubini's theorem. 

\begin{proposition}[\cite{Th:TAMS75}]\label{prop:LCS.Fubini}
 Suppose that $E$ is a separable Fr\'echet space. Let $f:X_1\times X_2 \rightarrow E$ be a $\mu$-integrable map. 
 \begin{enumerate}
  \item $f(x_1,\cdot) \in L^1_{\mu_2}(X_2; E)$ for almost every $x_1\in X_1$. 
      \item $f(\cdot, x_2) \in L^1_{\mu_1}(X_1; E)$ for almost every $x_2\in X_2$. 
  \item $\int_{X_2} f(x_1,x_2)d\mu_2(x_2)\in  L^1_{\mu_1}(X_1; E)$ and $\int_{X_1} f(x_1,x_2)d\mu_1(x_1)\in  L^1_{\mu_2}(X_2; E)$. 
  \item We have 
  \begin{align*}
\int_X f(x_1,x_2) d\mu(x_1,x_2)  & =  \int_{X_1} \left( \int_{X_2} f(x_1,x_2)d\mu_2(x_2)\right) d\mu_1(x_1)\\
&  = \int_{X_2} \left( \int_{X_1} f(x_1,x_2)d\mu_1(x_1)\right) d\mu_2(x_2).  
\end{align*}
\end{enumerate}
\end{proposition}

\begin{remark}
Proposition~\ref{prop:LCS.Fubini} holds \emph{verbatim} for Suslin locally convex spaces with quasi-complete nuclear barrelled preduals (\emph{cf.}~\cite{Th:TAMS75}). It also holds for general quasi-complete Suslin locally convex spaces providing we further assume that the map $f$ is ``totally integrable" 
(see~\cite{Th:TAMS75} for the precise statement).
\end{remark}

\section{Differentiable Maps with Values in Locally Convex Spaces}\label{app:LCS-diff} 
In this appendix, we review differentiable maps on an Euclidean open set with values in locally convex spaces. This includes results about differentiation under the integral and a description of the Fourier transform in this setting.  

\subsection{Differentiation}  
In what follows we assume that $E$ is a (Hausdorff) locally convex space, and we let $U$ be an open subset of $\R^d$, $d\geq 1$. 

\begin{definition}
Given any open neighborhood $V$ of the origin in $\R^d$, we shall say that a map $\varepsilon:V\rightarrow E$ is $\op{o}(|h|)$ near $h=0$ when, for every continuous semi-norm $p$ on $E$, the function  $p(\varepsilon(h))$ is $\op{o}(|h|)$ near $h= 0$. 
\end{definition}

\begin{definition}
 We say that a map $f:U\rightarrow E$ is \emph{differentiable} at a given point $a\in U$ when there is a (continuous) $\R$-linear map $Df(a):\R^d \rightarrow E$ such that
 \begin{equation*}
 f(a+h)=f(a) + Df(a)h + \op{o}(|h|) \qquad \text{near $h=0$}. 
\end{equation*}
We say that $f$ is \emph{differentiable} on $U$ when it is differentiable at every point of $U$. 
\end{definition}
 
\begin{remark}
 If $f$ is differentiable at $a\in U$, then $f$ is continuous at the point $a$. 
\end{remark}

\begin{remark}
 Suppose that $\Phi:E\rightarrow F$ is a continuous $\R$-linear map from $E$ to some locally convex space $F$. If $f$ is differentiable at some point $a\in U$, then $\Phi \circ f$ is differentiable at $a$ as well, and we have $D(\Phi\circ f)(a)= \Phi\circ [ Df(a)]$. 
\end{remark}

In the following we let $(e_1,\ldots, e_d)$ be the canonical basis of $\R^d$. 

\begin{definition}
 For $j=1,\ldots, d$, the partial derivative with respect to $x_j$ at a given point $a\in U$ of a map $f:U\rightarrow E$  is defined by 
 \begin{equation*}
 \partial_{x_j} f(a)=\lim_{t\rightarrow 0} \frac{1}{t} \left( f(a+te_j) -f(a)\right),
\end{equation*}
whenever the limit exists. 
\end{definition}

\begin{remark}
 If $f$ is differentiable at $a \in U$, then all the partial derivatives $\partial_{x_1}f(a), \ldots, \partial_{x_d}f(a)$ exist and we have $\partial_{x_j}f(a)=Df(a)e_j$ for $j=1,\ldots, d$. 
\end{remark}

\begin{remark}\label{rmk:LCS.partial-derivatives-Phi}
 Suppose that $\Phi:E\rightarrow F$ is a continuous $\R$-linear map from $E$ to some locally convex space $F$. If $\partial_{x_j}f(a)$ exists, then $\partial_{x_j} (\Phi \circ f)(a)$ exists as well, and we have 
 \begin{equation}
 \partial_{x_j} (\Phi \circ f)(a)=\Phi[\partial_{x_j}f(a)]. 
 \label{eq:LCS.partial-derivatives-Phi}
\end{equation}
\end{remark}

\begin{definition}
 We say that a map $f:U\rightarrow E$ is $C^1$ when the partial derivatives $\partial_{x_1}f(x), \ldots, \partial_{x_d}f(x)$ exist for all $x\in U$ and define continuous maps from $U$ to $E$.
\end{definition}

\begin{lemma}\label{lem:LCS.C1-Phi}
 Assume that $f:U\rightarrow E$ is a $C^1$-map. Let $\Phi:E\rightarrow F$ be a continuous $\R$-linear map from $E$ to some locally convex space $F$. 
 Then $\Phi \circ f:U\rightarrow F$ is a $C^1$-map  and its partial derivatives are given by~(\ref{eq:LCS.partial-derivatives-Phi}). 
\end{lemma}
\begin{proof}
 It follows from Remark~\ref{rmk:LCS.partial-derivatives-Phi} that $\partial_{x_1}[\Phi \circ f](x), \ldots, \partial_{x_d}[\Phi \circ f](x)$ exist for all $x\in U$ and are given by~(\ref{eq:LCS.partial-derivatives-Phi}). As $\Phi$ and $\partial_{x_1}f, \ldots, \partial_{x_d}f$ are continuous maps, we further see that the partial derivatives of $\Phi \circ f$ are continuous, i.e., $\Phi \circ f$ is a $C^1$-map. 
\end{proof}
 
\begin{lemma}\label{lem:LCS.C1-differentiable}
 Suppose that $f:U\rightarrow E$ is a $C^1$-map. 
\begin{enumerate}
\item[(i)] Let $x\in U$ and $\delta>0$ be such that $B(x,\delta)\subset U$. Then, for $|h|<\delta$, we have 
\begin{equation}
 f(x+h)= f(x) + \sum_{1\leq j \leq d} h_j \int_0^1 \partial_{x_j} f(x+th)dt.
 \label{eq:LCS.1st-order-Taylor} 
\end{equation}
  
 \item[(ii)] $f$ is differentiable on $U$ and, for all $x\in U$, we have 
       \begin{equation}
             Df(x)h= \sum_{1\leq j \leq d} h_j \partial_{x_j}f(x), \qquad h\in \R^d.
             \label{eq:LCS.differential-partial} 
       \end{equation}
   In particular, $f$ is continuous on $U$. 
\end{enumerate}
\end{lemma}
\begin{proof}
 Let $x\in U$ and $\delta>0$ be such that $B(x,\delta)\subset U$. In addition, let $\varphi \in E'$. By Lemma~\ref{lem:LCS.C1-Phi} the function $\varphi\circ f$ is $C^1$ and we have $\partial_{x_j} (\varphi \circ f)(x)=\varphi (\partial_{x_j}f(x))$ for all $x\in U$. Therefore, the Taylor formula at order~$1$ ensures us that for $|h|<\delta$ we have
 \begin{align*}
 \varphi\left[ f(x+h)\right] -  \varphi\left[ f(x)\right] & = \sum_{1\leq j \leq d} h_j \int_0^1 \varphi\left[\partial_{x_j} f(x+th)\right]dt \\
& =   \sum_{1\leq j \leq d} h_j  \varphi\left(  \int_0^1 \partial_{x_j} f(x+th)dt\right). 
\end{align*}
As $E'$ separates the points of $E$, we obtain~(\ref{eq:LCS.1st-order-Taylor}). 

It follows from~(\ref{eq:LCS.1st-order-Taylor}) that, for $|h|<\delta$, we have 
\begin{equation*}
 f(x+h)-f(x)-\sum_{1\leq j \leq d} h_j \partial_{x_j}f(x)= \sum_{1\leq j \leq d} h_j \int_0^1 \left[\partial_{x_j} f(x+th)- \partial_{x_j} f(x)\right]dt.
\end{equation*}
In order to prove that $f$ is differentiable at $x$ and its differential $Df(x)$ is given by~(\ref{eq:LCS.differential-partial}) it is enough to show that, for $j=1,\ldots, d$, we have 
\begin{equation}
 \lim_{h\rightarrow 0}  \int_0^1 \left[\partial_{x_j} f(x+th)- \partial_{x_j} f(x)\right]dt =0. 
 \label{eq:LCS.rate-of-change-integral}
\end{equation}
To see this let $p$ be a continuous semi-norm on $E$. In addition, let $\epsilon >0$. As $\partial_{x_j}f$ is continuous at $x$, there is $\delta'\in (0, \delta)$ such that $p(\partial_{x_j}f(x')-\partial_{x_j}f(x))\leq \epsilon$ whenever $|x'-x|\leq \delta'$. Combining this with~(\ref{eq:LCS.semi-norm-integral}) we deduce that, for $|h|\leq \delta'$, we have 
\begin{equation*}
 p\left( \int_0^1 \left[\partial_{x_j} f(x+th)- \partial_{x_j} f(x)\right]dt\right) \leq \int_0^1 p\left[\partial_{x_j} f(x+th)- \partial_{x_j} f(x)\right]dt\leq \epsilon. 
\end{equation*}
This proves~(\ref{eq:LCS.rate-of-change-integral}) and completes the proof. 
\end{proof}

\begin{definition}
 We say that  a map $f:U\rightarrow E$ is $C^N$, $N\geq 1$, when all the partial derivatives $\partial_{x_{j_1}} \cdots \partial_{x_{j_l}}f(x)$ of order~$\leq N$ exist at every point $x\in U$ and give rise to continuous maps from $U$ to $E$. 
\end{definition}

\begin{remark}
For a map $f:U\rightarrow E$ to be $C^N$ it is enough to assume the existence of all partial derivatives of order~$\leq N$ and require those of order~$N$ to be continuous. By Lemma~\ref{lem:LCS.C1-differentiable} this implies the continuity of the partial derivatives of order~$\leq N-1$.
\end{remark}

\begin{lemma}\label{lem:LCS.C2-partial-derivatives}
 Let $f:U\rightarrow E$ be a $C^N$-map, $N\geq 2$. Then, for every $(j_1, \ldots, j_N)\in \{1,\ldots, d\}^N$ and every permutation $(j_1', \ldots, j_N')$ of  $(j_1, \ldots, j_N)$
 we have 
 \begin{equation*}
 \partial_{x_{j_1'}} \cdots \partial_{x_{j_N'}}f=\partial_{x_{j_1}} \cdots \partial_{x_{j_N}}f.   
\end{equation*}
 \end{lemma}
\begin{proof}
Let $\varphi \in E'$. We know by Lemma~\ref{lem:LCS.C1-Phi} that $\varphi \circ f$ is a $C^1$-function and $\partial_{x_i}(\varphi \circ f)= \varphi\circ( \partial_{x_i}f)$ for $i=1,\ldots, d$. As $\partial_{x_i}f$ is a $C^1$-map this also implies that $\partial_{x_i}(\varphi \circ f)$ is $C^1$ and $\partial_{x_j} \partial_{x_i}(\varphi \circ f)= \partial_{x_j}\varphi\circ( \partial_{x_i}f)= 
 \varphi\circ(\partial_{x_j} \partial_{x_i}f)$. In particular, the function $\varphi \circ f$ is $C^2$. Thus, for all $x\in U$, we have 
 \begin{equation*}
 \varphi\circ\left[\partial_{x_j} \partial_{x_i}f(x)\right] = \partial_{x_j}\partial_{x_i}(\varphi \circ f)(x)= \partial_{x_i}\partial_{x_j}(\varphi \circ f)(x)=  \varphi\circ\left[\partial_{x_i} \partial_{x_j}f(x)\right].
\end{equation*}
 As $E'$ separates the points of $E$, we see that $\partial_{x_j} \partial_{x_i}f= \partial_{x_i}\partial_{x_j}f$ for $i,j=1,\ldots, d$. This proves the result for $N=2$. An induction then proves the result for $N\geq 3$. The proof is complete.
\end{proof}

In what follows, given any $C^N$-map $f:U\rightarrow E$  and any multi-order $\alpha \in \N_0^d$, $|\alpha|\leq N$, we set 
\begin{equation*}
\partial_x^\alpha f(x) = \partial_{x_1}^{\alpha_1} \cdots \partial_{x_d}^{\alpha_d} f(x), \qquad x\in U. 
\end{equation*}
We make the convention that $\partial_x^\alpha f(x)=f(x)$ when $\alpha =0$. By using Lemma~\ref{lem:LCS.C1-Phi} and arguing by induction we obtain the following statement. 

\begin{proposition}\label{prop:LCS.CN-Phi}
 Suppose that $\Phi: E\rightarrow F$ is a continuous $\R$-linear map from $E$ to some locally convex space $F$.  Let $f:U\rightarrow E$ be a $C^N$-map, $N\geq 1$. Then $\Phi \circ f$ is a $C^N$-map and, for every multi-order $\alpha \in \N_0^d$, $|\alpha|\leq N$, we have 
 \begin{equation}
 \partial_x^\alpha \left[ \Phi \circ f\right](x) = \Phi \left[ \partial_x^\alpha f(x)\right] \qquad \forall x\in U. 
 \label{eq:LCS.higher-partial-derivatives-Phi}
\end{equation}
\end{proposition}

By using Proposition~\ref{prop:LCS.CN-Phi}, the Taylor formula for $C^N$-functions and arguing as in the proof of Lemma~\ref{lem:LCS.C1-differentiable} we obtain the following version of Taylor's formula. 

\begin{proposition}\label{prop:LCS-Taylor}
 Let $f:U\rightarrow E$ be a $C^N$-map, $N\geq 1$. In addition, let $x\in U$ and $\delta>0$ be such that $B(x,\delta)\subset U$. Then, for $|h|<\delta$, we have 
 \begin{equation*}
 f(x+h)= \sum_{|\alpha|<N} \frac{h^\alpha}{\alpha!} \partial_x^\alpha f(x) + \sum_{|\alpha|=N} h^\alpha \int_0^1 (1-t)^{N-1} \partial_x^\alpha f(x+th)dt. 
\end{equation*}
\end{proposition}

\begin{definition}
 We say that a map $f:U\rightarrow E$ is $C^\infty$ (or smooth) when it is $C^N$ for all $N\geq 1$.
\end{definition}

As an immediate consequence of Proposition~\ref{prop:LCS.CN-Phi} we obtain the following result. 

\begin{proposition}\label{prop:LCS.smooth-Phi}
Suppose that $\Phi: E\rightarrow F$ is a continuous $\R$-linear map from $E$ to some locally convex space $F$.  Let $f:U\rightarrow E$ be a  smooth map.  Then $\Phi \circ f:U\rightarrow F$ is a  smooth map whose partial derivatives of any order are given by~(\ref{eq:LCS.higher-partial-derivatives-Phi}). 
\end{proposition}

We also mention the following version of Leibniz's rule. 

\begin{proposition}\label{prop:LCS.Leibniz-rule}
 Suppose that $\Phi:E_1\times E_2 \rightarrow E$ is a (jointly) continuous $\R$-bilinear map. Let $f_1:U\rightarrow E_1$ and $f_2:U\rightarrow E_2$ be $C^\infty$-maps. Then $x\rightarrow \Phi[f_1(x),f_2(x)]$ is a $C^\infty$-map from $U$ to $E$. Moreover, for every multi-order $\alpha$, we have 
 \begin{equation}
 \partial_x^\alpha \Phi\left[ f_1(x),f_2(x)\right] = \sum_{\beta+\gamma=\alpha} \binom{\alpha}{\beta} 
 \Phi\left[\partial_x^\beta f_1(x), \partial_x^\gamma f_2(x)\right] \qquad \forall x\in U. 
 \label{eq:LCS.Leibniz-rule}
\end{equation}
\end{proposition}
\begin{proof}
 Let $x\in U$. For $j=1,\ldots, d$ and $t\neq 0$ small enough, we have 
 \begin{multline}
 \frac{1}{t} \biggl( \Phi\left[ f_1(x+te_j), f_2(x+te_j)\right] -  \Phi\left[ f_1(x), f_2(x)\right] \biggr)\\
 = \Phi\left(\frac{1}{t}\left[f_1(x+te_j)-f_1(x)\right], f_2(x+te_j)\right) +  \Phi\left( f_1(x), \frac1{t}\left[f_2(x+te_j)-f_2(x)\right]\right) .
 \label{eq:LCS.bilinear-Phi-rate-of-change}  
\end{multline}
By assumption $\Phi:E_1\times E_2 \rightarrow E$ is a continuous bilinear map and $\frac{1}{t}[f_\ell(x+te_j)-f_\ell(x)]$, $\ell =1,2$, converges to $\partial_{x_j}f_\ell(x)$ in $E_\ell$. Therefore, the r.h.s.\ of~(\ref{eq:LCS.bilinear-Phi-rate-of-change}) converges to $\Phi[\partial_{x_j}f_1(x), f_2(x)] + \Phi[f_1(x), \partial_{x_j}f_2(x)]$ in $E$ as $t\rightarrow 0$. Thus, for $j=1, \ldots, d$, we have 
\begin{equation*}
 \partial_{x_j} \Phi\left[ f_1(x),f_2(x)\right] =  \Phi\left[\partial_{x_j}f_1(x), f_2(x)\right] + \Phi\left[f_1(x), \partial_{x_j}f_2(x)\right]. 
\end{equation*}
This proves~(\ref{eq:LCS.Leibniz-rule}) when $|\alpha|=1$. Together with the continuity of $\Phi$ and of $f_1$ and $f_2$ and their partial derivatives this further ensures us that $x\rightarrow \Phi[f_1(x),f_2(x)]$ is a $C^1$-map from $U$ to $E$. An induction then shows that, for every integer $N\geq 1$, the map $x\rightarrow \Phi[f_1(x),f_2(x)]$ is a $C^N$-map and all its partial derivatives of order~$\leq N$ satisfy~(\ref{eq:LCS.Leibniz-rule}). Incidentally, this is a $C^\infty$-map. The proof is complete.  
\end{proof}

\begin{remark}
 As it follows from the above proof, if we only assume the maps $f_1$ and $f_2$ to be $C^N$, $N\geq 1$, then the map $x\rightarrow \Phi[f_1(x),f_2(x)]$ is $C^N$ and all its partial derivatives of order~$\leq N$ satisfy the Leibniz rule~(\ref{eq:LCS.Leibniz-rule}). 
\end{remark}

Specializing Proposition~\ref{prop:LCS.Leibniz-rule} to the scalar multiplication $\C\times E \rightarrow E$ yields the following statement. 

\begin{corollary}
Let $u:U\rightarrow \C$ and $f:U\rightarrow E$ be $C^\infty$-maps. Then $x\rightarrow u(x)f(x)$ is a $C^\infty$-map from $U$ to $E$, and, for every multi-order $\alpha$, we have 
 \begin{equation*}
\partial_x^\alpha\left[ u(x)f(x)\right] = \sum_{\beta+\gamma=\alpha} \binom{\alpha}{\beta} 
 \partial_x^\beta u(x) \partial_x^\gamma f(x) \qquad \forall x\in U. 
\end{equation*}
\end{corollary}

In the case of locally convex algebras we obtain the following result. 

\begin{corollary}\label{cor:LCS.Leibniz-rule-LCA}
 Suppose that $\cA$ is a locally convex algebra. Let $f_1:U\rightarrow \cA$ and $f_2:U\rightarrow \cA$ be $C^\infty$-maps. Then the pointwise product $x\rightarrow f_1(x)f_2(x)$ is a smooth map from $U$ to $\cA$ and, for every multi-order $\alpha$, we have
 \begin{equation*}
\partial_x^\alpha\left[ f_1(x)f_2(x)\right] = \sum_{\beta+\gamma=\alpha} \binom{\alpha}{\beta} 
 \partial_x^\beta f_1(x) \partial_x^\gamma f_2(x) \qquad \forall x\in U. 
\end{equation*}
\end{corollary}

In what follows, we denote by $C^\infty(U;E)$ the space of smooth maps $f:U\rightarrow E$. We equip it with the locally convex topology generated by the semi-norms, 
\begin{equation}
 f\rightarrow \sup_{x\in K} p\left[\partial^\alpha_x f(x)\right],
 \label{eq:LCS.semi-norms-Cinfty}
\end{equation}
where $\alpha$ ranges over $\N_0^d$, $K$ ranges over compact sets of $U$, and $p$ ranges over continuous semi-norms on $E$. 

\begin{lemma}\label{lem:LCS.convergence-C1}
Suppose that $E$ is metrizable. Let $(f_\ell)_{\ell \geq 0}$ be a sequence of $C^1$-maps from $U$ to $E$ such that: 
 \begin{enumerate}
  \item[(i)] The maps $f_\ell(x)$ converge to $f(x)$ uniformly on compact sets of $U$. 
  
  \item[(ii)] For $j=1,\ldots, d$, the partial derivatives $\partial_{x_j} f_\ell(x)$ converge to $g_j(x)$ uniformly on compact sets of $U$.   
\end{enumerate}
Then the map $f:U\rightarrow E$ is $C^1$ and $\partial_{x_j} f=g_j$ for $j=1,\ldots, d$. 
\end{lemma}
\begin{proof}
 The above assumptions imply that $f, g_1,\ldots, g_d$ are continuous maps from $U$ to $E$. Let $x\in U$ and $\delta>0$ be such that $\overline{B}(x,\delta)\subset U$. We also let $(e_1, \ldots, e_d)$ be the canonical basis of $\R^d$. By Lemma~\ref{lem:LCS.C1-differentiable}, for $|t|\leq \delta$ and $j=1,\ldots, d$,  we have 
 \begin{equation*}
 f_\ell(x+te_j) -f_{\ell}(x) = t \int_0^{1} \partial_{x_j} f_\ell(x+ste_j)ds \qquad \text{for all $\ell\geq 0$}.  
\end{equation*}
 The assumption (i) ensures us that the l.h.s.~converges to $f(x+te_j)-f(x)$ in $E$ as $\ell \rightarrow \infty$. In addition, the assumption (ii) implies that the maps $[0,1]\ni s \rightarrow \partial_{x_j}f_\ell(x+ste_j)\in E$ converge uniformly to the map $s\rightarrow g_j(x+ste_j)$. Therefore, the continuity of the Riemann integral 
 (\emph{cf.}~Proposition~\ref{prop:LCS.Riemann integral}) ensures us that $\int_0^{1} \partial_{x_j} f_\ell(x+ste_j)dt$ converges to $\int_0^{1}g_j(x+ste_j)dt$ in $E$ as $\ell \rightarrow \infty$. It then follows that
 \begin{equation*}
f(x+te_j) -f(x) = t \int_0^{1}g_j(x+ste_j)dt. 
\end{equation*} 
 
 As $E$ is metrizable, the continuous map $g_j:U\rightarrow E$ is uniformly continuous on the compact set $\overline{B}(x,\delta)$. This implies that, as $t\rightarrow 0$, the continuous maps $[0,1]\ni s\rightarrow g_j(x+ste_j)\in E$ converge uniformly to the constant map $s\rightarrow g_j(x)$. Using the continuity of the Riemann integral once again we see that 
 \begin{equation*}
\lim_{t\rightarrow 0}\frac{1}{t} \left(f(x+te_j) -f(x)\right)= \lim_{t\rightarrow 0}  \int_0^{1}g_j(x+ste_j)ds =  \int_0^{1} g_j(x)ds = g_j(x). 
\end{equation*}
This is shows that, for all $x\in U$ and $j=1,\ldots, d$, the partial derivative $\partial_{x_j}f(x)$ exists and is equal to $g_j(x)$. As $g_1, \ldots, g_d$ are continuous maps we then conclude that $f$ is a $C^1$-map. The proof is complete. 
\end{proof}

\begin{proposition}
 Suppose that $E$ is a Fr\'echet space. Then $C^\infty(U;E)$ is a Fr\'echet space. 
\end{proposition}
\begin{proof}
As $E$ is a Fr\'echet space, $C^\infty(U;E)$ is Hausdorff and its topology is generated by a countable family of semi-norms. It remains to show that $C^\infty(U;E)$ is complete. In what follows we equip the space $C^0(U;E)$ with the topology of uniform convergence on compact sets of $U$, i.e., the locally convex topology generated by the semi-norms~(\ref{eq:LCS.semi-norms-Cinfty}) with $\alpha =0$. As $E$ is a complete metrizable space, $C^0(U;E)$  is a complete metrizable space as well. 

Let $(f_\ell)_{\ell \geq 0}$ be a Cauchy sequence in  $C^\infty(U;E)$. For every multi-order $\alpha$, the sequence $(\partial_x^\alpha f_\ell)_{\ell\geq 0}$ is a Cauchy sequence in $C^0(U;E)$, and so there is a map $g_\alpha \in C^0(U;E)$ such that $\partial_\alpha f_\ell $ converges to $g_\alpha$ in $C^0(U;E)$. Using Lemma~\ref{lem:LCS.convergence-C1} we see that $g_\alpha:U\rightarrow E$ is a $C^1$-map and $\partial_{x_j}\partial_x^\alpha f_\ell$ converges to $\partial_{x_j} g_{\alpha}$ in $C^0(U;E)$ for $j=1,\ldots, d$. Thus, if we set $f:=g_0$, then an induction shows that, for every $N\geq 0$, the map $f:U\rightarrow E$ is $C^N$ and $\partial^\alpha_x f=g_\alpha$ for $|\alpha|\leq N$. It then follows that $f:U\rightarrow E$ is a $C^\infty$-map and, for all multi-orders $\alpha$, the map $\partial_x^\alpha f_\ell$ converges to $g_{\alpha}=\partial_x^\alpha f$ in $C^0(U;E)$. That is, $f_\ell$ converges to $f$ in $C^\infty(U;E)$. This establishes the completeness of $C^\infty(U;E)$. The proof is complete. 
\end{proof}

\begin{remark}
 If we further assume that $E$ is a Fr\'echet-Montel space, then it can be shown that $C^\infty(U;E)$ is a Fr\'echet-Montel space as well. 
\end{remark}

\begin{remark}
 Given $N\in \N$, let us denote by $C^N(U;E)$ the space of $C^N$-maps $f:U \rightarrow E$. We equip it with the topology generated by the semi-norms~(\ref{eq:LCS.semi-norms-Cinfty}) with $|\alpha|\leq N$. Then the above proof shows that $C^N(U;E)$ is a Fr\'echet space when $E$ is a Fr\'echet space. In particular, this is a Banach space when $E$ is a Banach space. 
\end{remark}

\subsection{Differentiation under the integral sign} \label{subec:diff-int}
In what follows we assume that $E$ is a quasi-complete Suslin locally convex space. We also let $U$ be an open set of $\R^d$, $d\geq 1$, and $(X,\fS, \mu)$ a $\sigma$-finite measured space. 

\begin{lemma}\label{lem:LCS.continuity-integrals}
 Let $(x,y)\rightarrow F(x,y)$ be a map from $U\times X$ to $E$ such that:
 \begin{enumerate}
\item[(i)] $F(x,y)$ is integrable with respect to $y$ and is continuous with respect to $x$. 

\item[(ii)]  For every continuous semi-norm $p$ on $E$, there is a function $g_p(y)\in L^1_\mu(X)$ such that 
\begin{equation*}
 p\left[ F(x,y)\right] \leq g_p(y) \qquad \forall (x,y)\in U\times X. 
\end{equation*}
\end{enumerate}
Then $x\rightarrow \int_X F(x,y)d\mu(y)$ is a continuous map from $U$ to $E$. 
\end{lemma}
\begin{proof}
 Let $x\in U$ and $(x_\ell)_{\ell \geq 0}\subset U$ be a sequence converging to $x$. By (i) the maps $F(x_\ell,\cdot)$, $\ell\geq 0$, are integrable and converge pointwise to $F(x,\cdot)$. In addition, by (ii) for every continuous semi-norm $p$ on $E$, we have $p[F(x_\ell,y)]\leq g_p(y)$ for all $y\in X$ and $\ell\geq 0$. It then follows from Proposition~\ref{prop:LCS.DCT} that $\int_X F(x_\ell,y)d\mu(y) \rightarrow \int_X F(x,y)d\mu(y)$. This shows that the map $x\rightarrow \int_X F(x,y)d \mu(y)$ is continuous everywhere on $U$. The proof is complete. 
 \end{proof}

\begin{lemma}\label{lem:LCS.C1-integrals}
 Let $(x,y)\rightarrow F(x,y)$ be a map from $U\times X$ to $E$ such that:
 \begin{enumerate}
\item[(i)] $F(x,y)$ is integrable with respect to $y$ and is $C^1$ with respect to $x$. 

\item[(ii)]  For every continuous semi-norm $p$ on $E$, there is a function $g_{p}(y)\in L^1_\mu(X)$ such that, for $j=1, \ldots, d$, we have
\begin{equation*}
p\left[ \partial_{x_j}F(x,y)\right] \leq g_{p}(y) \qquad \forall (x,y)\in U\times X. 
\end{equation*}
\end{enumerate}
Then $x\rightarrow \int_X F(x,y)d\mu(y)$ is a $C^1$-map from $U$ to $E$, and, for $j=1, \ldots, d$, we have
\begin{equation*}
 \partial_{x_j} \int_X F(x,y)d\mu(y) = \int_X  \partial_{x_j} F(x,y)d\mu(y) \qquad \forall x\in U. 
\end{equation*}
\end{lemma}
\begin{proof}
 Let $x\in U$ and $\delta>0$ be such that $B(x,\delta)\subset U$. We denote by $(e_1, \ldots, e_d)$ the canonical basis of $\R^d$. Given $j\in \{1, \ldots, d\}$,  for $0<|t|<\delta$ and $y\in X$, we set 
 \begin{equation*}
 G_j(t,y)= \frac{1}{t} \left( F(x+te_j,y)-F(x,y)\right).  
\end{equation*}
 As $F(x,y)$ is $C^1$ with respect to $x$, it follows from Lemma~\ref{lem:LCS.C1-differentiable} that we have
  \begin{equation*}
 G_j(t,y)= \int_0^1 \partial_{x_j} F(x+ste_j,y)ds. 
\end{equation*}
 Combining this with (ii) and~(\ref{eq:LCS.semi-norm-integral}) shows that, for every continuous semi-norm $p$ on $E$, we have 
 \begin{equation}
 p\left[G_j(t,y)\right] \leq  \int_0^1 p\left[\partial_{x_j} F(x+ste_j,y)\right]ds \leq g_p(y). 
 \label{eq:LCS.pGj-domination}
\end{equation}
 
Let $(t_\ell)_{\ell \geq 0}\subset (-\delta,0)\cup (0,\delta)$ be a sequence converging to $0$. It follows from (i) that the maps $G_j(t_\ell,\cdot)$ are integrable maps that converge pointwise to $\partial_{x_j}F(x,\cdot)$. Combining this with~(\ref{eq:LCS.pGj-domination}) allows us to apply Proposition~\ref{prop:LCS.DCT} to get
\begin{align*}
\lim_{\ell \rightarrow \infty} \frac{1}{t_\ell}\left ( \int_X F(x+t_\ell e_j,y) d\mu(y) - \int_X F(x,y) d\mu(y)\right) 
  & = \lim_{\ell \rightarrow \infty}  \int_X G_j(t_\ell,y)d\mu(y)\\ 
  & = \int_X \partial_{x_j} F(x,y)d\mu(y). 
\end{align*}
 This shows that $\partial_{x_j} \int_X F(x,y) d\mu(y)$ exists and is equal to $\int_X \partial_{x_j} F(x,y)d\mu(y)$. Furthermore, it follows from (i)--(ii) and Lemma~\ref{lem:LCS.continuity-integrals} that $x\rightarrow \int_X \partial_{x_j} F(x,y)d\mu(y)$ is a continuous map from $U$ to $E$. Therefore, we see that $x\rightarrow \int_X F(x,y)d\mu(y)$ is a $C^1$-map from $U$ to $E$. The proof is complete. 
\end{proof}

Lemma~\ref{lem:LCS.C1-integrals} and a straightforward induction lead us to the following result. 

\begin{proposition}\label{prop:LCS.smoothness-integrals}
Let $(x,y)\rightarrow F(x,y)$ be a map from $U\times X$ to $E$ such that:
 \begin{enumerate}
\item[(i)] $F(x,y)$ is integrable with respect to $y$ and is $C^\infty$ with respect to $x$. 

\item[(ii)]  For every continuous semi-norm $p$ on $E$ and multi-order $\alpha \in \N_0^d$, there is a function $g_{p, \alpha}(y)\in L^1_\mu(X)$ such that
\begin{equation*}
p \left[ \partial_{x}^\alpha F(x,y)\right] \leq g_{p, \alpha}(y) \qquad \forall (x,y)\in U\times X. 
\end{equation*}
\end{enumerate}
Then $x\rightarrow \int_X F(x,y)d\mu(y)$ is a $C^\infty$-map from $U$ to $E$, and, for every $\alpha \in \N_0^d$, we have
\begin{equation*}
 \partial_{x}^\alpha \int_X F(x,y)d\mu(y) = \int_X   \partial_{x}^\alpha F(x,y)d\mu(y) \qquad \forall x\in U. 
\end{equation*}
\end{proposition}

\subsection{Fourier transform and Schwartz's class}\label{subsec:Fourier-Schwartz}
Suppose that $E$ is a quasi-complete Suslin locally convex space. In what follows all notions of integrability are with respect to the Lebesgue measure on $\R^n$, $n\geq 1$. 

We define the Fourier transform of integrable $E$-valued maps as follows. 

\begin{definition}
 Let $f:\R^n \rightarrow E$ be an integrable map. Its \emph{Fourier transform} $\hat{f}:\R^n \rightarrow E$ is defined by
 \begin{equation*}
 \hat{f}(\xi) = \int e^{-i x\cdot \xi} f(x)dx, \qquad \xi \in \R^n. 
\end{equation*}
\end{definition}

\begin{remark}
 The Fourier transform $\hat{f}(\xi)$ is well defined since, for every continuous semi-norm $p$ on $E$, we have 
\begin{equation}
 \int p\left[ e^{-i x\cdot \xi} f(x)\right] dx =   \int p\left[  f(x)\right] dx <\infty. 
 \label{eq:LCS.integrability-Fourier}
\end{equation}
\end{remark}

\begin{remark}
 Suppose that $\Phi: E\rightarrow F$ is a continuous $\C$-linear map from $E$ to some quasi-complete Suslin locally convex space $F$. Let $f:\R^n \rightarrow E$ be an integrable map. We know by Proposition~\ref{prop:LCS.Phi-integral} that the map $\Phi \circ f: \R^n \rightarrow F$ is integrable. Moreover, by using~(\ref{eq:LCS.Phi-Lebesgue-integral}) we get 
 \begin{equation}
 \Phi \left[ \hat{f}(\xi)\right] = \int e^{-ix \cdot \xi} \Phi\left[ f(x)\right] dx = \left[ \Phi \circ f\right]^\wedge\!\! (\xi). 
 \label{eq:LCS.Phi-Fourier-transform}
\end{equation}
\end{remark}

Let us denote by $C^0_b(\R^n; E)$ the space of bounded continuous maps $f:\R^n \rightarrow E$.  We equip it with the locally convex topology generated by the semi-norms, 
\begin{equation*}
 f \longrightarrow \sup_{x\in \R^n} p\left[ f(x)\right], 
\end{equation*}
where $p$ ranges over continuous semi-norms on $E$. 

\begin{proposition}
 The Fourier transform gives rise to a continuous linear map from $L^1(\R^n;E)$ to $C^0_b(\R^n;E)$. 
\end{proposition}
\begin{proof}
 Let $f:\R^n \rightarrow E$ be an integrable map. The map $\R^n \times \R^n \ni (\xi, x)\rightarrow e^{-ix\cdot \xi} f(x)\in E$ is continuous with respect to $\xi$ and integrable with respect to $x$. Moreover, given any continuous semi-norm $p$ on $E$, we have $p[e^{-ix\cdot \xi} f(x)]=p[f(x)]\in L^1(\R^n)$ for all $\xi \in \R^n$. 
 Therefore, it follows from Lemma~\ref{lem:LCS.continuity-integrals} that $\hat{f}:\xi \rightarrow \int e^{-i x\cdot \xi} f(x)dx$ is a continuous map from $\R^n$ to $E$. In addition, by using~(\ref{eq:LCS.semi-norm-integral}) and~(\ref{eq:LCS.integrability-Fourier}) we see that, for every continuous semi-norm $p$ on $E$, we have 
 \begin{equation}
\sup_{\xi \in \R^n}p\left[\hat{f}(\xi)\right] \leq \sup_{\xi \in \R^n}\int p\left[ e^{-ix\cdot \xi} f(x)\right]dx= \int p\left[f(x)\right]dx.
\label{eq:LCS.boundedness-Fourier}
\end{equation}
Therefore, we see that $\hat{f}\in C^0_b(\R^n;E)$ and depends continuously on $f \in L^1(\R^n; E)$. This proves the result.  
\end{proof}

\begin{definition}
 The Schwartz class $\cS(\R^n;E)$ consists of all smooth maps $f:\R^n \rightarrow E$ such that the maps $x^\alpha \partial_x^\beta f:\R^n\rightarrow E$ are bounded for all multi-orders $\alpha$ and $\beta$. 
\end{definition}

\begin{remark}
 The above definition of $\cS(\R^n;E)$ continues to makes sense when $E$ is any locally convex space. 
\end{remark}

We equip $\cS(\R^n;E)$  with the locally convex topology generated by the semi-norms, 
\begin{equation*}
 f \longrightarrow \sup_{x\in \R^n} p\left[ x^\alpha \partial_x^\beta f(x)\right], 
\end{equation*}
where $\alpha$ and $\beta$ range over $\N_0^n$ and $p$ ranges over all continuous semi-norms on $E$. 

\begin{proposition}\label{prop:LCS.Schwartz-Phi}
Suppose that $\Phi: E\rightarrow F$ is a continuous $\R$-linear map from $E$ to some locally convex space $F$. Then, for every $f\in \cS(\R^n; E)$, the composition $\Phi \circ f$ is in $\cS(\R^n; F)$. 
\end{proposition}
\begin{proof}
Let $f\in \cS(\R^n; E)$. It follows from Proposition~\ref{prop:LCS.smooth-Phi} that $\Phi \circ f$ is a smooth map from $\R^n$ to $F$ such that $\partial_x^\alpha [\Phi \circ f]= \Phi \circ \partial_x^\alpha f$ for all $\alpha \in \N_0^n$. Given $\alpha, \beta \in \N_0^n$, set $B_{\alpha \beta}=\{x^\alpha \partial_x^\beta f(x); \ x\in \R^n\}$. Then 
$\{x^\alpha \partial_x^\beta[\Phi\circ f](x); \ x\in \R^n\}=\Phi(B_{\alpha\beta})$. As $\Phi$ is a continuous map and $B_{\alpha\beta}$ is a bounded set of $E$, we deduce that 
the range of the map $x^\alpha \partial_x^\beta[\Phi\circ f]$ is bounded in $F$. It then follows that $\Phi \circ f$ is in $\cS(\R^n; F)$. The proof is complete. 
\end{proof}

\begin{remark}\label{rmk:LCS.xD-action-cS}
It is immediate from the definition of the topology of $\cS(\R^n; E)$ that the partial derivatives $\partial_{x_j}$, $j=1,\ldots, n$, induce continuous linear endomorphism on $\cS(\R^n; E)$. By using the Leibniz rule~(\ref{eq:LCS.Leibniz-rule}) it can be shown that the  multiplication operators by the coordinates $x_j$, $j=1,\ldots, n$, also induce continuous linear endomorphism on $\cS(\R^n; E)$. It then follows that, for all multi-orders $\alpha$ and $\beta$, the differential operator $x^\alpha \partial_x^\beta$ gives rise to a continuous linear endomorphism on $\cS(\R^n; E)$. 
\end{remark}

\begin{remark}
When $E$ is a Fr\'echet space (resp., Fr\'echet-Montel space) it can be shown that $\cS(\R^n; E)$ is a Fr\'echet  space (resp.,  Fr\'echet-Montel space). 
\end{remark}

\begin{proposition}
 The following holds. 
 \begin{enumerate}
 \item The Fourier transform induces a continuous endomorphism on $\cS(\R^n; E)$. 
 
 \item Let $f\in \cS(\R^n; E)$. Then, for every multi-order $\alpha$, we have 
           \begin{equation}
                 D_\xi^\alpha f =(-1)^{|\alpha|} \left[ x^\alpha f\right]^\wedge \qquad \text{and} \qquad \xi^\alpha \hat{f} = \left[ D_x^\alpha f\right]^\wedge.
                 \label{eq:LCS.derivative-Fourier-transform}
          \end{equation}
\end{enumerate}
\end{proposition}
\begin{proof}
 Given any continuous semi-norm $p$ on $E$, we denote by $q_p$ the semi-norm on $\cS(\R^n;E)$ defined by 
 \begin{equation*}
 q_{p}(f) = \sup_{x\in \R^n} ( 1+|x|)^{n+1} p\left[ f(x)\right], \qquad f\in \cS(\R^n;E). 
\end{equation*}
This is a continuous semi-norm on $\cS(\R^n;E)$. Furthermore, by using~(\ref{eq:LCS.boundedness-Fourier}) we see that, for every $f\in \cS(\R^n;E)$,  we have
\begin{equation}
 \sup_{\xi \in \R^n} p\left[ \hat{f}(\xi)\right] \leq \int p\left[f(x)\right] dx \leq \int q_{p}(f)(1+|x|)^{-(n+1)}dx =C q_p(f), 
 \label{eq:LCS.estimate-sup-qp}
\end{equation}
where we have set $C=\int (1+|x|)^{-(n+1)}dx$. 

Let $f\in \cS(\R^n;E)$. The map $\R^n \times \R^n \ni (\xi,x)\rightarrow e^{-ix\cdot\xi} f(x)\in E$ is smooth with respect to $\xi$ and is integrable with respect to $x$. Moreover, given any multi-order $\alpha$ and any continuous semi-norm $p$ on $E$, we have 
\begin{equation*}
 p\left[ \partial_\xi^\alpha ( e^{-ix\cdot\xi} f(x))\right] = p\left[ x^\alpha f(x)\right] \leq q_p(x^\alpha f) \left(1+|x|\right)^{-(n+1)}. 
\end{equation*}
Therefore, it follows from Proposition~\ref{prop:LCS.smoothness-integrals} that the Fourier transform $\hat{f}: \xi \rightarrow \int e^{-i x\cdot \xi} f(x)dx$ is a  smooth map from $\R^n$ to $E$. Furthermore, for every multi-order $\alpha$, we have 
\begin{equation}
 D_\xi^\alpha \hat{f}(\xi) = \int D_\xi^\alpha \left[ e^{-ix\cdot\xi} f(x)\right] dx = (-1)^{|\alpha|} \int x^\alpha e^{-ix\cdot\xi} f(x)dx =  (-1)^{|\alpha|}\left[ x^\alpha f\right]^\wedge\!\!(\xi). 
 \label{eq:LCS.derivative-Fourier-transform2}
\end{equation}

Let $\alpha \in \N_0^n$ and $\varphi \in E'$. By Proposition~\ref{prop:LCS.smooth-Phi} and Proposition~\ref{prop:LCS.Schwartz-Phi} the function $\varphi\circ f$ is Schwartz-class and $D_x^\alpha [\varphi\circ f]=\varphi \circ D_x^\alpha f$. Therefore, by using~(\ref{eq:LCS.Phi-Fourier-transform}) we see that, for all $\xi\in \R^n$, we have
\begin{equation*}
 \varphi\left[ \left(D_x^\alpha f\right)^\wedge \!\! (\xi)\right] = \left[ \varphi \circ D^\alpha_x f\right]^\wedge \!\! (\xi)  = 
 \left[ D^\alpha_x\left(\varphi \circ  f\right)\right]^\wedge \!\! (\xi) = \xi^\alpha \left(\varphi \circ  f\right)^\wedge \!\! (\xi)= \varphi\left[ \xi^\alpha \hat{f}(\xi)\right]. 
\end{equation*}
As $E'$ separates the points of $E$, we then deduce that $\left[ D_x^\alpha f\right]^\wedge =\xi^\alpha \hat{f}$. Together with~(\ref{eq:LCS.derivative-Fourier-transform2}) this proves~(\ref{eq:LCS.derivative-Fourier-transform}). 

Combining~(\ref{eq:LCS.derivative-Fourier-transform}) and~(\ref{eq:LCS.estimate-sup-qp}) shows that, given any continuous semi-norm $p$ on $E$ and  multi-orders $\alpha$ and $\beta$, we have 
\begin{equation*}
\sup_{\xi\in \R^n} p\left[ \xi^\alpha D_\xi^\beta \hat{f}(\xi)\right] =  \sup_{\xi\in \R^n} p\left[ \left[D_x^\alpha (x^\beta f)\right]^\wedge \!\!(\xi)\right] \leq C q_p\left[ D_x^\alpha (x^\beta f)\right].  
\end{equation*}
This shows that $\hat{f}:\R^n \rightarrow E$ is a Schwartz-class map. Note that $f \rightarrow q_p\left[ D_x^\alpha (x^\beta f)\right]$ is a continuous semi-norm on $\cS(\R^n; E)$, since $q_p$ is a continuous semi-norm on $\cS(\R^n; E)$ and we know by Remark~\ref{rmk:LCS.xD-action-cS} that $ D_x^\alpha x^\beta:\cS(\R^n;E)\rightarrow \cS(\R^n;E)$ is a continuous linear endomorphism. It then follows that the Fourier transform induces a continuous linear endomorphism on $\cS(\R^n;E)$. The proof is complete. 
\end{proof}

If $f:\R^n \rightarrow E$ is an integrable map, then its \emph{inverse Fourier transform} $\check{f}:\R^n \rightarrow E$ is given by
\begin{equation*}
 \check{f}(x) = \int e^{ix\cdot \xi} f(\xi)\dbar \xi, \qquad x\in \R^n,
\end{equation*}
where we have set $\dbar \xi =(2\pi)^{-n}d\xi$. The inverse Fourier transform satisfies the same kind of properties as that of the Fourier transform. In particular, it induces a continuous linear endomorphism on $\cS(\R^n;E)$.

\begin{proposition}\label{prop:LCS.Fourier-inversion}
 For all $f \in \cS(\R^n;E)$, we have 
 \begin{equation*}
 \left[ \hat{f}\right]^\vee =f \qquad \text{and} \qquad \left[ \check{f}\right]^\wedge =f. 
\end{equation*}
In particular, the Fourier transform and the inverse Fourier transform induce on $\cS(\R^n;E)$ continuous linear isomorphisms that are inverses of each other. 
\end{proposition}
\begin{proof}
 Let $f\in \cS(\R^n;E)$. Then $f$ and $\hat{f}$ are both integrable maps from $\R^n$ to $E$. In addition, let $\varphi \in E'$. We know by 
 Proposition~\ref{prop:LCS.Schwartz-Phi} that $\varphi \circ f$ is a Schwartz-class function. Therefore, by using~(\ref{eq:LCS.Phi-Fourier-transform}) and the usual Fourier inversion formula for Schwartz-class functions, we see that, for all $x\in \R^n$, we have
 \begin{equation*}
 \varphi\left[ \left(\hat{f}\right)^\vee\!\!(x)\right] =  \left[\varphi\circ \hat{ f}\right]^\vee \!\!(x) = \left[ \left(\varphi\circ f\right)^\wedge\right]^\vee \!\!(x) =\varphi\left[ f(x)\right].
\end{equation*}
 As $E'$ separates the points of $E$, we deduce that $\left(\hat{f}\right)^\vee=f$. The equality $\left(\check{f}\right)^\wedge=f$ is proved similarly. The proof is complete. 
\end{proof}

\begin{remark}
 All the above results on the Fourier transform hold \emph{verbatim} for Schwartz-class maps with values in locally convex spaces that are not quasi-complete Suslin spaces, providing we define the Fourier and inverse Fourier transforms as improper Riemann integrals. 
\end{remark}

\end{document}